\documentclass[10pt, reqno, a4paper]{amsart}
\usepackage{amssymb,latexsym,amsmath,amsfonts,bbm, amsthm}
\usepackage{paralist}
\usepackage{graphics} 
\usepackage{epsfig} 
\usepackage{graphicx}  \usepackage{epstopdf}
\usepackage[linkcolor=red, citecolor=magenta]{hyperref}
\hypersetup{colorlinks=true, urlcolor=blue}
\usepackage[nameinlink]{cleveref}
\numberwithin{equation}{section}
\renewcommand{\mod}[1]{\left\lvert#1\right\rvert}
\usepackage{comment}
\numberwithin{equation}{section}
\def\currentvolume{X}
\def\currentissue{X}
\def\currentyear{200X}
\def\currentmonth{XX}
\def\ppages{X--XX}

\newtheorem{theorem}{Theorem}[section]

\newtheorem{cor}[theorem]{Corollary}
\newtheorem{defn}[theorem]{Definition}

\newtheorem{lem}[theorem]{Lemma}
\newtheorem{prop}[theorem]{Proposition}
\newtheorem{rem}[theorem]{Remark}
\newcommand{\pa} {\partial}

\newcommand{\N}{\mathbb{N}}

\renewcommand{\leq}{\leqslant}
\newcommand{\norm}[1]{\left\Vert#1\right\Vert}

\def\hmath$#1${\texorpdfstring{{\rmfamily\textit{#1}}}{#1}}
\newcommand{\rd}{{\rm d}}
\newcommand{\mc}{\mathcal}
\def\textmatrix#1&#2\\#3&#4\\{\bigl({#1 \atop #3}\ {#2 \atop #4}\bigr)}
\def\dispmatrix#1&#2\\#3&#4\\{\left({#1 \atop #3}\ {#2 \atop #4}\right)}
\newcommand{\ip}[2]{\left<{#1},{#2}\right>}
\newcommand{\rea}{\mathbb{R}}
\newcommand{\abs}[1]{\left|#1\right|}

\newcommand{\cplx}{\mathbb{C}}
\newcommand{\z}{\mathbb{Z}}
\newcommand{\tl}{\tilde}
\theoremstyle{remark}
\newsavebox{\savepar}		 

\usepackage{xcolor}

\address{Department of Mathematics \& Statistics, Indian Institute of Technology Kanpur, Kanpur, UP- 208016}
\email{sakilahamed777@gmail.com}
\address{Department of Mathematics, Indian Institute of Technology Bombay, Powai, Mumbai- 400076}
\email{subratamajumdar634@gmail.com}

\date{\today}
\begin{document}
	\title[Control of Compressible Navier-Stokes system with Maxwell's law]{Controllability and Stabilizability of the linearized compressible Navier-Stokes system with Maxwell's law}
	\author[Sakil Ahamed, and Subrata Majumdar]{Sakil Ahamed$^*$, and Subrata Majumdar$^{\dagger}$}
	\thanks{$^*$Department of Mathematics \& Statistics, Indian Institute of Technology Kanpur, Kanpur, UP- 208016, India; e-mail:  sakilahamed777@gmail.com}
	\thanks{$^\dagger$ Department of Mathematics, Indian Institute of Technology Bombay, Powai, Mumbai- 400076, India; e-mail: subratamajumdar634@gmail.com}
	\thanks{$^*$Corresponding author: Sakil Ahamed}
	

		\medskip

		\begin{abstract}
			In this paper, we study the control properties of the linearized compressible Navier-Stokes system with Maxwell's law around a constant steady state $(\rho_s, u_s, 0), \rho_s>0, u_s>0$ in the interval $(0, 2\pi)$ with periodic boundary data. We explore the exact controllability of the coupled system by means of a localized interior control acting in any of the equations when time is large enough. We also study the boundary exact controllability of the linearized system using a single control force when the time is sufficiently large. In both cases, we prove the exact controllability of the system in the space $L^2(0,2\pi)\times L^2(0, 2\pi)\times L^2(0, 2\pi)$. We establish the exact controllability results by proving an observability inequality with the help of an Ingham-type inequality. Moreover, we prove that the system is exactly controllable at any time if the control acts everywhere in the domain in any of the equations. Next, we prove the small time lack of controllability of the concerned system. 
			
			Further, using a Gramian-based approach demonstrated by Urquiza, we prove the exponential stabilizability of the corresponding closed-loop system with an arbitrary prescribed decay rate using boundary feedback control law. 
		\end{abstract}
		\maketitle
		\tableofcontents
		
		\section{Introduction and Main results}
		\subsection{Setting of the problem}
		The control and stabilization of fluid flow have been extensively studied due to their numerous practical applications such as weather prediction, water flow in a pipe, designing aircraft and cars. Many researchers have been interested in the subject of the controllability of fluid flows, more so for incompressible flow (see \cite{VB06}, \cite{Fur04}, \cite{Raymond}, \cite{Raymond1}, \cite{VK08}) than for compressible flow (see \cite{SC00}, \cite{SC02}). In our research, we focus on investigating the exact controllability and stabilizability of the linearized compressible Navier-Stokes system with Maxwell's law.
		
		Let us consider the one-dimensional compressible  Navier-Stokes system in the domain $(0, 2\pi)$:
		\begin{equation}\label{eq1}
			\left.
			\begin{aligned}
				&\partial_t\hat{\rho}+\partial_x(\hat{\rho}\hat{u})=0 \qquad\qquad\qquad\qquad\text{ in }(0, T)\times (0, 2\pi),\\&\partial_t\left( {\hat{\rho}\hat u}\right) +\partial_x\left( {\hat{\rho}\hat u^2}\right) +\partial_xp =\partial_x\hat{S}\qquad\text{in }(0, T)\times (0, 2\pi),
			\end{aligned}
			\right\}
		\end{equation} 
		where, $\hat{\rho}$, $\hat{u}$, $p$ and  $\hat{S}$ represent the density, velocity, pressure and stress tensor of the fluid, respectively. Equation \eqref{eq1}$_1$ is the consequence of conservation of mass and equation \eqref{eq1}$_2$ is the consequence of conservation of momentum. We assume that the pressure $p$ satisfies the following constitutive law:
		\begin{equation}\label{pr}
			p(\hat \rho)=a\hat\rho^{\gamma},\: \quad \gamma\geq 1,
		\end{equation}
		where $a$ is a positive constant. While, the stress tensor $\hat S$ is assumed to satisfy the Maxwell's law:
		\begin{equation}\label{max law}
			\kappa\partial_t\hat{S}+\hat{S}=\mu\partial_x\hat{u}, \text{ in }(0, T)\times (0, 2\pi).
		\end{equation} Here $\mu$ and $\kappa$ are positive constants, with $\mu$ representing the fluid viscosity and $\kappa$ denoting the relaxation time that characterizes the time delay in the response of the stress tensor to the velocity gradient. For more detailed information, refer to \cite{hu2019global} and references therein. Relation $\eqref{max law}$ is first proposed by Maxwell, in order to describe the relation of stress tensor and velocity gradient for a non-simple fluid.

		In this paper, we study the control aspects of the one-dimensional compressible Navier-Stokes system with Maxwell's law and with periodic boundary conditions
		linearized around a constant steady state $\left( \rho_s, u_s, 0\right) , \rho_s > 0, u_s > 0$ of \eqref{eq1}.
		More precisely we consider the following system:
		\begin{equation}\label{nmaxeq3}
			\left.
			\begin{aligned}
				&\partial_t\rho+u_s\partial_x\rho+\rho_s\partial_xu=\mathbbm{1}_{\mathcal{O}_{1}} f_1, \qquad&&\text{in }(0, T)\times (0, 2\pi),\\&\partial_tu+u_s\partial_xu+a\gamma{\rho_s}^{\gamma-2}\partial_x\rho-\frac{1}{\rho_s}\partial_xS=\mathbbm{1}_{\mathcal{O}_{2}}  f_2,\qquad&&\text{in }(0, T)\times (0, 2\pi),\\&\partial_tS+\frac{1}{\kappa}S-\frac{\mu}{\kappa}\partial_xu=\mathbbm{1}_{\mathcal{O}_{3}} f_3, \qquad&&\text{in }(0, T)\times (0, 2\pi),\\&\rho(t, 0)=\rho(t, 2\pi),\: u(t,0)=u(t,2\pi),\: S(t,0)=S(t,2\pi),\qquad&&  t\in (0, T),\\&\rho(0,x)=\rho_0(x),\quad u(0,x)=u_0(x),\quad S(0,x)=S_0(x),\qquad&&  x\in (0,2\pi),
			\end{aligned}
			\right\}
		\end{equation}
		where $\mathbbm{1_{\mathcal{O}_{j}}}$ is  the characteristic function of an open set ${\mathcal{O}}_j \subset (0,2\pi),$  $j=1,2,3$ and $f_i$, $i=1,2,3$ are the distributed controls.
		
		\begin{defn}
			The system \eqref{nmaxeq3} is exactly controllable in $(L^2(0, 2\pi))^3$ at time $T>0$, if for any initial condition $(\rho_0, u_0, S_0)^{\top} \in (L^2(0, 2\pi))^3$ and any other $(\rho_1, u_1, S_1)^{\top}\in (L^2(0, 2\pi))^3$, there exist controls $f_i\in L^2\left( 0,T;L^2(\mathcal{O}_i) \right), i=1,2,3,$ such that the corresponding solution $(\rho, u, S)^{\top}$ of \eqref{nmaxeq3} satisfies
			\begin{equation}
				(\rho, u, S)^{\top}(T, x)=(\rho_1, u_1, S_1)^{\top}(x), \text{ for all } x \in (0, 2\pi).
			\end{equation}
		\end{defn}
		Let us first mention controllability results for some fluid models related to our system. The existence of the solution of the compressible Navier-Stokes system with Maxwell's law, along with the blow-up results has been studied, for example, in \cite{HuRacke, hu2019global, WangHu} and references therein. If $\kappa = 0$, then the Maxwell's law \eqref{max law} turns into the Newtonian law $\hat{S}=\mu\partial_x\hat{u}$ and the equation \eqref{eq1} becomes Navier-Stokes system of a viscous, compressible, isothermal barotropic fluid (density is function of pressure only), in a bounded domain $(0,2\pi)$. The compressible Navier-Stokes system linearized around a constant trajectory $(\rho_s, u_s)$ for $\rho_s>0, u_s>0$, yields a coupled transport-parabolic system with constant coefficients.
		The controllability of such systems with constant coefficients in one dimension has been extensively studied in the literature. In \cite{CM15, chowdhury2014null}, the authors studied this system in the domain  $(0,2\pi)$ with periodic boundary conditions and localized interior control acting only in the parabolic equation. In \cite{CM15}, using the moment method, the authors proved the null controllability in $H^{s+1}_{\textnormal{m}}(0,2\pi) \times H^{s}(0,2\pi),$ $s > 6.5,$ at time $T > \frac{2\pi}{|u_{s}|},$ where $H^{s}_{\textnormal{m}}(0, 2\pi)$ is the space of periodic Sobolev space with mean zero. This result was improved in \cite{chowdhury2014null} by showing that the null controllability holds for any initial data in $H^{1}_{\textnormal{m}}(0,2\pi) \times L^{2}(0,2\pi).$ Moreover, the authors also proved that, the system in consideration is not null controllable in $H^{s}_{\textnormal{m}}(0,2\pi) \times L^{2}(0,2\pi),$ $0\leqslant s < 1,$ at any time $T > 0$ by $L^{2}$-control acting in the parabolic equation. Thus  $H^{1}_{\textnormal{m}}(0,2\pi) \times L^{2}(0,2\pi)$ is the largest space in which the system is  null controllable by a $L^{2}$-parabolic control.
		
		
		Recently, in \cite{Beauchard}, the above results have been extended to more general coupled transport-parabolic systems with constant coefficients. The authors considered coupling of several transport and parabolic equations in one-dimensional torus $\mathbb{T}$ and studied the null controllability with localized interior control in optimal time. Moreover, an algebraic necessary and sufficient condition on the coupling term was obtained when controls act only on the parabolic or transport components. For the extension of these results to the general coupling matrix, one can refer to the work \cite{PLissy}.
		
		The local null controllability of the nonlinear compressible Navier-Stokes system around a trajectory with non-zero velocity at large time $T>0$ has been obtained in \cite{ervedoza1, ervedoza2, ervedoza2018local, MR-17, MRR-17, MN-19}.

		In \cite{ahamed}, the authors considered the one-dimensional compressible Navier-Stokes equations with Maxwell's law linearized around a constant steady state $(\rho_s, 0,0), \rho_s>0$ with Dirichlet boundary conditions and with interior controls in the interval $(0, \pi)$. They have proved that the system is not null controllable at any time
		using localized controls in density and stress equations and even everywhere control in the velocity equation. Moreover, they have shown that the system is null controllable at any time $T>0$ in the space $(L^2(0, \pi))^3$, if the control acts everywhere in the density or stress equation. Approximate controllability at large time using localized controls has also been studied. 
		
		In the following subsection, we state our main results regarding the interior controllability of the system \eqref{nmaxeq3}
		\subsection{Interior controllability}\label{intc}
		At first, we consider the case when $f_2=0=f_3$ in \eqref{nmaxeq3}. Performing integration by parts and using the boundary conditions, from the system \eqref{nmaxeq3}, we deduce
		\begin{equation*}
			\frac{d}{dt}\left( \int_{0}^{2\pi}u(t,x)\:\rd x\right)=0,\quad \frac{d}{dt}\left( \int_{0}^{2\pi}S(t,x)e^{\frac{t}{\kappa}}\:\rd x\right)=0,
		\end{equation*}
		and therefore, 
		\begin{equation*}
			\int_{0}^{2\pi}u(T,x)\:\rd x=\int_{0}^{2\pi}u_0(x)\:\rd x,\quad \int_{0}^{2\pi}S(T,x)e^{\frac{T}{\kappa}}\:\rd x=\int_{0}^{2\pi}S_0(x)\:\rd x.
		\end{equation*}
		Thus if the system \eqref{nmaxeq3} with $f_2=0=f_3$ is exactly controllable at time $T>0$ then necessarily
		\begin{equation}\label{zm}
			\int_{0}^{2\pi}u_0(x)\:\rd x=0=\int_{0}^{2\pi}S_0(x)\:\rd x.
		\end{equation}
		
		{Since we're working with a system governed by periodic boundary conditions, it is convenient to consider the Sobolev spaces for the periodic functions. For $s\in \mathbb{N}\cup \{0\}$, we denote by $H^s_{\text{p}}\left(0,2\pi \right) $, the space of $2\pi$-periodic functions belonging to $H^s_{\text{loc}}\left(\mathbb{R}\right) $, and by $\dot H^s_{\text{p}}\left(0,2\pi \right) $, the subspace of the functions
			belonging to $H^s_{\text{p}}\left(0,2\pi \right) $, with mean value zero. Consequently, the space $\dot L^2\left(0,2\pi \right) $, is given by
			\begin{equation*}
				\dot L^2\left(0,2\pi \right)=\left\lbrace f\in L^2_{\text{loc}}(\mathbb{R}) : \: f(x+2\pi)=f(x), \text{ for a.e. }x\in \mathbb{R},\: \int_{0}^{2\pi}f(x)\, \rd x=0\right\rbrace. 
			\end{equation*}
			
		}
		
		\begin{theorem}\label{nmaxthm_pos} Let $f_2=0=f_3$ in \eqref{nmaxeq3} and $\mathcal{O}_1\subseteq (0,2\pi)$. Then there exists a $T_0>0$ such that the system \eqref{nmaxeq3} is exactly controllable in $L^{2}(0,2\pi) \times {\dot L^{2}(0,2\pi) \times \dot L^{2}(0,2\pi)}$ at time $T>T_0$, by an interior
			control $f_1\in L^2\left( 0,T;L^2(\mathcal{O}_1) \right) $ for the density.
		\end{theorem}
		\begin{rem}\label{nmaxremnull} In fact, if the control is used in one of the equations, then the following exact controllability results can be obtained:
			\begin{enumerate}
				\item Let $f_1=0=f_3$ in \eqref{nmaxeq3} and $\mathcal{O}_2\subseteq (0,2\pi)$. Then the system \eqref{nmaxeq3} is exactly controllable in ${\dot L^{2}(0,2\pi) \times L^{2}(0,2\pi) \times \dot L^{2}(0,2\pi)}$ at time $T>T_0$, by an interior
				control $f_2\in L^2\left( 0,T;L^2(\mathcal{O}_2) \right) $ for the velocity.
				
				\item Let $f_1=0=f_2$ in \eqref{nmaxeq3} and $\mathcal{O}_3\subseteq (0,2\pi)$. Then the system \eqref{nmaxeq3} is exactly controllable in ${\dot L^{2}(0,2\pi) \times \dot L^{2}(0,2\pi)} \times L^{2}(0,2\pi)$ at time $T>T_0$, by an interior control $f_3\in L^2\left( 0,T;L^2(\mathcal{O}_3) \right) $ for the stress.
			\end{enumerate}
		\end{rem}
		\begin{rem}\label{optimal time}
			Note that the characteristics equations associated to our system \eqref{nmaxeq3} can be written as:
			\begin{equation*}
				x+\beta_1 t=c_1,\quad x+\beta_2 t=c_2,\quad x+\beta_3 t=c_3,
			\end{equation*}
			where $c_1,c_2,c_3$ are constants and $\beta_1, \beta_2, \beta_3$ are the roots (distinct and nonzero) of the equation
			\begin{equation}\label{nmaxpolybeta intro}
				r^3+2u_sr^2+\left( u_s^2-b\rho_s-\frac{\mu}{\kappa\rho_s}\right)r-\frac{\mu u_s}{\kappa\rho_s}=0.
			\end{equation} 
			In the above theorem, the waiting time $T_0$  is of the form $$T_0={2\pi}{\left(\frac{1}{|\beta_1|}+\frac{1}{|\beta_2|}+\frac{1}{|\beta_3|}\right)}.$$ In \Cref{nmaxthm_pos}, the exact controllability result is proved for any time $T>T_0$. We can not claim that $T_0$ is the minimal time to have the exact controllability of the system. Determine the minimal time $T_{min}>0$, such that the system is exactly controllable at $T\geq T_{min}$ and the system is not exactly controllable at $T<T_{min}$ is a challenging open problem. In particular, notice that the minimal control time should depend on the support of the localized control.
			
		\end{rem}
		{Next result regarding the controllability of the system when control acts everywhere in the domain.
			\begin{theorem}\label{thm_pos}
				Let $f_2=0=f_3$ in \eqref{nmaxeq3}. Then for any $T>0$ the system \eqref{nmaxeq3} is exactly controllable in $L^{2}(0,2\pi) \times {\dot L^{2}(0,2\pi) \times \dot L^{2}(0,2\pi)}$ at time $T>0$, by a control $f_1\in L^2\left( 0,T;L^2(0,2\pi) \right) $ acting everywhere in the density.
			\end{theorem}
			
			\begin{rem}
				Additionally, we achieve exact controllability of the system \eqref{nmaxeq3} at time $T>0$ by means of interior control acting either velocity or stress equation applied everywhere in the domain.
			\end{rem}

		}
		
		
		Next, let us study the controllability properties of the linearized compressible Navier–Stokes system with Maxwell's law when the control acts only in the boundary.

		\subsection{Boundary controllability}
		Let us consider the following system:
		\begin{equation}\label{nmaxeq3 bd}
			\left.
			\begin{aligned}
				&\partial_t\rho+u_s\partial_x\rho+\rho_s\partial_xu=0, \qquad&&\text{in }(0, T)\times (0, 2\pi),\\&\partial_tu+u_s\partial_xu+a\gamma{\rho_s}^{\gamma-2}\partial_x\rho-\frac{1}{\rho_s}\partial_xS=0,\qquad&&\text{in }(0, T)\times (0, 2\pi),\\&\partial_tS+\frac{1}{\kappa}S-\frac{\mu}{\kappa}\partial_xu=0, \qquad&&\text{in }(0, T)\times (0, 2\pi),\\&\rho(0,x)=\rho_0(x),\quad u(0,x)=u_0(x),\quad S(0,x)=S_0(x),\qquad&&  x\in (0,2\pi),
			\end{aligned}
			\right\}
		\end{equation}
		along with one of the following boundary conditions in the time interval $(0, T):$
		\begin{align}
			\label{density}&\textit{Control in density: } \rho(t, 0)=\rho(t, 2\pi)+q(t),\: u(t,0)=u(t,2\pi),\: S(t,0)=S(t,2\pi),\\ 
			\label{velocity}&\textit{Control in velocity: } \rho(t, 0)=\rho(t, 2\pi),\: u(t,0)=u(t,2\pi)+r(t),\: S(t,0)=S(t,2\pi),\\
			\label{stress}&\textit{Control in stress: } \rho(t, 0)=\rho(t, 2\pi),\: u(t,0)=u(t,2\pi),\: S(t,0)=S(t,2\pi)+p(t),
		\end{align}
		where $q, r, p$, are the controls. We recall the definition of boundary exact controllability of the system \eqref{nmaxeq3 bd}-\eqref{density}.	
		\begin{defn}
			The system \eqref{nmaxeq3 bd}-\eqref{density} is boundary exactly controllable in $(L^2(0, 2\pi))^3$ at time $T>0$, if for any initial condition $(\rho_0, u_0, S_0)^{\top} \in (L^2(0, 2\pi))^3$ and any other $(\rho_1, u_1, S_1)^{\top}\in (L^2(0, 2\pi))^3$, there exists control $q\in L^2\left( 0,T \right)$ such that the corresponding solution $(\rho, u, S)^{\top}$ of \eqref{nmaxeq3 bd}-\eqref{density} satisfies
			\begin{equation*}
				(\rho, u, S)^{\top}(T, x)=(\rho_1, u_1, S_1)^{\top}(x), \text{ for all } x \in (0, 2\pi).
			\end{equation*}
		\end{defn}
		We now mention some boundary controllability results of the linearized compressible Navier-Stokes system. In \cite{CM15}, the authors studied null controllability of the concerned system in $H^{s+1}_{\textnormal{m}}(0,2\pi) \times H^{s}(0,2\pi),$ $s > 4.5,$ using boundary control acting at the parabolic component at time $T > \frac{2\pi}{|u_{s}|}$ with periodic boundary data. Recently, in the work \cite{jiten2}, the above result has been improved. The author established the boundary null controllability of the system in the space $L^2(0, 2\pi)\times L^2(0, 2\pi)$ using density control and in the space $H^1_m(0, 2\pi)\times L^2(0, 2\pi)$ using velocity control in time $T > \frac{2\pi}{|u_{s}|}$. This paper extends these two results for the linearized compressible Navier-Stokes equation with non-barotropic fluids. In the context of Dirichlet boundary, the paper \cite{jiten} dealt with the boundary null controllability of this system with a density control in the space $H^s(0,1)\times L^2(0,1), s>\frac{1}{2}$ at time $T>1$. The authors have also explored the approximate controllability of the system in the space $L^2(0,1)\times L^2(0,1)$ at time $T>1$. Null controllability ``up to a finite-dimensional space" in the space $H^s(0,1)\times L^2(0,1), s>\frac{1}{2}$ has been shown by a velocity control in time $T>1$ with { periodic boundary condition for density and Dirichlet boundary conditions for velocity.}

		Next, we will state our boundary controllability results for the system \eqref{nmaxeq3 bd}. 
		Observe that, if the system \eqref{nmaxeq3 bd}-\eqref{density} is boundary exactly controllable then necessarily 
		\begin{align*}
			\int_{0}^{2\pi}\rho_0(x)\,\rd x+&u_s\int_{0}^{T}q(t)\rd t=0,\quad	\int_{0}^{2\pi}u_0(x)\,\rd x+b\int_{0}^{T}q(t)\, \rd t=0,\quad \int_{0}^{2\pi}S_0(x)\,\rd x=0,
		\end{align*}
		which imply
		\begin{equation}\label{condbdden}
			b\int_{0}^{2\pi}\rho_0(x)\, \rd x=u_s \int_{0}^{2\pi}u_0(x)\, \rd x, \quad \int_{0}^{2\pi}S_0(x)\,\rd x=0,
		\end{equation}
		{ where the positive constant $b$ is defined as
			\begin{equation*}\label{nmaxequ-constant}
				b=a\gamma{\rho_s}^{\gamma-2}.
			\end{equation*} 
			Similarly, if the system \eqref{nmaxeq3 bd} with the boundary  condition \eqref{velocity} is boundary exactly controllable then necessarily 
			\begin{equation}\label{condbdvel}
				u_s\int_{0}^{2\pi}\rho_0(x)\, \rd x=\rho_s \int_{0}^{2\pi}u_0(x)\, \rd x, \quad \int_{0}^{2\pi}S_0(x)\,\rd x=\frac{\mu}{\kappa}\int_{0}^{T}e^{\frac{t}{\kappa}}r(t)\, \rd t,
			\end{equation}
			and, if the system \eqref{nmaxeq3 bd} with the boundary  condition \eqref{stress} is boundary exactly controllable then necessarily 
			\begin{equation}\label{condbdstr}
				\int_{0}^{2\pi}\rho_0(x)\, \rd x= \int_{0}^{2\pi}S_0(x)\, \rd x=0, \quad \int_{0}^{2\pi}u_0(x)\,\rd x=\frac{1}{\rho_s}\int_{0}^{T}p(t)\, \rd t.
			\end{equation}
			Therefore, to investigate the exact boundary controllability of systems \eqref{nmaxeq3 bd}-\eqref{stress}, for the sake of simplicity,  we consider the following space which satisfy all the compatibility conditions \eqref{condbdden}-\eqref{condbdstr}: $$\dot L^2(0, 2\pi)\times \dot L^2(0, 2\pi)\times \dot L^2(0, 2\pi).$$
			\begin{theorem}\label{nmaxthm_pos bd} The system \eqref{nmaxeq3 bd}-\eqref{density} is exactly controllable in $\dot L^2(0, 2\pi)\times \dot L^2(0, 2\pi)\times \dot L^2(0, 2\pi)$ at time $T>T_0$ (same as in \Cref{optimal time}), by a boundary 
				control $q\in \dot L^2( 0,T) $ for the density.
			\end{theorem}
			\begin{rem}\label{nmaxremnul}If the control acts in one of the boundary, then we also get the exact controllability of the system \eqref{nmaxeq3 bd}:
				\begin{enumerate}
					\item  The system \eqref{nmaxeq3 bd} with the boundary condition \eqref{velocity} is exactly controllable in $\dot L^2(0, 2\pi)\times \dot L^2(0, 2\pi)\times \dot L^2(0, 2\pi)$ at time $T>T_0$, by a boundary
					control $r \in \dot L^2( 0,T) $ for the velocity.
					
					\item The system \eqref{nmaxeq3 bd} with the boundary condition \eqref{stress} is exactly controllable in $\dot L^2(0, 2\pi)\times \dot L^2(0, 2\pi)\times \dot L^2(0, 2\pi)$ at time $T>T_0$, by a boundary
					control $p \in \dot L^2( 0,T) $ for the stress.
				\end{enumerate}
			\end{rem}
		}

			%
		\subsection{Methodology of the proof of main results}
		The proof of the exact controllability result relies on an observability inequality and the spectral analysis of the linearized operator. The spectrum of the linear operator consists of three sequences of complex eigenvalues whose real parts converge to three distinct finite numbers, and the imaginary parts behave as $n$ for $|n|\rightarrow \infty$ (see, \Cref{nmaxsecspec}). Therefore the system behaves like a hyperbolic system. Hence the geometric control condition is needed for the controllability of the system using localized control. Moreover, the eigenfunctions of the linearized operator and its adjoint form Riesz bases on $L^2(0,2\pi)\times L^2(0,2\pi)\times L^2(0,2\pi)$. Then using the series representation of the solution of the adjoint problem and a hyperbolic type Ingham inequality, we can prove \Cref{nmaxthm_pos}. The proof of the Ingham-type inequality relies on the construction of a family biorthogonal to the family of exponentials $\{e^{-{\lambda_k^l}t}, k\in \z^*, l=1,2,3\}$, see \Cref{biorthogonal}. The technique we have adapted is inspired from \cite{LR14} and \cite{biccari2019null}. The entire study relies on the theory of logarithmic integral and complex analysis. One can refer to the book \cite{RY} for more details. Moreover, it is worthwhile to mention that a modification (for example \cite[Proposition 3.1]{chowdhury2014null} or \cite[Proposition 8.1]{ahamed}) to the general hyperbolic Ingham inequality (see \cite{ingham1936some}, \cite[Theorem 4.1]{micu2004introduction}) is not enough to prove the required Ingham-type inequality for our case.
		
		{
			The controllability result (\Cref{thm_pos}) is established using a direct approach: constructing the control explicitly to bring the system's solution at rest at any given finite time $T$. Since the eigenfunctions (or the generalized eigenfunctions) of the linearized operator form a Riesz Basis, the system \eqref{nmaxeq3} with control acting everywhere in the domain can be projected onto each finite dimensional eigenspaces for each $n\in \mathbb{Z}$. For any given time $T>0$, 
			we can achieve null controllability of this finite-dimensional system using the `Hautus lemma'. Furthermore, constructing the control using the finite-dimensional controllability operator can be achieved alongside a uniform estimate that is independent of $n$ (see \Cref{lem-cntrlfin-est}). Next, summing up these finite dimensional controls obtained for each projected system, we construct a control for the whole system and we show that the control brings the solution of the whole system to rest at time $T>0$. This approach closely resembles that utilized in \cite{chowdhury2012controllability}.
		}
		
		\medskip
		
		This paper also deals with the lack of controllability of the system \eqref{nmaxeq3}. For the influence of the transport component of the system, it is relevant to guess that the system \eqref{nmaxeq3} cannot be steered to zero (as exact and null controllability are equivalent for our system) at small time (see \Cref{lack theorem}). More precisely, in \Cref{lack sec}, we have established that, the system is not exactly controllable at small time by means of localized interior control {acting on every equations}. We have utilized the work \cite[Section 3]{Beauchard} to establish this result. {Furthermore, we also proved that the control system \eqref{nmaxeq3 bd}, \eqref{velocity} is not controllable in small time with the boundary control acting in density.  Similarly, the same result can be demonstrated for  other boundary cases (\eqref{density}, \eqref{stress}) also.}
		\subsection{Stabilizability}
		In this paper, we also study the boundary feedback stabilization of linearized compressible Navier-Stokes system with Maxwell's law \eqref{nmaxeq3 bd} with the control acting only in the any boundary of density \eqref{density}, velocity \eqref{velocity} and stress \eqref{stress}. More precisely, we construct a feedback control law which forces the solution of \eqref{nmaxeq3 bd}-\eqref{density} to decay exponentially towards the origin with any order. 

		In our case, we prove the exponential stabilization result by an explicit feedback law following a method demonstrated by J.M. Urquiza \cite{UR05} which relies on the Gramian approach and the Riccati equations, see \cite{AV}, \cite{KV1} for more details. This method addresses the exponential stabilization issue for an exactly controllable system with an operator that generates an infinitesimal group of continuous operator. This approach has been successfully utilized to study rapid exponential stabilization of KdV equation \cite{CE09}, Boussinesq system of KdV-KdV type in \cite{Filho2021RapidES}, linearized compressible Navier-Stokes equation in the case of creeping flow in \cite{Shirshendu}. Our work regarding exponential stabilization is inspired from \cite{CE09}, \cite{Filho2021RapidES} and \cite{Shirshendu}. Recently, classical moment method has been extensively utilized to establish the exact controllability and stabilizability of some dispersive system, see \cite{leal2021control}, \cite{leal2021simple}. In these works, utilizing the group structure of the corresponding system, the authors conclude the stabilizability from the controllability in some periodic Sobolev spaces.
		
		Let us state the main stabilization result for the closed-loop system \eqref{nmaxeq3 bd}-\eqref{density}.
		\begin{theorem}\label{nmaxthm_pos st 1}The control system \eqref{nmaxeq3 bd}-\eqref{density} is completely stabilizable in $\dot L^2(0, 2\pi)\times \dot L^2(0, 2\pi)\times \dot L^2(0, 2\pi)$ by boundary feedback control $q\in L^2\left( 0,\infty\right) $ for the density. {That is, for any $\nu>0$, there exists a continuous linear map  $\Pi: L^2(0, 2\pi)\times  L^2(0, 2\pi)\times  L^2(0, 2\pi)\to \cplx$ such that for any initial data $(\rho_0, u_0, S_0) \in \dot L^2(0, 2\pi)\times \dot L^2(0, 2\pi)\times \dot L^2(0, 2\pi)$, the solution $(\rho, u, S)$ of system \eqref{nmaxeq3 bd}-\eqref{density} with $q(t)= \Pi(\left(\rho(t,\cdot), u(t,\cdot), S(t,\cdot)\right))$ satisfies} the following
			{\small
				\begin{align*}
					\norm{\rho(t, \cdot )}_{L^2(0,2\pi)}+\norm{u(t, \cdot )}_{L^2(0,2\pi)}+\norm{S(t, \cdot )}_{L^2(0,2\pi)}& \leq C e^{-\nu t} \bigg[ \norm{\rho_0}_{L^2(0,2\pi)}+\norm{u_0}_{L^2(0,2\pi)}+\norm{S_0}_{L^2(0,2\pi)} \bigg], 
			\end{align*}}
			for all  $t>0$,	where $C=C(T)$ is a positive constant, independent of $\rho_0, u_0, S_0, q$ and $t$.
		\end{theorem}
		\subsection{Organization of the paper}
		The plan of the paper is as follows. In \Cref{nmaxwell-posedness}, we investigate the well-posedness of the system \eqref{nmaxeq3} using semigroup theory and determine the adjoint of the linear operator associated with the coupled system. Next, in \Cref{nmaxsecspec}, we analyze the behavior of the spectrum of the linearized operator associated with system \eqref{nmaxeq3}, demonstrating that the eigenfunctions of the linear operator form a Riesz basis in $(L^2(0, 2\pi))^3$. The exact controllability results, \Cref{nmaxthm_pos} and \Cref{thm_pos}, are proved in \Cref{nmaxsecnullcont}. The proof of the boundary exact controllability result, \Cref{nmaxthm_pos bd}, is given in \Cref{nmax-bd}. In \Cref{lack sec}, we discuss the lack of controllability of system \eqref{nmaxeq3} in small time. Subsequently, in \Cref{stab}, we examine the rapid exponential stabilization result (\Cref{nmaxthm_pos st 1}). In \Cref{biorthogonal}, we present the proof of an Ingham-type inequality adapted to our context. Finally, in \Cref{secmultiev}, we provide some details on managing multiple eigenvalues of the linear operator.

		\section{Linearized operator}\label{nmaxwell-posedness}
		
		{
			For any $s\in \mathbb{N}\cup \{0\},$ we consider the spaces,
			\begin{align*}
				H^s_{\text{p}}\left(0,2\pi \right)=&\left\lbrace f:\: f=\sum_{k\in \mathbb{Z}}c_ke^{2\pi ikx},\: \sum_{k\in \mathbb{Z}}\left| k\right|^{2s}\left| c_k\right|^2< \infty  \right\rbrace,  \\
				\dot H^s_{\text{p}}\left(0,2\pi \right)=&\left\lbrace f\in H^s_{\text{p}}\left(0,2\pi \right):\: \int_{0}^{2\pi}f(x)\, \rd x=0 \right\rbrace, 
			\end{align*}
			with the norms
			\begin{equation*}
				\left\| f\right\|_{H^s_{\text{p}}\left(0,2\pi \right)}=\left(\sum_{k\in \mathbb{Z}}\left(1+ \left| k\right|^{2s}\right)\left| c_k\right|^2  \right)^{\frac{1}{2}},\: 	\left\| f\right\|_{\dot H^s_{\text{p}}\left(0,2\pi \right)}=\left(\sum_{k\in \mathbb{Z}\setminus\{0\}} \left| k\right|^{2s}\left| c_k\right|^2  \right)^{\frac{1}{2}}.
			\end{equation*}

		}
		
		Let us define $\mathcal {{Z}}=L^2(0, 2\pi)\times L^2(0, 2\pi)\times L^2(0, 2\pi)$.
		Let $\mathcal {{Z}}$ be endowed with the inner product
		\begin{equation}\label{nmaxinnerproduct}
			\left\langle \begin{pmatrix}
				\rho\\ u\\S
			\end{pmatrix},\begin{pmatrix}
				\sigma\\ v\\\tilde{S}
			\end{pmatrix} \right\rangle_{\mathcal {{Z}}}=b\int_{0}^{2\pi}\rho\bar{\sigma} \,\rd x+\rho_s\int_{0}^{2\pi}u\bar{v}\, \rd x+\frac{\kappa}{\mu}\int_{0}^{2\pi}S\bar{\tilde{S}}\, \rd x. 
		\end{equation}
		We now define the unbounded operator $\left( \mathcal A, \mathcal D(\mathcal A;\mathcal {{Z}})\right) $ in $\mathcal {{Z}}$ by{\small$$\mathcal D(\mathcal A;\mathcal{{Z}})=\left\lbrace \begin{pmatrix}
				\rho\\u\\S
			\end{pmatrix}\in \mathcal {{Z}}\: :\:\left(\rho,u,S \right)^\top \in  {H_p^1(0,2\pi)\times  H_p^1(0,2\pi)\times  H_p^1(0,2\pi)}
			\right\rbrace
			$$}and
		\begin{equation}\label{op}
			\mathcal A=\begin{bmatrix}
				-u_s\frac{d}{dx}\hspace{1mm}&-\rho_s\frac{d}{dx}\hspace{1mm}&0\vspace{2mm}\\-b\frac{d}{dx}&-u_s\frac{d}{dx}&\frac{1}{\rho_s}\frac{d}{dx}\vspace{2mm}\\0&\frac{\mu}{\kappa}\frac{d}{dx}&-\frac{1}{\kappa}
			\end{bmatrix}.
		\end{equation}
		The control operator $\mathcal{B} \in \mathcal{L}(\mathcal {{Z}};\mathcal {{Z}})$ is defined by 
		\begin{equation}\label{nmaxeqcontr}
			\mathcal{B} f  = \left( \mathbbm{1}_{\mathcal{O}_{1}} f_1, \mathbbm{1}_{\mathcal{O}_{2}} f_2, \mathbbm{1}_{\mathcal{O}_{3}} f_3\right)^{\top}, \qquad f=(f_1, f_2, f_3)^{\top} \in \mathcal{{ Z}}.
		\end{equation}
		With the above introduced notations, the system  \eqref{nmaxeq3} can be rewritten as 
		\begin{equation} \label{nmaxop-eqn}
			\dot{z}(t) = \mathcal{A} z(t) + \mathcal{B} f(t), \quad t\in (0,T), \qquad z(0) = z_{0},
		\end{equation}
		where, we have set $z(t) = (\rho(t,\cdot), u(t, \cdot), S(t,\cdot))^{\top},$ $z_{0} =(\rho_{0}, u_{0}, S_0)^{\top}, \text{ and } f(t)=(f_1(t,\cdot), f_2(t,\cdot), f_3(t,\cdot))^{\top} .$
		
		Next, we will prove the well-posedness of the system \eqref{nmaxop-eqn}.
		
		\begin{prop} \label{nmaxpr:semigroup-z} 
			The operator $(\mathcal{A}, \mathcal{D}(\mathcal{A}; \mathcal {{Z}}))$ is the infinitesimal generator of a strongly continuous semigroup $\left\lbrace \mathbb{T}_{t} \right\rbrace_{t\geq 0} $ on $\mathcal {{Z}}.$
			Further, for any $f\in L^2(0,T; \mathcal{{ Z}})$ and for any $z_0\in \mathcal {{Z}}$, \eqref{nmaxop-eqn} admits a unique solution $(\rho, u, S)\in C([0,T]; \mathcal {{Z}})$ with 
			\begin{equation*}
				\|(\rho, u, S)\|_{C([0,T];\mathcal {{Z}})} \leqslant  C \Big(\|z_0\|_{\mathcal {{Z}}}+ \|f\|_{L^2(0,T;\mathcal{{ Z}})}\Big).
			\end{equation*}
		\end{prop}
		\begin{proof}
			We rewrite $\mathcal{A}:= \mathcal{A}_{1} + \mathcal{A}_{2},$ with 
			\begin{equation*}
				\mathcal{A}_{1} = \begin{bmatrix}
					-u_s\frac{d}{dx}\hspace{1mm}&-\rho_s\frac{d}{dx}\hspace{1mm}&0\vspace{2mm}\\-b\frac{d}{dx}&-u_s\frac{d}{dx}&\frac{1}{\rho_s}\frac{d}{dx}\vspace{2mm}\\0&\frac{\mu}{\kappa}\frac{d}{dx}&0
				\end{bmatrix}, \quad 
				\mathcal{A}_{2} = \begin{bmatrix}
					0&0&0\\0&0&0\\0&0&-\frac{1}{\kappa}
				\end{bmatrix}.
			\end{equation*}
			Note that 
			$$ \norm{\mathcal{A}_{2}\begin{pmatrix} \rho \\ u \\ S \end{pmatrix}}_{\mathcal {{Z}}}\leqslant \frac{1}{{\kappa}} \norm{\begin{pmatrix} \rho \\ u \\ S  \end{pmatrix} }_{\mathcal {{Z}}} \quad \mbox{ for all }
			\begin{pmatrix} \rho \\ u \\ S \end{pmatrix}  \in \mathcal {{Z}}.$$
			Thus $\mathcal{A}_{2}$ is a bounded perturbation of the operator $\mathcal{A}_{1}$ on $\mathcal {{Z}}$.
			Therefore, it is sufficient to show that $\mathcal{A}_{1}$ generates a $C^{0}$-semigroup on $\mathcal {{Z}},$ (for details, see \cite[Theorem 2.11.2]{TW09}).
			
			The adjoint $\mathcal{A}^*_1$ of $\mathcal{A}_1$ with the domain 
			$\mathcal D(\mathcal A_1^*;\mathcal {{Z}})=\mathcal D(\mathcal A;\mathcal {{Z}})$ is given by
			\begin{equation*}
				\mathcal A^*_1=\begin{bmatrix}
					u_s\frac{d}{dx}\hspace{1mm}&\rho_s\frac{d}{dx}\hspace{1mm}&0\vspace{2mm}\\b\frac{d}{dx}&u_s\frac{d}{dx}&-\frac{1}{\rho_s}\frac{d}{dx}\vspace{2mm}\\0&-\frac{\mu}{\kappa}\frac{d}{dx}&0
				\end{bmatrix}=-\mathcal A_1.
			\end{equation*}
			Thus $\mathcal{A}_1$ is a skew-adjoint operator. Then using Stone Theorem (\cite[Theorem 3.8.6]{TW09}), we get that $\mathcal{A}_1$  generates a unitary group on $\mathcal {{Z}}$.
		\end{proof}
		Thus, we have the following well-posedness result for the system \eqref{nmaxeq3}.
		\begin{theorem}
			For any $(\rho_0, u_0, S_0)^{\top}\in \mathcal{{ Z}}$ and $f_i\in L^2(0, T; L^2(0, 2\pi)), \, i=1,2,3,$ the system \eqref{nmaxeq3} admits a unique solution $(\rho, u, S)^{\top}\in C\left( [0, T];\mathcal{{ Z}}\right)$ with
			\begin{equation*}
				\|(\rho, u, S)\|_{C([0,T];\mathcal {{Z}})} \leqslant  C \Big(\|(\rho_0, u_0, S_0)\|_{\mathcal {{Z}}}+ \|(f_1, f_2, f_3)\|_{L^2(0,T;\mathcal{{ Z}})}\Big).
			\end{equation*}
		\end{theorem}
		Note that, the exact controllability of the pair $(\mc A, \mc B)$ is equivalent to the observability of the pair $(\mc A^*, \mc B^*),$ where $\mc A^*$ and $\mc B^*$ are the adjoint operators of $\mc A \text{ and } \mc B$ respectively (see \cite[Chapter 2.3]{Co07} for details). Thus it is important to derive the adjoint of the operator $\mc A:$ 
		\begin{prop} \label{nmaxadj-op}
			The adjoint of  $(\mathcal{A}, \mathcal{D}(\mathcal{A} ; \mathcal {{Z}}))$ in $\mathcal {{Z}}$ is defined by 
			\begin{equation} \label{nmaxdom-A*}
				\mathcal{D}(\mathcal{A}^*;\mathcal {{Z}}) =\mathcal{D}(\mathcal{A};\mathcal {{Z}}),
			\end{equation}
			and 
			\begin{equation} \label{nmaxop-A*}
				\mathcal{A}^* = \begin{bmatrix}
					u_s\frac{d}{dx}\hspace{1mm}&\rho_s\frac{d}{dx}\hspace{1mm}&0\vspace{2mm}\\b\frac{d}{dx}&u_s\frac{d}{dx}&-\frac{1}{\rho_s}\frac{d}{dx}\vspace{2mm}\\0&-\frac{\mu}{\kappa}\frac{d}{dx}&-\frac{1}{\kappa}
				\end{bmatrix}.
			\end{equation}
			Moreover, $(\mathcal{A}^*, \mathcal{D}(\mathcal{A}^*; \mathcal {{Z}}))$ is the infinitesimal generator of a strongly continuous semigroup $\left\lbrace \mathbb{T}^*_{t} \right\rbrace_{t\geq 0} $ on $\mathcal {{Z}}.$
		\end{prop}
		As mentioned in \Cref{intc}, the controllability results for \eqref{nmaxeq3} is expected in the subspaces of $\mathcal{Z}$ satisfying \eqref{zm}. Thus, we need to have that the restriction of the operator $\mathcal{A}$ on those subspaces of  $\mathcal{Z}$ also generates a semigroup. 
		From \cite[Proposition 2.4.4, Section 2.4, Chapter 2]{TW09}, the following result holds. 
		\begin{prop}
			Let us define the spaces
			\begin{align}
				{\mathcal Z_m}=& L^2(0, 2\pi)\times {\dot L^2(0, 2\pi)\times \dot L^2(0, 2\pi)},\\
				{\mathcal Z_{m,m}}=& {\dot L^2(0, 2\pi)\times \dot L^2(0, 2\pi)\times \dot L^2(0, 2\pi)}\label{defn_Z_mm}.
			\end{align}
			where the spaces ${\mathcal Z_m}, {\mathcal Z_{m,m}}$ are equipped with the inner product \eqref{nmaxinnerproduct}.
			
			Then $\mathcal Z_m$ is invariant under the semigroups $\mathbb{T}$ and $\mathbb{T}^*$. The restriction of $\mathbb{T}$ to $\mathcal Z_m$ is a strongly continuous
			semigroup in $\mathcal Z_m$ generated by $(\mathcal{A}, \mathcal{D}(\mathcal{A};\mathcal{Z}_m))$, where $\mathcal{D}(\mathcal{A};\mathcal{Z}_m)=\mathcal{D}(\mathcal{A}, \mathcal{Z})\cap \mathcal{Z}_m$. Also, the restriction of $\mathbb{T}^*$ to $\mathcal Z_m$ is a strongly continuous semigroup in $\mathcal Z_m$  generated by $(\mathcal{A}^*, \mathcal{D}(\mathcal{A}^*;\mathcal{Z}_{m}))$, where $\mathcal{D}(\mathcal{A}^*;\mathcal{Z}_{m})=\mathcal{D}(\mathcal{A}^*, \mathcal{Z})\cap \mathcal{Z}_m$. 
			
			Similarly, $\mathcal{Z}_{m,m}$ is invariant under the semigroups $\mathbb{T}$ and $\mathbb{T}^*$. 
			The restriction of $\mathbb{T}$ to $\mathcal Z_{m,m}$ is a strongly continuous
			semigroup in $\mathcal Z_{m,m}$ generated by $(\mathcal{A}, \mathcal{D}(\mathcal{A};\mathcal{Z}_{m,m}))$, where $\mathcal{D}(\mathcal{A};\mathcal{Z}_{m,m})=\mathcal{D}(\mathcal{A}, \mathcal{Z})\cap \mathcal{Z}_{m,m}$. Also, the restriction of $\mathbb{T}^*$ to $\mathcal Z_{m,m}$ is a strongly continuous semigroup in $\mathcal Z_{m,m}$  generated by $(\mathcal{A}^*, \mathcal{D}(\mathcal{A}^*;\mathcal{Z}_{m,m}))$, where $\mathcal{D}(\mathcal{A}^*;\mathcal{Z}_{m,m})=\mathcal{D}(\mathcal{A}^*, \mathcal{Z})\cap \mathcal{Z}_{m,m}$. 
		\end{prop}
		\subsection{Spectral analysis of the linearized operator}\label{nmaxsecspec}
		We define a Fourier basis $\left\lbrace \phi_{0,1},\:\phi_{n,l},\: l =1,2,3\right\rbrace_{n\in\mathbb{Z}^*}$ $\left(\: \mathbb{Z}^* = \mathbb{Z}\setminus \{0\}\right) $ in $\mathcal {{Z}}_m$ as follows :
		\begin{align}\label{nmaxfbasis}
			&\phi_{0,1}(x)=\frac{1}{\sqrt{2\pi b}}\begin{pmatrix}
				1\\0\\0
			\end{pmatrix},\:\phi_{n,1}(x)=\frac{1}{\sqrt{2\pi b}}\begin{pmatrix}
				e^{inx}\\0\\0
			\end{pmatrix},\notag\\
			&\phi_{n,2}(x)=\frac{1}{\sqrt{2\pi \rho_s}}\begin{pmatrix}
				0\\e^{inx}\\0
			\end{pmatrix},\:\phi_{n,3}(x)=\sqrt{\frac{\mu}{2\pi\kappa}}\begin{pmatrix}
				0\\0\\e^{inx}
			\end{pmatrix},\quad n\in\mathbb{Z}^*.
		\end{align}
		Let us define the following finite dimensional spaces
		$$\mathbf{V}_{0}=\text{span} \left\lbrace \phi_{0,1}\right\rbrace , \:\mathbf{V}_n=\text{span} \left\lbrace \phi_{n,l},\: l=1,2,3\right\rbrace , n\in\mathbb{Z}^*,$$where `span' stands for the vector space generated by those functions. One can verify that $\mathcal {{Z}}_m=\oplus_{n\in \mathbb{Z}}\mathbf{V}_n.$ 
		\begin{lem}\label{nmaxinv}
			For all $n\in\mathbb{Z}^*$, $\mathbf{V}_n$ is invariant under $\mathcal A$ and $$\mathcal A_n=\mathcal A|_{\mathbf{V}_n}\in\mathcal{L}(\mathbf{V}_n),$$has the following matrix representation$$\begin{pmatrix}
				-inu_s&-in\sqrt{b\rho_s}&0\\
				-in\sqrt{b\rho_s}&-inu_s&in\sqrt{\frac{\mu}{\kappa\rho_s}}\\
				0&in\sqrt{\frac{\mu}{\kappa\rho_s}}&-\frac{1}{\kappa}
			\end{pmatrix}$$in the basis $\left\lbrace \phi_{n,l},\: l=1,2,3\right\rbrace_{n\in\mathbb{Z}^*} $ of $\mathbf{V}_n$.
		\end{lem}
		\begin{prop}\label{sp}
			The spectrum of the operator $(\mathcal{A}, \mathcal{D}(\mathcal{A}; \mathcal {{Z}}_m))$ consists of $0$ and three sequences $\lambda_n^1$, $\lambda_n^2$ and $\lambda_n^3$ of eigenvalues. Furthermore, we have the following asymptotic behaviors:
			\begin{equation}\label{nmaxasmlambda}
				\left.
				\begin{aligned}
					\lambda_n^1=&-\omega_1+i \beta_1n + O\left( \frac{1}{|n|}\right),\\
					\lambda_n^2=&-\omega_2+i \beta_2n + O\left( \frac{1}{|n|}\right),\\
					\lambda_n^3=&-\omega_3+i \beta_3n + O\left( \frac{1}{|n|}\right),
				\end{aligned}
				\right\}
			\end{equation}
			for $|n|$ large enough, where $\beta_j, j=1,2,3$ are the distinct real roots of the equation
			\begin{equation}\label{nmaxpolybeta}
				r^3+2u_sr^2+\left( u_s^2-b\rho_s-\frac{\mu}{\kappa\rho_s}\right)r-\frac{\mu u_s}{\kappa\rho_s}=0,
			\end{equation}
			and 
			\begin{equation}\label{nmaxwn}
				\omega_j=\frac{ {b\rho_s-u_s^2-2u_s\beta_j-\beta_j^2}}{\kappa\left( 3\beta_j^2+4u_s\beta_j+u_s^2-b\rho_s-\mu/\kappa\rho_s\right) }{\neq \omega_l\neq 0,\text{ for }j\neq l; \quad j,l=1,2,3.}
			\end{equation}
		\end{prop}
		\begin{proof}
			We see that $\mathcal{A}\phi_{0,1}=0\cdot\phi_{0,1}$ and $ \phi_{0,1}\neq \mathbf 0$. Thus $0$ is an eigenvalue of $\mathcal{A}$.
			
			From \Cref{nmaxinv}, the characteristic equation is given by:
			\begin{align}\label{nmaxchar.poly}
				\mathcal{F}_n(\lambda) : = \lambda^3+\left( \frac{1}{\kappa}+2inu_s\right)\lambda^2+&\left(-u_s^2n^2+b\rho_sn^2+\frac{\mu n^2}{\kappa\rho_s}+\frac{2iu_sn}{\kappa}
				\right)\lambda\notag\\
				+&\left( -\frac{u_s^2n^2}{\kappa}+\frac{b\rho_sn^2}{\kappa}+\frac{i\mu u_sn^3}{\kappa\rho_s}\right) =0,\quad n\in\mathbb{Z}^*.
			\end{align}
			Let $\lambda_n^l=\eta_n^l+i\tau_n^l$ with $\eta_n^l\in\mathbb{R}$, $\tau_n^l\in \mathbb{R}$, and $l=1,2,3$ are the roots of \eqref{nmaxchar.poly}. Using \cite[Theorem 3.2]{ZZ}, we can show that all the roots of \eqref{nmaxchar.poly} have negative real parts, i.e., 
			\begin{equation}\label{negrealpart}
				\eta_n^l <0 , \text{ for } l=1,2,3 \text{  and } n\in\mathbb{Z}^*.
			\end{equation}
			
			From the relation between roots and coefficients of the equation \eqref{nmaxchar.poly}, we get
			\begin{align}
				&\eta_n^1+\eta_n^2+\eta_n^3=-\frac{1}{\kappa},\label{nmaxrc1}\\
				&\tau_n^1+\tau_n^2+\tau_n^3=-2u_s n,\label{nmaxrc2}\\
				&\left(\eta_n^1\eta_n^2+\eta_n^1\eta_n^3+\eta_n^2\eta_n^3 \right) -\left(\tau_n^1\tau_n^2+\tau_n^1\tau_n^3+\tau_n^2\tau_n^3 \right) =\left(b\rho_s+\frac{\mu}{\kappa\rho_s}-u^2_s \right)n^2,\label{nmaxrc3}\\ 
				&\eta_n^1\left(\tau_n^2+\tau_n^3 \right) +\eta_n^2\left(\tau_n^1+\tau_n^3 \right)+\eta_n^3\left(\tau_n^1+\tau_n^2 \right)=\frac{2u_s}{\kappa}n,\label{nmaxrc4}\\
				&\eta_n^1\eta_n^2\eta_n^3-\eta_n^1\tau_n^2\tau_n^3-\eta_n^2\tau_n^1\tau_n^3-\eta_n^3\tau_n^1\tau_n^2=\frac{u^2_s-b\rho_s}{\kappa}n^2,\label{nmaxrc5}\\
				&\tau_n^1\tau_n^2\tau_n^3-\tau_n^1\eta_n^2\eta_n^3-\eta_n^1\tau_n^2\eta_n^3-\eta_n^1\eta_n^2\tau_n^3=\frac{\mu u_s}{\kappa\rho_s}n^3.\label{nmaxrc6}
			\end{align}
			We will prove \eqref{nmaxasmlambda} in the following steps.
			
			\textit{Step 1: Asymptotic behavior of real part of eigenvalues: }
			
			Since $\eta_n^l <0 $, for $l=1,2,3$ and $n\in\mathbb{Z}^*$, from \eqref{nmaxrc1}, we observe that the sequence $\left\lbrace \eta_n^l\right\rbrace_{n\in\mathbb{Z}^*} $ is bounded for $l=1,2,3.$ Thus we get $\frac{\eta_n^l}{n}\rightarrow 0$ as $|n|\rightarrow \infty$, for $l=1,2,3.$
			
			\vspace{2mm}
			
			\textit{Step 2: Asymptotic behavior of imaginary part of  eigenvalues: }	
			
			Thanks to \eqref{nmaxrc1}, \eqref{nmaxrc2} and \eqref{nmaxrc3}, we have
			\begin{align}\label{nmaxrc61}
				\left(\left( \eta_n^1\right)^2+\left( \eta_n^2\right)^2+\left( \eta_n^3\right)^2 \right) - \left(\left( \tau_n^1\right)^2+\left( \tau_n^2\right)^2+\left( \tau_n^3\right)^2 \right)= \frac{1}{\kappa^2}- 2\left( u_s^2+b\rho_s+\frac{\mu}{\kappa\rho_s}\right)n^2.
			\end{align} 
			Set $\tilde{\tau}_n^l=\frac{\tau_n^l}{n}$, $l=1,2,3$. Then using Step 1 and \eqref{nmaxrc61}, it follows that
			\begin{equation}
				\lim\limits_{|n|\rightarrow \infty}\left(\left( \tilde\tau_n^1\right)^2+\left( \tilde\tau_n^2\right)^2+\left( \tilde\tau_n^3\right)^2 \right)= 2\left( u_s^2+b\rho_s+\frac{\mu}{\kappa\rho_s}\right).
			\end{equation}
			Thus the sequence $\left\lbrace \tilde\tau_n^l\right\rbrace_{n\in\mathbb{Z}^*} $ is bounded for $l=1,2,3.$	
			
			Also from \eqref{nmaxrc2}, \eqref{nmaxrc3} and \eqref{nmaxrc6},  we deduce that 
			\begin{equation}\label{nmaxrc7}
				\left.
				\begin{aligned}
					\tilde\tau_n^1+\tilde\tau_n^2+\tilde\tau_n^3 =& -2u_s,\\
					\lim\limits_{|n|\rightarrow \infty}	\left( \tilde\tau_n^1\tilde\tau_n^2+\tilde\tau_n^1\tilde\tau_n^3+\tilde\tau_n^2\tilde\tau_n^3\right) =& u_s^2-b\rho_s-\frac{\mu}{\kappa\rho_s},\\
					\lim\limits_{|n|\rightarrow \infty}	 \tilde\tau_n^1\tilde\tau_n^2\tilde\tau_n^3 =& \frac{\mu  u_s}{\kappa\rho_s}.
				\end{aligned}
				\right\}
			\end{equation}
			Thus from \eqref{nmaxrc7}, it follows that
			\begin{equation}\label{nmaxrc8}
				\left.
				\begin{aligned}
					\tilde\tau_n^1+\tilde\tau_n^2+\tilde\tau_n^3 =& -2u_s,\\
					\tilde\tau_n^1\tilde\tau_n^2+\tilde\tau_n^1\tilde\tau_n^3+\tilde\tau_n^2\tilde\tau_n^3 =& u_s^2-b\rho_s-\frac{\mu}{\kappa\rho_s}+\delta^1_n,\\
					\tilde\tau_n^1\tilde\tau_n^2\tilde\tau_n^3 =& \frac{\mu  u_s}{\kappa\rho_s}+\delta^2_n,
				\end{aligned}
				\right\}
			\end{equation}
			where $\left| \delta^l_n\right| \rightarrow 0$, as $|n|\rightarrow \infty$, for $l=1,2$. Therefore, $\tilde\tau_n^l,  l=1,2,3 $ satisfies the following equation:
			\begin{equation}\label{p_nx}
				\mathcal{P}_n\left(r \right) : = r^3+2u_sr^2+\left( u_s^2-b\rho_s-\frac{\mu}{\kappa\rho_s}+\delta^1_n\right)r-\left( \frac{\mu u_s}{\kappa\rho_s}+\delta^2_n\right) =0,\quad n\in\mathcal{Z}^*.
			\end{equation}
			Set
			\begin{equation}\label{nmaxrimag}
				\mathcal{P}\left(r \right) : = r^3+2u_sr^2+\left( u_s^2-b\rho_s-\frac{\mu}{\kappa\rho_s}\right)r-\frac{\mu u_s}{\kappa\rho_s}=0.
			\end{equation}
			Let us denote the roots of \eqref{nmaxrimag} by $\beta_l$, $ l=1,2,3$. Observe that the coefficients of $\mathcal{P}_n$ converge to the coefficients of $\mathcal{P}$, as $|n|\rightarrow \infty$. Then using the fact that the roots of a polynomial continuously depend on the coefficients of the polynomial (see \cite{RCDC}), we get that the roots of the polynomial $\mathcal{P}_n$ converge to the roots of the polynomial $\mathcal{P}$, as $|n|\rightarrow \infty$, i.e., $\frac{\tau_n^l}{n}\rightarrow \beta_l$ as $|n|\rightarrow \infty$, for $l=1,2,3.$
			
			\vspace{2mm}
			
			\textit{Step 3: Behavior of roots of the polynomial $\mathcal{P}$: }
			
			The discriminant of the cubic polynomial $\mathcal{P}$ is given by
			{
				
				{
					\begin{align}\label{nmaxdiscriminant}
						D= &\left( 2u_s\right)^2\left( u_s^2-b\rho_s-\frac{\mu}{\kappa\rho_s}\right)^2-4\left(u_s^2-b\rho_s-\frac{\mu}{\kappa\rho_s} \right)^3-4\left(2u_s \right)^3\left(-\frac{\mu u_s}{\kappa\rho_s} \right)\notag\\
						-& 27\left(-\frac{\mu u_s}{\kappa\rho_s} \right)^2+18\left( 2u_s\right)\left( u_s^2-b\rho_s-\frac{\mu}{\kappa\rho_s}\right)\left(-\frac{\mu u_s}{\kappa\rho_s} \right)\notag\\
						= & 4b\rho_s\left( b\rho_s-u_s^2\right)^2+\frac{\mu^2 u_s^2}{\kappa^2\rho_s^2}+\frac{20\mu b u_s^2}{\kappa}+\frac{12\mu b^2\rho_s}{\kappa}+\frac{12\mu^2 b}{\kappa^2\rho_s}+\frac{4\mu^3}{\kappa^3\rho_s^3} . 
					\end{align}	
				}

			}
			{	We observe that $D>0$, indicating that the roots of the polynomial $\mathcal{P}$ are both real and distinct. Consequently, we obtain:}	
			\begin{equation}\label{nmaxnodoubleroot}
				\mathcal{P}'\left( \beta_l\right)  = 3\beta_l^2+4u_s\beta_l +u_s^2-b\rho_s -\frac{\mu}{\kappa\rho_s}\neq 0,\text{ for }l=1,2,3.	
			\end{equation}

			
			
			\textit{Step 4: Asymptotic behavior of the eigenvalues: }
			
			Using Step 1 and Step 2, we can rewrite the roots of \eqref{nmaxchar.poly} as 
			\begin{equation}\label{nmaxsetlambda}
				\lambda_n = i\beta n +\epsilon_n,
			\end{equation}
			where $\beta$ is a root of the equation \eqref{nmaxrimag} and
			\begin{equation}\label{expepsilonn}
				\epsilon_n= n\theta_n^1+i n \theta_n^2,\quad \theta_n^l\rightarrow 0 \text{ as } \left|n \right|\rightarrow \infty, \text{ for }l=1,2. 
			\end{equation}
			From \eqref{expepsilonn}, it follows that
			\begin{equation}\label{cgsepsilon_n}
				\frac{\epsilon_n}{n}\rightarrow 0 \text{ as } \left|n \right|\rightarrow \infty.
			\end{equation}
			Now, we will see the asymptotic behavior of $\epsilon_n$. From \eqref{nmaxchar.poly} and \eqref{nmaxsetlambda}, we get that $\epsilon_n$ satisfy the following equation:
			\begin{align*}
				\epsilon_n^3+\left(\frac{1}{\kappa}+\left( 3\beta +2u_s\right) in \right)\epsilon_n^2+\left(\left( -3\beta^2-4u_s\beta -u_s^2+b\rho_s +\frac{\mu}{\kappa\rho_s}\right) n^2+2\frac{\beta+u_s}{\kappa}in \right)\epsilon_n&\notag\\
				+ \frac{b\rho_s-u_s^2-2u_s\beta-\beta^2}{\kappa}n^2=0.&
			\end{align*}
			Using \eqref{nmaxnodoubleroot} and the fact that $\beta$ is root of \eqref{nmaxrimag}, we can rewrite the above equation in the following form
			\begin{align}\label{nmaxepsilonn}
				\epsilon_n=  -\frac{\frac{\epsilon_n^2}{\kappa n^2} {+}\frac{b\rho_s-u_s^2-2u_s\beta-\beta^2}{\kappa}}{\left(\frac{2u_s+3\beta}{2}+\frac{\epsilon_n}{in} \right)^2+\mathcal{P}'(\beta)-\left(\frac{\left(2u_s+3\beta \right)^2 }{4}+2i\frac{u_s+\beta}{\kappa n}\right)}.
			\end{align}
			Denoting numerator of \eqref{nmaxepsilonn} by $c_n$ and denominator of \eqref{nmaxepsilonn} by $d_n$, then using \eqref{nmaxnodoubleroot} and \eqref{cgsepsilon_n}, we get
			\begin{align}
				&c_n\rightarrow {-}\frac{b\rho_s-u_s^2-2u_s\beta-\beta^2}{\kappa} \text{ as } \left|n \right|\rightarrow \infty,\label{cgsepsilon1}\\
				&d_n\rightarrow \mathcal{P}'(\beta)\neq 0 \text{ as } \left|n \right|\rightarrow \infty.\label{cgsepsilon2}
			\end{align}
			Therefore \eqref{nmaxepsilonn} is well-defined for $\left| n\right|$ large enough. Thus from \eqref{nmaxepsilonn}, \eqref{cgsepsilon1} and \eqref{cgsepsilon2}, it follows that
			\begin{equation}\label{cgsepsilon4}
				\epsilon_n\rightarrow -\frac{ {b\rho_s-u_s^2-2u_s\beta-\beta^2}}{\kappa \mathcal{P}'(\beta)} : = -\omega (\text{say}), \text{ as } \left|n \right|\rightarrow \infty.
			\end{equation}
			Next, we will find the explicit asymptotic of the eigenvalues. 
			For that, let us define
			\begin{equation}\label{nmaximplicitfunc}
				F\left( \zeta,\eta\right)=\zeta+ \frac{-\frac{1}{\kappa }\zeta^2\eta^2{+}\frac{b\rho_s-u_s^2-2u_s\beta-\beta^2}{\kappa} }{\left(\frac{2u_s+3\beta}{2}-{\zeta \eta} \right)^2+\mathcal{P}'(\beta)-\left(\frac{\left(2u_s+3\beta \right)^2 }{4}+2\frac{u_s+\beta}{\kappa }\eta\right)},\quad \zeta,\eta\in \mathbb{C}.
			\end{equation}
			{ Using \eqref{nmaxnodoubleroot}, we can see that $F$ is well defined and analytic in a neighborhood of $\left( -\omega,0\right) $. }Moreover
			\begin{equation*}
				F\left( -\omega,0\right)= \lim\limits_{\eta\rightarrow 0}F\left( -\omega,\eta\right)=0,
			\end{equation*}
			and
			\begin{align*}
				\frac{\partial F}{\partial \zeta}\left( -\omega,0\right)= \lim\limits_{\eta\rightarrow 0}\frac{\partial F}{\partial \zeta}\left( -\omega,\eta\right)=1\neq  0 .
			\end{align*}
			Then using the Implicit Function Theorem, there exist neighborhoods $N\left( -\omega\right) $ and $N\left(0 \right) $ such that the equation $F\left( \zeta,\eta\right)=0$ has a unique root $\zeta=G\left( \eta\right) $ in $N\left( -\omega\right) $ for any given $\eta\in N\left(0 \right) $. Moreover, the function $G\left( \eta\right) $ is single-valued and analytic in $N\left(0 \right)$, and satisfies the condition $G\left( 0\right)=-\omega$. Therefore we obtain
			\begin{align}\label{nmaximplicit2}
				\zeta=G\left( \eta\right)=&G\left( 0\right)+G'\left( 0\right)\eta+\dots\notag\\
				=& -\omega + O\left( \eta\right)\quad\text{ in }N\left(0 \right).
			\end{align}
			From \eqref{nmaxepsilonn}, \eqref{cgsepsilon4}, \eqref{nmaximplicitfunc} and \eqref{nmaximplicit2}, it follows that, for $|n|$ large enough, $\epsilon_n\in N\left( -\omega\right)$, $\frac{i}{n}\in N(0)$, $F\left( \epsilon_n,\frac{i}{n}\right)=0$ and
			\begin{equation*}
				\epsilon_n=-\omega + O\left( \frac{1}{|n|}\right).
			\end{equation*}
			Thus from \eqref{nmaxsetlambda}, we get
			\begin{equation}\label{nmaxlambdaasm}
				\lambda_n = -\omega + i\beta n + O\left( \frac{1}{|n|}\right),
			\end{equation}
			for $|n|$ large enough.
			
			Since the equation \eqref{nmaxrimag} has three distinct real roots, we get three distinct real values of $\beta$ and for each $\beta$, from \eqref{cgsepsilon4}, we get a real $\omega.$ 
			Furthermore, we can show that $\omega_1\neq\omega_2\neq\omega_3\neq 0$. Indeed, 
			suppose that $\omega_j=0$, which implies
			\begin{equation}\label{nembj}
				\beta_j^2+2u_s \beta_j+u_s^2-b\rho_s=0.
			\end{equation}
			Since $\beta_j$ is a root of the equation \eqref{nmaxrimag}, we get	
			\begin{align*}
				&\beta_j^3+2u_s\beta_j^2+\left( u_s^2-b\rho_s-\frac{\mu}{\kappa\rho_s}\right)\beta_j-\frac{\mu u_s}{\kappa\rho_s}=0,
			\end{align*}
			which along with \eqref{nembj} implies $\beta_j=-u_s$. 
			This is a contradiction, since $-u_s$ is not a root of the equation \eqref{nmaxrimag} $\left( \text{ as }\mathcal{ P}\left( -u_s\right)=b\rho_s u_s >0  \right) $.

			Moreover, from \eqref{nmaxnodoubleroot} and \eqref{cgsepsilon4}, we have
			\begin{equation}\label{wj-wl}
				\omega_j-\omega_l=\frac{\beta_j-\beta_l}{\kappa \mathcal{P}'(\beta_j)\mathcal{P}'(\beta_l)}\left[ \frac{2\mu u_s^2}{\kappa \rho_s\beta_p}{+}\beta_p\left(2b\rho_s-2u_s^2-\frac{\mu}{\kappa\rho_s}\right){+}2u_s(b\rho_s-u_s^2)\right],
			\end{equation}
			where $ j,l,p\in \{1,2,3\}, j\neq l\neq p.$ Using the fact that $\beta_j$ are distinct, from \eqref{wj-wl}, one
			can show that $\omega_j\neq\omega_l$. Thus \Cref{sp} follows.
		\end{proof}
		
		\begin{rem}\label{rem multi}
			From the asymptotic behaviors \eqref{nmaxasmlambda}-\eqref{nmaxwn} of the eigenvalues, it is clear that all eigenvalues are simple at least for large $n$. Thus if there exist multiple eigenvalues  depending on the system parameters $\rho_s, u_s, b, \kappa, \mu$, they are finite in numbers.
		\end{rem}
		{ For the case of multiple eigenvalues, the following results can be derived by appropriately introducing generalized eigenfunctions (for more details, see \Cref{secmultiev}). From the previous remark, we observe that multiple eigenvalues of $\mc A$, if they occur at all, are only finitely many. Since our later analysis requires the asymptotic behaviors of the coefficients of the eigenvectors, we will avoid the details of the explicit computation of the generalized eigenfunctions. However, for clarity, we will compute the generalized eigenfunctions of $\mathcal{A}$ in the case of double eigenvalues.}

		Thus, now onwards, we assume the following :
		\begin{equation}\label{nmaxsimpleeigen}
			\text{The spectrum of $\mathcal{A}$ has simple eigenvalues on $\mathcal {{Z}}_m$.}\tag{${{\mc H}}$}
		\end{equation}

		We define the family $\left\lbrace \xi_0\right\rbrace \cup \left\lbrace \xi_{n,l}\: |\: 1\leq l\leq 3, n\in\mathbb{Z}^*\right\rbrace$ as follows : 
		
		\noindent
		We choose a normalized eigenfunctions of $\mathcal{A}$ for the eigenvalue $\lambda_0=0$ defined by
		\begin{equation*}
			\xi_0=\frac{1}{\sqrt{2b\pi}}\begin{pmatrix}
				1\\0\\0
			\end{pmatrix},
			\end{equation*}
			and for the eigenvalue $\lambda_n^l$, the normalized eigenfunction defined by
			\begin{equation}\label{nmaxequ-xi_n}
				\xi_{n,l}=\frac{1}{\theta_{n,l}}\begin{pmatrix}
					-1\vspace{1mm}\\\frac{\lambda_n^l+inu_s}{{in\rho_s}}\vspace{1mm}\\\frac{\mu\left( \lambda_n^l+inu_s\right) }{\rho_s\left( 1+\kappa\lambda_n^l\right) }
				\end{pmatrix}e^{inx},\: l\in \left\lbrace 1,2,3\right\rbrace,\: n\in\mathbb{Z}^*,
			\end{equation}
			where 
			\begin{equation}\label{nmaxequ-theta_nl}
				\theta_{n,l}=\sqrt{2{\pi}\left(b+\frac{\left| \lambda_n^l+inu_s\right|^2 }{\rho_sn^2} + \frac{\kappa\mu\left| \lambda_n^l+inu_s\right|^2 }{\rho_s^2\left| 1+\kappa\lambda_n^l\right|^2}\right) }.
			\end{equation}
			
			\begin{figure}[ht!]
				\includegraphics[width=.8\textwidth]{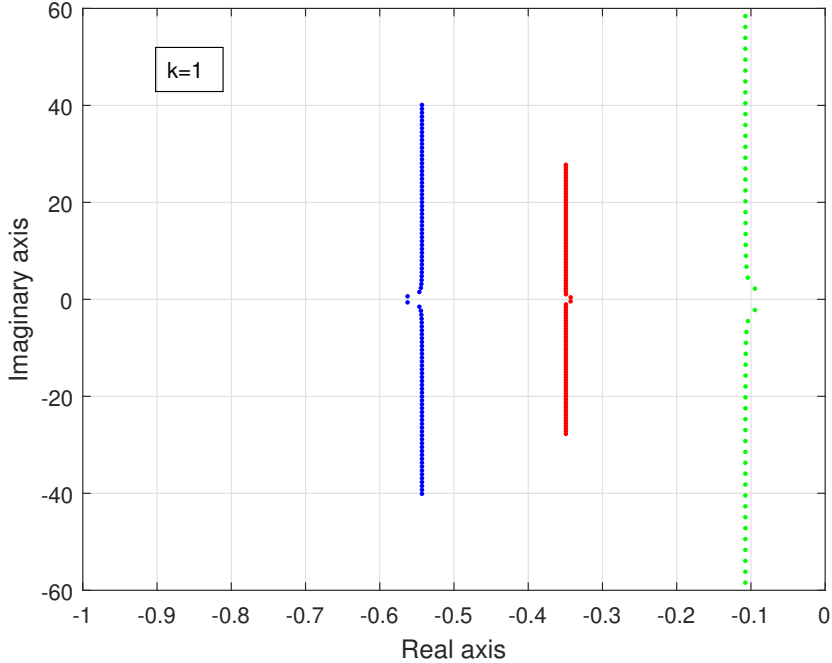}
				
				\caption{Eigenvalues of $\mathcal A$ in the complex plane for $|n|$ varies from $1$ to $30$ when $\mu=\rho_s=u_s=b=1$ and k=1. }     
			\end{figure}
			
			$\left| 1+\kappa\lambda_n^l\right|\neq 0  \left( l=1,2,3\right) $, as $-1/\kappa$ is not a root of the characteristic polynomial \eqref{nmaxchar.poly}.\\
			Similarly, we choose $\left\lbrace \xi^*_0\right\rbrace \cup \left\lbrace \xi^*_{n,l}\: |\: 1\leq l\leq 3, n\in\mathbb{Z}^*\right\rbrace$ as follows : 
			
			$\xi^*_0$ is an eigenfunction of $\mathcal{A}^*$ with the eigenvalue $\lambda_0=0$ defined by
			\begin{equation}\label{nmaxequ-xi_0^*}
				\xi^*_0=\frac{1}{\sqrt{2b\pi}}\begin{pmatrix}
					1\\0\\0
				\end{pmatrix},
			\end{equation}
			and for the eigenvalue $\overline{\lambda_n^l}$, the eigenfunction defined by
			\begin{equation}\label{nmaxequ-xi_n^*}
				\xi^*_{n,l}=\frac{1}{\psi_{n,l}}\begin{pmatrix}
					1\vspace{1mm}\\\frac{\overline{\lambda_n^l}-inu_s}{{in\rho_s}}\vspace{1mm}\\-\frac{\mu\left( \overline{\lambda_n^l}-inu_s\right) }{\rho_s\left( 1+\kappa\overline{\lambda_n^l}\right) }
				\end{pmatrix}e^{inx},\: l\in \left\lbrace 1,2,3\right\rbrace,\: n\in\mathbb{Z}^*,
			\end{equation}
			where
			\begin{equation}\label{nmaxequ-psi_nl}
				\psi_{n,l}=\frac{\sqrt{{2\pi}}\left(-b+\frac{\left(  \overline{\lambda_n^l+inu_s}\right)^2 }{\rho_sn^2} - \frac{\mu\kappa\left(  \overline{\lambda_n^l+inu_s}\right)^2 }{\rho_s^2\left(  1+\kappa\overline{\lambda_n^l}\right)^2}\right)  }{\sqrt{\left(b+\frac{\left| \lambda_n^l+inu_s\right|^2 }{\rho_sn^2} + \frac{\kappa\mu\left| \lambda_n^l+inu_s\right|^2 }{\rho_s^2\left| 1+\kappa\lambda_n^l\right|^2}\right) }}.
			\end{equation} 
			Here for all $n\in\mathbb{Z}^*$, $\psi_{n,l}\neq 0$ for $l=1,2,3 $, because of the assumption \eqref{nmaxsimpleeigen}.		
			
			The choice of this family of eigenfunctions of $\mathcal{A}$ and $\mathcal{A}^*$ ensures the following lemma.
			
			\begin{lem}\label{nmaxbiorthonormality}
				Under the assumption \eqref{nmaxsimpleeigen}, the families $\left\lbrace \xi_0\right\rbrace \cup \left\lbrace \xi_{n,l}\: |\: 1\leq l\leq 3, n\in\mathbb{Z}^*\right\rbrace$ and $\left\lbrace \xi^*_0\right\rbrace \cup \left\lbrace \xi^*_{n,l}\: |\: 1\leq l\leq 3, n\in\mathbb{Z}^*\right\rbrace$ satisfy the following bi-orthonormality relations :
				\begin{align}
					&\left\langle \xi_0, \xi^*_{k,p}\right\rangle_{\mathcal{Z}}=0,\quad \left\langle \xi_{n,l}, \xi^*_{k,p}\right\rangle_{\mathcal{Z}}=\delta_k^n\delta_p^l,\: l,p\in \left\lbrace 1,2,3\right\rbrace,\: k,n\in\mathbb{Z}^*,\label{nmaxbiortho}\\
					& \left\langle \xi_0, \xi^*_{0}\right\rangle_{\mathcal{Z}}=1,\quad \left\langle \xi_{n,l}, \xi^*_{0}\right\rangle_{\mathcal{Z}}=0,\: l\in \left\lbrace 1,2,3\right\rbrace,\: n\in\mathbb{Z}^*,
				\end{align}
				where 
				$$
				\delta_n^l=\begin{cases}
					0, \: n\neq l\\
					1,\: n=l.
				\end{cases}
				$$
				
				Moreover, the asymptotic behaviors of $\theta_{n,l}$ and $\psi_{n,l}$ are as follows :
				\begin{equation}\label{nmaxcgs-theta_nl}
					\begin{aligned}
						\left|\theta_{n,l} \right|, \left|\psi_{n,l} \right|\rightarrow  \sqrt{2{\pi}\left( b+\frac{\left( \beta_l+u_s\right)^2 }{\rho_s}+\frac{\mu\left( \beta_l+u_s\right)^2 }{\kappa\rho_s^2 \beta_l^2}\right)   }, \text{ for }l=1,2,3\text{ as }|n|\rightarrow \infty.\\
					\end{aligned}
				\end{equation}
				Also if we write the eigenfunctions of $\mathcal{A}^*$ in the following way
				\begin{equation}\label{nmaxcoeffofxi}
					\xi_{n,l}^{*}=\frac{1}{\psi_{n,l}}\begin{pmatrix}
						\alpha_{n,l}^1\\\alpha_{n,l}^2\\\alpha_{n,l}^3
					\end{pmatrix} e^{inx},\quad n\in \mathbb{Z}^*, l=1,2,3, 
				\end{equation}
				then we have the following asymptotic behaviors
				\begin{align}
					&\alpha_{n,l}^1= 1, \text{ for }l=1,2,3, \text{ for all  }n\in\mathbb{Z}^*,\label{nmaxcgsalpha1}\\
					&\left|\alpha_{n,l}^2 \right| \rightarrow \left| \frac{ \beta_l+u_s }{ \rho_s}\right|, \text{ for }l=1,2,3, \text{ as }|n|\rightarrow \infty,\label{nmaxcgsalpha2}\\
					&\left|\alpha_{n,l}^3 \right| \rightarrow \left| \frac{\mu\left( \beta_l+u_s\right) }{\kappa \rho_s \beta_l}\right| , \text{ for }l=1,2,3, \text{ as }|n|\rightarrow \infty.\label{nmaxcgsalpha3}
				\end{align} 
				
			\end{lem}	
			Now, we show that the eigenfunctions of $\mathcal{A}$ form a Riesz basis in $\mathcal{Z}_m$.
			Recall from \eqref{nmaxfbasis} that $\left\lbrace \phi_{0,1}, \phi_{n,1},\phi_{n,2}, \phi_{n,3} | n\in\mathbb{Z}^*\right\rbrace $  forms an orthonormal basis in $\mathcal{Z}_m$.

			For the definition of Riesz basis in $\mathcal{Z}_m$, we refer \cite[Definition 2.5.1, Chapter 2, Section
			2.5]{TW09}. In other words, the sequence $\left\lbrace \xi_0\right\rbrace \cup \left\lbrace \xi_{n,l}\: |\: 1\leq l\leq 3\right\rbrace_{ n\in\mathbb{Z}^*}$ forms a Riesz basis in $\mathcal{Z}_m$ if there exists an
			invertible operator $Q\in\mathcal{L}\left( \mathcal{Z}_m\right) $ such that
			\begin{equation*}
				Q\phi_{0,1}=\xi_0,\: Q\phi_{n,1}=\xi_{n,1},\:Q\phi_{n,2}=\xi_{n,2},\:Q\phi_{n,3}=\xi_{n,3},\:\forall n\in\mathbb{Z}^*.
			\end{equation*}
			
			\begin{lem}\label{nmaxlem-bais_representation}
				Assume \eqref{nmaxsimpleeigen}. Let us set $\mathcal{B}_n^1=\left\lbrace \phi_{n,l}\: | l=1,2,3\right\rbrace $
				and $\mathcal{B}_n^2=\left\lbrace \xi_{n,l}\: | l=1,2,3\right\rbrace $ for all $n\in\mathbb{Z}^*$, where $\left\lbrace \phi_{n,l}\right\rbrace_{1\leq l\leq 3} $ and $\left\lbrace \xi_{n,l}\right\rbrace_{1\leq l\leq 3} $ are defined as in \eqref{nmaxfbasis} and \eqref{nmaxequ-xi_n}. Let $z_n\in \mathbf{V}_n$ be expressed in the basis $\mathcal{B}_n^1$ and $\mathcal{B}_n^2$ as follows :
				\begin{equation}\label{nmaxequ-z_n}
					z_n=\sum\limits_{p=1}^{3}c_{n,p}\phi_{n,p}=\sum\limits_{l=1}^{3}d_{n,l}\xi_{n,l},
				\end{equation}
				and let the operator $\Gamma_n\in \mathbb{C}^{3\times 3}$ be defined by
				\begin{equation}\label{nmaxwedge}
					\Gamma_n \begin{pmatrix}
						c_{n,1}\\c_{n,2}\\c_{n,3}
					\end{pmatrix}= \begin{pmatrix}
						d_{n,1}\\d_{n,2}\\d_{n,3}
					\end{pmatrix}.
				\end{equation}
				Then there exist positive constants $l_1$ and $l_2$ independent of $n$ such that
				\begin{equation*}
					\left\| \Gamma_n\right\| < l_1, \quad  \left\| \Gamma_n^{-1}\right\| < l_2,\text{ for all }n\in \mathbb{Z}^*.
				\end{equation*}
			\end{lem}
			\begin{proof}
				Let $n\in\mathbb{Z}^*$ and $z_n$ as in \eqref{nmaxequ-z_n}. Taking inner product of $z_n$ with each $\xi_{n,l}^*$ for $1\leq l\leq 3$, and using the bi-orthonormality relation \eqref{nmaxbiortho}, we get
				\begin{align*}
					d_{n,l}=&\left\langle z_n,\xi_{n,l}^*\right\rangle_{\mathcal {{Z}}}=\sum\limits_{p=1}^{3}c_{n,p}\left\langle \phi_{n,p},\xi_{n,l}^*\right\rangle_{\mathcal {{Z}}}\\
					=&\sqrt{2\pi b}\frac{\overline{\alpha_{n,l}^1}}{\overline{\psi_{n,l}}}c_{n,1}+\sqrt{2\pi \rho_s}\frac{\overline{\alpha_{n,l}^2}}{\overline{\psi_{n,l}}}c_{n,2}+\sqrt{\frac{2\pi \kappa}{\mu}}\frac{\overline{\alpha_{n,l}^3}}{\overline{\psi_{n,l}}}c_{n,3}.
				\end{align*}
				Then the matrix $\Gamma_n$ in \eqref{nmaxwedge} is given by
				\begin{equation}\label{nmaxGamma}
					\Gamma_n=\begin{pmatrix}
						\sqrt{2\pi b}\frac{\overline{\alpha_{n,1}^1}}{\overline{\psi_{n,1}}}&\sqrt{2\pi \rho_s}\frac{\overline{\alpha_{n,1}^2}}{\overline{\psi_{n,1}}}&\sqrt{\frac{2\pi \kappa}{\mu}}\frac{\overline{\alpha_{n,1}^3}}{\overline{\psi_{n,1}}}\\
						\sqrt{2\pi b}\frac{\overline{\alpha_{n,2}^1}}{\overline{\psi_{n,2}}}&\sqrt{2\pi \rho_s}\frac{\overline{\alpha_{n,2}^2}}{\overline{\psi_{n,2}}}&\sqrt{\frac{2\pi \kappa}{\mu}}\frac{\overline{\alpha_{n,2}^3}}{\overline{\psi_{n,2}}}\\
						\sqrt{2\pi b}\frac{\overline{\alpha_{n,3}^1}}{\overline{\psi_{n,3}}}&\sqrt{2\pi \rho_s}\frac{\overline{\alpha_{n,3}^2}}{\overline{\psi_{n,3}}}&\sqrt{\frac{2\pi \kappa}{\mu}}\frac{\overline{\alpha_{n,3}^3}}{\overline{\psi_{n,3}}}
					\end{pmatrix}.
				\end{equation}
				Then using \eqref{nmaxsimpleeigen}, we obtain
				\begin{align}\label{nmaxdetwedge}
					\det \left( \Gamma_n\right) = & \frac{2\pi\sqrt{2\pi b\rho_s\kappa\mu}}{in\rho_s^2\overline{\psi_{n,1}}\overline{\psi_{n,2}}\overline{\psi_{n,3}}}\begin{vmatrix}
						1&\lambda_n^1+inu_s&\frac{\lambda_n^1+inu_s}{1+\kappa\lambda_n^1}\\
						1&\lambda_n^2+inu_s&\frac{\lambda_n^2+inu_s}{1+\kappa\lambda_n^2}\\
						1&\lambda_n^3+inu_s&\frac{\lambda_n^3+inu_s}{1+\kappa\lambda_n^3}
					\end{vmatrix}\notag\\
					= & \frac{2\pi\kappa\sqrt{2\pi b\rho_s\kappa\mu}}{in\rho_s^2\overline{\psi_{n,1}}\overline{\psi_{n,2}}\overline{\psi_{n,3}}} \frac{\left( \lambda_n^1-\lambda_n^2\right) \left( \lambda_n^1-\lambda_n^3\right)\left( \lambda_n^2-\lambda_n^3\right)\left(1-\kappa inu_s \right)  }{\left( 1+\kappa\lambda_n^1\right) \left(1+\kappa\lambda_n^2 \right)\left( 1+\kappa\lambda_n^3\right) } \neq 0. 
				\end{align}
				Using \eqref{nmaxasmlambda} and \eqref{nmaxcgs-theta_nl}, we get
				\begin{equation*}
					\left| \det \left( \Gamma_n\right)\right|\rightarrow C\text{ as }|n|\rightarrow \infty,
				\end{equation*}
				where $C$ is a positive constant.
				Thus, for each $n\in\mathbb{Z}^*, \Gamma_n$ is invertible with the inverse $\Gamma_n^{-1}$.

				From the matrix $\Gamma_n$, using  \eqref{nmaxcgs-theta_nl}, \eqref{nmaxcgsalpha1}, \eqref{nmaxcgsalpha2} and \eqref{nmaxcgsalpha3}, we get a positive constant $l_1$ independent of $n$ such that
				\begin{equation*}
					\left\| \Gamma_n\right\| < l_1, \:\forall\: n\in\mathbb{Z}^*.
				\end{equation*}
				Similarly, we can get a positive constant $l_2$ independent of $n$ such that
				\begin{equation*}
					\left\| \Gamma_n^{-1}\right\| < l_2, \:\forall\: n\in\mathbb{Z}^*.
				\end{equation*}
			\end{proof}
			Using \Cref{nmaxlem-bais_representation} and from the second part of  \cite[Proposition 2.5.3, Chapter 2, Section 2.5]{TW09}, we can see that the eigenfunctions of $\mathcal{A}$, $\big\lbrace\xi_0 , \xi_{n,l}\: | 1\leq l\leq 3, \: n\in \mathbb{Z}^*\big\rbrace $ forms a Riesz basis in $\mathcal{Z}_m$. Also, using \cite[Proposition 2.8.6, Chapter 2, Section 2.8]{TW09}, \cite[Definition 2.6.1, Chapter 2, Section 2.6]{TW09} and  \Cref{nmaxbiorthonormality}, we obtain
			$\left\lbrace\xi_0^* , \xi^*_{n,l}\: | 1\leq l\leq 3, \: n\in \mathbb{Z}^*\right\rbrace $ forms a Riesz basis in $\mathcal{Z}_m$.
			\begin{prop}\label{nmaxbasis_representation}
				Assume \eqref{nmaxsimpleeigen} holds. Then any $z\in{\mathcal {{Z}}_m}$ can be uniquely represented as
				\begin{equation}\label{nmaxexpansion_z}
					z=\sum\limits_{n\in\mathbb{Z}^*}\sum\limits_{l=1}^{3}\left\langle z, \xi^*_{n,l} \right\rangle_{\mathcal {{Z}}} \xi_{n,l}+\left\langle z, \xi^*_{0} \right\rangle_{\mathcal {{Z}}} \xi_{0}. 
				\end{equation}
				There exist positive numbers $C_1$ and $C_2$ such that
				\begin{equation}\label{nmaxequ-est-z_0}
					C_1\left( \left| \left\langle z, \xi^*_{0} \right\rangle\right|^2+\sum_{n\in\mathbb{Z}^*}\sum_{l=1}^{3}\left| \left\langle z, \xi^*_{n,l} \right\rangle\right|^2 \right) \leq \left\|z \right\|^2_{{\mathcal {{Z}}}}\leq C_2\left(\left| \left\langle z, \xi^*_{0} \right\rangle\right|^2+ \sum_{n\in\mathbb{Z}^*}\sum_{l=1}^{3}\left| \left\langle z, \xi^*_{n,l} \right\rangle\right|^2 \right). 
				\end{equation}

				Similarly, any $z\in\mathcal {{Z}}_m$ can be uniquely represented in the basis $\left\lbrace\xi_0^* \right\rbrace \cup \{ \xi^*_{n,l}\: | 1\leq l\leq 3, \: n\in \mathbb{Z}^*\} $ by
				\begin{equation}\label{nmaxexpansion_z-adj}
					z=\sum\limits_{n\in\mathbb{Z}^*}\sum\limits_{l=1}^{3}\left\langle z, \xi_{n,l} \right\rangle_{\mathcal {{Z}}} \xi^*_{n,l}+\left\langle z, \xi_{0} \right\rangle_{\mathcal {{Z}}} \xi^*_{0}, 
				\end{equation}
				and
				\begin{equation}\label{nmaxequ-est-adj-z_0}
					\frac{1}{C_2}\left( \left| \left\langle z, \xi_{0} \right\rangle\right|^2+\sum_{n\in\mathbb{Z}^*}\sum_{l=1}^{3}\left| \left\langle z, \xi_{n,l} \right\rangle\right|^2 \right) \leq \left\|z \right\|^2_{{\mathcal {{Z}}}}\leq \frac{1}{C_1}\left(\left| \left\langle z, \xi_{0} \right\rangle\right|^2+ \sum_{n\in\mathbb{Z}^*}^{\infty}\sum_{l=1}^{3}\left| \left\langle z, \xi_{n,l} \right\rangle\right|^2 \right). 
				\end{equation}
			\end{prop}
			
			As a consequence of the above proposition, from \cite[Remark 2.6.4, Chapter 2, Section
			2.6]{TW09}, we get the following result of spectrum of $\mathcal{A}$ and $\mathcal{A}^*$.
			\begin{theorem}
				Assume \eqref{nmaxsimpleeigen} holds. Then the spectrum of the operator $(\mathcal{A}, \mathcal{D}(\mathcal{A}; \mathcal {{Z}}_m))$ and $(\mathcal{A}^*, \mathcal{D}(\mathcal{A}^*; \mathcal {{Z}}_m))$ are given by 
				\begin{align*}
					\sigma(\mathcal{A}) & =\left\lbrace \lambda_0=0,  \lambda_n^j,\: n\in\mathbb{Z}^*,\: j=1,2,3\right\rbrace,\\
					\sigma(\mathcal{A}^*) & =\left\lbrace \lambda_0,  \overline{\lambda_n^j},\: n\in\mathbb{Z}^*,\: j=1,2,3\right\rbrace.
				\end{align*}
			\end{theorem}
			{
				
				\subsection{ Finite dimensional projected system}\label{secproj}
				Here, we will discuss the projection of the operators onto eigenspaces, which will aid us in proving \Cref{thm_pos}. Let's define the following finite-dimensional spaces:
				\begin{equation}\label{def-Zn}
					\mathbf{Z}_0=\text{span} \left\lbrace \xi_{0} \right\rbrace,\:\mathbf{Z}_n=\text{span} \left\lbrace \xi_{n,l},\: l=1,2,3\right\rbrace , n\in \mathbb{Z}^*.
				\end{equation}
				\Cref{nmaxbasis_representation} shows that the space ${\mathcal Z_m}$ is  the orthogonal sum of the subspaces $\left\lbrace  \mathbf{Z}_n\right\rbrace_{n\in\mathbb{Z}} $, i.e., 
				\begin{equation*}
					\mathcal{Z}_m= \oplus_{n\in\mathbb{Z}}\mathbf{Z}_n.
				\end{equation*}
				With reference to \eqref{nmaxexpansion_z}, let us define the projection $\pi_n\in \mathcal{L}\left(\mathcal{Z}_m \right) $ with Range $\pi_n\subset \mathbf{Z}_n$ by
				\begin{equation}\label{def-pi}
					\pi_n z=  \sum\limits_{l=1}^{3}\left\langle z, \xi^*_{n,l} \right\rangle \xi_{n,l},\quad \forall\, n\in\mathbb{Z}^*,\quad 
					\pi_0 z = \left\langle z, \xi^*_{0} \right\rangle \xi_{0}, \quad \forall\, z\in \mathcal{Z}_m.
				\end{equation}
				Similarly, defining 
				\begin{equation}\label{def-adjZn}
					\mathbf{Z}^*_0=\text{span} \left\lbrace \xi^*_{0} \right\rbrace,\:\mathbf{Z}^*_n=\text{span} \left\lbrace \xi^*_{n,l},\: l=1,2,3\right\rbrace , n\in\mathbb{Z}^*,
				\end{equation}
				we get $\mathcal{Z}_m= \oplus_{n\in\mathbb{Z}}\mathbf{Z}^*_n$. 
				Referring to \eqref{nmaxexpansion_z-adj}, the projection $\pi^*_n\in \mathcal{L}\left(\mathcal{Z}_m \right) $ with Range $\pi^*_n\subset \mathbf{Z}^*_n$ by
				\begin{equation}\label{def-adj-pi}
					\pi^*_n z=  \sum\limits_{l=1}^{3}\left\langle z, \xi_{n,l} \right\rangle \xi^*_{n,l},\quad \forall\, n\in\mathbb{Z}^*,\quad 
					\pi^*_0 z = \left\langle z, \xi_{0} \right\rangle \xi^*_{0}, \quad \forall\, z\in \mathcal{Z}_m.
				\end{equation}
				
				Next, let $f_2=0, f_3=0, f_1=f$ and $\mathcal{O}_1= (0,2\pi)$. Then the system \eqref{nmaxeq3} becomes:
				\begin{equation} \label{opp-eqn}
					\dot{z}(t) = \mathcal{A} z(t) + \mathcal{B} f(t), \quad t\in (0,T), \qquad z(0) = z_{0},
				\end{equation}
				where $z(t) = (\rho(t,\cdot), u(t, \cdot), S(t,\cdot))^{\top},$ $z_{0} =(\rho_{0}, u_{0}, S_0)^{\top}\in\mathcal{Z}_m$ , $\left( \mathcal{A},\mathcal{D}\left( \mathcal{A};\mathcal{Z}_m\right) \right) $ is as defined in \eqref{op} and the control operator $\mathcal{B} \in \mathcal{L}(L^2(0,2\pi);{\mathcal{Z}_m})$ is defined as follows:
				\begin{equation}\label{defn-Bf(t)}
					\mathcal{B}f=\begin{pmatrix}
						f\\0\\0
					\end{pmatrix}, \quad \forall\, f\in L^2(0,2\pi).
				\end{equation}

				For each $n\in\mathbb{Z}$, since $\mathbf{Z}_n$ is the eigenspace of $\mathcal{A}$ with finite dimension, $\mathbf{Z}_n$ is a subspace of $\mathcal{D}(\mathcal A ; \mathcal{Z}_m)$ and the following result holds.

				\begin{lem}\label{lem-An}
					For each $n\in\mathbb{Z}$, the space $\mathbf{Z}_n$ defined in \eqref{def-Zn} is invariant under $\mathcal{A}$ and
					\begin{equation}
						\pi_n \mathcal{A}=\pi_n\mathcal{A}\pi_n=\mathcal{A}\pi_n.
					\end{equation}
					For each $n\in\mathbb{Z}^*$, the matrix representation of $\pi_n\mathcal{A}|_{\mathbf{Z}_n}$ in the basis $\left\lbrace \xi_{n,l},\xi_{n,2},\xi_{n,3}\right\rbrace $ is denoted as $\mathcal{A}_n$, defined by
					\begin{equation}
						\mathcal{A}_n=\begin{pmatrix}
							\lambda_n^1&0&0\\
							0&\lambda_n^2&0\\
							0&0&\lambda_n^3
						\end{pmatrix},
					\end{equation}
					where $\{\lambda_n^\ell \mid \ell=1,2,3\}$ represents the eigenvalues of $\mathcal{A}$ as mentioned in \Cref{sp}. For $n=0$, $\mathcal{A}_0=0$.
					
					For all $n\in\mathbb{Z}$, the adjoint of $\pi_n \mathcal{A}|_{\mathbf{Z}_n}$ is defined by $(\mathcal{A}^*\pi^*_n)|_{\mathbf{Z}^*_n}$ with the matrix representation 
					$\mathcal{A}^*_n$ in the basis of $\mathbf{Z}^*_n$. 
					
					Furthermore, there exists a positive constant $C$ independent of $n$ such that 
					\begin{equation}\label{equniformfinsemigr}
						\|e^{t \pi_n \mathcal{A}}\|_{\mathcal{L}(\mathbf{Z}_n)}\le C, \quad \|e^{t (\pi_n \mathcal{A})^*}\|_{\mathcal{L}(\mathbf{Z}^*_n)}\le C, \quad \forall\, n\in\mathbb{Z}.
					\end{equation}
				\end{lem}

				Now, we introduce the finite dimensional subspaces $E_n$ of $L^2(0,2\pi)$ defined by
				\begin{equation}\label{def-E_n}
					E_0=\text{ span }\left\lbrace \frac{1}{\sqrt{2\pi}}\right\rbrace,\quad E_n=\text{ span }\left\lbrace \frac{1}{\sqrt{2\pi}} e^{inx};\: x\in (0,2\pi)\right\rbrace,\quad n\in\mathbb{Z}^*. 
				\end{equation}
				Then $L^2(0,\pi)=\oplus_{n\in\mathbb{Z}}E_n$ and any $g\in L^2\left(0,T;L^2(0,2\pi) \right) $ can be uniquely expressed in the form
				\begin{equation}\label{f=f_n}
					g=\sum\limits_{n\in\mathbb{Z}}g_n,\text{ where }g_n\in E_n, \quad \text{with}\quad 
					\left\| g\right\|^2_{L^2(0,2\pi)}= \sum\limits_{n\in\mathbb{Z}}\left\| g_n\right\|^2_{L^2(0,2\pi)}.
				\end{equation}
				
				\begin{lem}\label{lem-Bn}
					Let $g\in L^2(0,2\pi) $ be decomposed in the form \eqref{f=f_n}. Then
					\begin{equation}\label{pi_nf=pi_nf_n}
						\pi_n\mathcal{B}g= \pi_n\mathcal{B}g_n,\text{ for all }n\in\mathbb{Z},
					\end{equation}
					and $\pi_n\mathcal{B}|_{E_n}\in \mathcal{L}(E_n, \mathbf{Z}_n)$. 		
					For any $n\in\mathbb{Z}^*$, the matrix representation of $\pi_n \mathcal{B}|_{E_n}$ in the basis $\left\lbrace \frac{1}{\sqrt{2\pi}} e^{inx}\right\rbrace $ of $E_n$ and $\left\lbrace \xi_{n,1},\xi_{n,2}, \xi_{n,3}\right\rbrace $ of $\mathbf{Z}_n$ is
					\begin{equation}\label{B_n}
						\mathcal{B}_n= b\sqrt{2\pi}\begin{pmatrix}
							\frac{1}{{\bar\psi_{n,1}}}\\\frac{1}{{\bar\psi_{n,2}}}\\\frac{1}{{\bar\psi_{n,3}}}
						\end{pmatrix},
					\end{equation}
					where $\psi_{n,\ell}$, for $n\in\mathbb{Z}^*$ and $\ell=1,2,3$ is defined in \eqref{nmaxequ-psi_nl}. 
					For $n=0$, $\mathcal{B}_0=\sqrt{b}$.	
					
					For all, $n\in\mathbb{Z}$, the adjoint of $\pi_n\mathcal{B}|_{E_n}$ is defined by $\mathcal{B}^*\pi^*_n|_{\mathbf{Z}_n^*}\in \mathcal{L}(\mathbf{Z}_n^*, E_n)$ with the matrix representation $\mathcal{B}^*_n$ in the basis  of $\mathbf{Z}^*_n$ and $E_n$.

					Furthermore, there exists a positive constant $C_B>0$ such that 
					\begin{equation}\label{equniformbddcntrl}
						\|\pi_n\mathcal{B}\|_{\mathcal{L}(E_n, \mathbf{Z}_n)}\le C_B, \quad \|(\pi_n\mathcal{B})^*\|_{\mathcal{L}(\mathbf{Z}_n^*, E_n)}\le C_B, \forall\, n\in\mathbb{Z}.
					\end{equation}
				\end{lem}
				\begin{proof}
					Since $\mathcal{B}g\in\mathcal{Z}_m$, from \eqref{def-pi}, we have
					\begin{equation}\label{eqexx}
						\pi_n\mathcal{B}g = \sum_{l=1}^{3}\left\langle \begin{pmatrix}
							g\\0\\0
						\end{pmatrix}, \xi_{n,l}^*\right\rangle_{\mathcal{Z}}\xi_{n,l},  \quad  \forall\, n\in \mathbb{Z}^*, \quad 
						\pi_0\mathcal{B}g = \left\langle \begin{pmatrix}g\\0\\0 \end{pmatrix}, \xi_{0}^*\right\rangle_{\mathcal{Z}}\xi_{0}. 
					\end{equation}
					Now, using \eqref{f=f_n} and the fact that 
					\begin{equation*}
						\left\langle \begin{pmatrix}
							g_k\\0\\0
						\end{pmatrix}, \xi_{n,l}^*\right\rangle_{\mathcal{Z}}= 0\text{ for }k\neq n,\quad \left\langle \begin{pmatrix}
							g_k\\0\\0
						\end{pmatrix}, \xi_{0}^*\right\rangle_{\mathcal{Z}}= 0\text{ for }k\neq 0,
					\end{equation*}
					from \eqref{eqexx}, \eqref{pi_nf=pi_nf_n} follows.

					Now note that $\pi_n\mathcal{B}|_{E_n}\in \mathcal{L}(E_n, \mathbf{Z}_n)$ and for all $x\in (0,2\pi)$, $g_n(x)=\frac{1}{\sqrt{2\pi}} e^{inx}$, for $n\in\mathbb{Z}^*$ and 
					$g_0(x)=\frac{1}{\sqrt{2\pi}}$, we get 
					$$
					\pi_n\mathcal{B}g_n =b\sqrt{{2\pi}}\sum\limits_{l=1}^{3}\frac{1}{{\bar\psi_{n,l}}}\xi_{n,l}, \text{ for all} \, n\in\mathbb{Z}^*, 
					\quad 
					\pi_0\mathcal{B}g_0 = \sqrt{b}\xi_0.
					$$
					Thus, the representation $\mathcal{B}_n$ for all $n\in\mathbb{Z}$ holds.  Using \eqref{nmaxcgs-theta_nl} and above representation, \eqref{equniformbddcntrl} can be proved. 
				\end{proof}
				For each $n\in\mathbb{Z}$, projecting \eqref{opp-eqn} on $\mathbf{Z}_n$ and considering control from $E_n$,  we obtain the finite dimensional system 
				\begin{equation} \label{opp-eqn_proj}
					\dot{z}_n(t) = \pi_n\mathcal{A} z_n(t) + \pi_n\mathcal{B}f_n(t), \quad t\in (0,T), \qquad z_n(0) =z_{0,n},
				\end{equation}
				where $\dot{z}_n(t)=\pi_n\dot{z}(t), \, z_n(t)=\pi_n z(t), \, z_{0,n}=\pi_n z_0$.

			}
			\section{Interior controllability}\label{nmaxsecnullcont}
			This section is devoted to the exact controllability of the system \eqref{nmaxeq3} by means of localized interior control acting in any of the equations. Our proof relies on the duality between the controllability of the system  and the observability of the corresponding adjoint system. Here we only prove the exact controllability result when the control acts only in the density equation. For the other two cases, the proofs follow in a similar fashion.
			
			We recall the final state observability of  $(\mathcal{A}^{*}, \mathcal{B}^{*}):$
			For $(\mathcal{A}^*, \mathcal{D}(\mathcal{A}^*; \mathcal {{Z}}_m))$ defined in \eqref{nmaxdom-A*}-\eqref{nmaxop-A*} and $(\sigma_{T}, v_{T}, \tilde{S}_T )^{\top} \in \mathcal {{Z}}_m,$ we set 
			\begin{equation*}
				(\sigma(t), v(t), \tilde{S}(t))^{\top} = \mathbb{T}_{T-t}^{*} (\sigma_{T}, v_{T}, \tilde{S}_T)^{\top} \qquad (t\in [0,T]),
			\end{equation*}
			where $\mathbb{T}^*$ is the $C^{0}$-semigroup generated by $(\mathcal{A}^*, \mathcal{D}(\mathcal{A}^*; \mathcal {{Z}}_m))$ on $\mathcal {{Z}}_m$. 
			In view of \Cref{nmaxadj-op}, $(\sigma, v, \tilde{S})^{\top}$ belongs to $C([0,T];\mathcal {{Z}}_m)$ and satisfies:
			\begin{equation}\label{nmaxeqadj}
				\begin{cases}
					\partial_t\sigma +u_s\pa_x\sigma+{\rho_s}\pa_x v =0, & \mbox{ in } (0,T) \times (0, 2\pi), \\
					\partial_t v+b \partial_{x} \sigma+u_s\pa_x v - \frac{1}{\rho_s} \partial_x\tilde{S}= 0, &  \mbox{ in } (0,T) \times (0, 2\pi),\\
					\partial_t\tilde{S}-\frac{1}{\kappa}\tilde{S} - \frac{\mu}{\kappa} \partial_x v= 0, &  \mbox{ in } (0,T) \times (0, 2\pi),\\
					\sigma(t,0) = \sigma(t, 2\pi), v(t,0) = v(t, 2\pi),\: \tilde S(t,0)=\tilde S(t,2\pi),&  \mbox{ in } (0,T), \\
					\sigma(T,x)=\sigma_{T}(x), \quad v(T,x)=v_T(x), \quad \tilde{S}(T,x)=\tilde{S}_T(x), & \mbox{ in } (0,2\pi).
				\end{cases}
			\end{equation} 
			Here, at first we discuss a standard approach to deduce the observability inequality for the adjoint system \eqref{nmaxeqadj}  which essentially gives the exact controllability of the main system \eqref{nmaxeq3}. Let us consider the linear map $\mathcal{F}_T: L^2(0,T, L^2(0,2\pi)\to (L^2(0,2\pi))^3$ by $\mathcal{F}_T(f_1)= \left(\rho(T,\cdot),u(T, \cdot), S(T, \cdot)\right)^{\top}$ where $(\rho, u, S)^{\top}$ is the solution of the system \eqref{nmaxeq3} with $(\rho_0, u_0, S_0)^{\top}=(0,0,0)^{\top}$ and  $f_2=f_3=0$. It is clear that exact controllability of \eqref{nmaxeq3} is equivalent to the surjectivity of the linear map $\mathcal{F}_T$. Note that the map $\mathcal{F}_T$ is surjective if and only if there exists a constant $C>0$ such that the following inequality holds:
			\begin{align}\label{eq:fstar}
				\norm{\mathcal{F}_T^*  {z}}_{L^2(0,T, L^2(0,2\pi)}\geq C\norm{ z}_{(L^2(0,2\pi))^3}, \, \text{ for all }  z \in (L^2(0,2\pi))^3.
			\end{align} A direct computation shows that
			$\mathcal{F}_T^*((\sigma_{T}, v_{T}, \tilde{S}_T)^{\top} )= \mathbbm{1}_{\mathcal{O}_{1}} \sigma$,  where $(\sigma,v, \tl S)$ is the solution of the adjoint system \eqref{nmaxeqadj} with the terminal data $(\sigma_{T}, v_{T}, \tilde{S}_T)$. Thus the exact controllability of the system \eqref{nmaxeq3} is equivalent to the following observability inequality: 
			
			\begin{prop}\label{nmaxthobs1}
				Let $T>0$. 
					Assume $f_2=0=f_3$. 	
					Then the system \eqref{nmaxeq3} is exactly controllable in $\mathcal {{Z}}_m$ at time $T>0$ using a control $f_1$ in $L^2(0,T;L^2(0,2\pi))$ with support in $\mathcal{O}_1$ acting only in the density equation, if and only if, there exists a positive constant $C_T>0$  such that for any $(\sigma_{T}, v_{T}, \tilde{S}_T)^{\top}\in \mathcal {{Z}}_m$, 
					$(\sigma, v, \tilde{S})^{\top}$, the solution of \eqref{nmaxeqadj}, satisfies the following observability inequality:
					\begin{multline}\label{nmaxobsden}
						\int_{0}^{2\pi} |\sigma_T(x)|^{2} \ \rd x \; + \; \int_{0}^{2\pi} |v_T(x)|^{2} \ \rd x + \; \int_{0}^{2\pi} |\tilde{S}_T(x)|^{2} \ \rd x 
						\leqslant C_T  \int_0^T\int_{\mathcal{O}_1}|\sigma(t,x)|^2\, \rd x\,\rd t. 
					\end{multline}
				\end{prop}
				Next, we prove \Cref{nmaxthm_pos} by showing the observability inequality \eqref{nmaxobsden} using the following Ingham-type inequality.
				\begin{prop}\label{propI-4}
					Let $ T> {2\pi}\left(\frac{1}{|\beta_1|}+\frac{1}{|\beta_2|}+\frac{1}{|\beta_3|}\right)$ and $M\in \mathbb{N}$.
					Then there exist positive constants $C$ and $C_1$ depending on $T$ such that for
					$ g(t)= \displaystyle\sum_{|n|\ge M}\sum_{l=1}^{3}{a_n^l e^{\overline{\lambda_n^l} (T-t)}}$ with
					$\displaystyle{\sum_{|n|\ge M}}\sum_{l=1}^{3}|a_n^l|^2<\infty$, the following inequality holds:
					\begin{equation}\label{I-4}
						C\sum_{|n|\ge M}\sum_{l=1}^{3}|a_n^l|^2\leq \int_0^T|g(t)|^2 \, \rd t\leq C_1\sum_{|n|\ge M}\sum_{l=1}^{3}|a_n^l|^2.
					\end{equation}
				\end{prop}
				\begin{proof}
					The proof follows trivially from \eqref{ingham2} of  \Cref{biorthogonal} by suitable choices of $\{a_n^l\}$ and a change of variable.
				\end{proof}
				\subsection{Proof of \Cref{nmaxthm_pos}}
					Let $(\sigma_{T}, v_{T}, \tilde{S}_T )^\top \in \mathcal {{Z}}_m.$ From \Cref{nmaxbasis_representation}, we have
					\begin{equation}\label{nmaxest_alphanl}
						\begin{pmatrix}
							\sigma_{T}\\v_{T}\\\tilde{S}_{T}
						\end{pmatrix}=\sum\limits_{n\in\mathbb{Z}^*}\sum\limits_{l=1}^{3}c_{n,l} \xi^*_{n,l} + c_0\xi_0^*, \quad \mathrm{with} \quad 
						\sum\limits_{n\in\mathbb{Z}^*}\sum\limits_{l=1}^{3}\left| c_{n,l}\right|^2 + \left|c_0 \right|^2  < \infty .
					\end{equation}
					
					Then the corresponding solution $(\sigma(t), v(t), \tilde{S}(t))^\top$ of the adjoint problem \eqref{nmaxeqadj} can be written as
					\begin{equation}\label{nmaxrepressoln}
						\begin{pmatrix}
							\sigma\\v\\\tilde{S}
						\end{pmatrix}(t,x)=\sum\limits_{n\in\mathbb{Z}^*}\sum\limits_{l=1}^{3}c_{n,l}e^{\overline{\lambda_n^l}(T-t)} \xi^*_{n,l}(x) + c_0e^{\overline{\lambda_0}(T- t)}\xi_0^*(x).
					\end{equation}
					In particular, using the expression of $\xi^*_0$ and $\xi^*_{n,l}$ ( eq. \eqref{nmaxequ-xi_0^*}, \eqref{nmaxequ-xi_n^*} and \eqref{nmaxcoeffofxi}), from \eqref{nmaxrepressoln} we get
					\begin{align}
						&\sigma(t,x)=\sum\limits_{n\in\mathbb{Z}^*}^{\infty}\sum\limits_{l=1}^{3}\frac{c_{n,l}}{\psi_{n,l}}e^{\overline{\lambda_n^l}(T-t)}e^{inx} + \frac{c_0}{\sqrt{2b\pi}}, \quad \forall\, t\in (0,T), \quad x\in (0,2\pi),\label{nmaxrepsigma}\\
						&v(t,x)=\sum\limits_{n\in\mathbb{Z}^*}^{\infty}\sum\limits_{l=1}^{3}\frac{c_{n,l}}{\psi_{n,l}}\alpha_{n,l}^2e^{\overline{\lambda_n^l}(T-t)}e^{inx}, \quad \forall\, t\in (0,T), \quad x\in (0,2\pi),\\
						&\tilde S(t,x)=\sum\limits_{n\in\mathbb{Z}^*}^{\infty}\sum\limits_{l=1}^{3}\frac{c_{n,l}}{\psi_{n,l}}\alpha_{n,l}^3 e^{\overline{\lambda_n^l}(T-t)}e^{inx}, \quad \forall\, t\in (0,T), \quad x\in (0,2\pi).
					\end{align} 
					
					Since $\left\lbrace e^{inx}, x\in (0,2\pi)\right\rbrace_{n\in\mathbb{Z}} $ is an orthonormal basis in $L^2\left( 0,2\pi\right) $ and $\{ e^{inx},$  $ x\in (0,2\pi)\}_{n\in\mathbb{Z}^*} $ is an orthonormal basis in ${\dot L^2\left( 0,2\pi\right) }$, then using Parseval's identity, we get
					\begin{align}\label{nmaxobsleft1}
						&\int_{0}^{2\pi} |\sigma^T(x)|^{2} \ \rd x \; + \; \int_{0}^{2\pi} |v^T(x)|^{2} \ \rd x + \; \int_{0}^{2\pi} |\tilde{S}^T(x)|^{2} \ \rd x\notag\\
						=&\sum\limits_{n\in\mathbb{Z}^*}^{\infty}\left(\left| \sum\limits_{l=1}^{3}\frac{c_{n,l}}{\psi_{n,l}}+ \frac{c_0}{\sqrt{2b\pi}}\right|^2+\left| \sum\limits_{l=1}^{3}\frac{c_{n,l}}{\psi_{n,l}}\alpha_{n,l}^2\right|^2+\left| \sum\limits_{l=1}^{3}\frac{c_{n,l}}{\psi_{n,l}}\alpha_{n,l}^3 \right|^2 \right). 
					\end{align} 
					Now using \eqref{nmaxcgsalpha2} and \eqref{nmaxcgsalpha3}, we get a positive constant $C$, independent of $n$ and a large $N\in \N$, such that
					\begin{align}\label{nmaxobsleft2}
						&\sum\limits_{|n|>N}\left(\left| \sum\limits_{l=1}^{3}\frac{c_{n,l}}{\psi_{n,l}}+ \frac{c_0}{\sqrt{2b\pi}}\right|^2+\left| \sum\limits_{l=1}^{3}\frac{c_{n,l}}{\psi_{n,l}}\alpha_{n,l}^2\right|^2+\left| \sum\limits_{l=1}^{3}\frac{c_{n,l}}{\psi_{n,l}}\alpha_{n,l}^3\right|^2 \right)\notag\\
						\leq & C \left( \sum\limits_{|n|>N}\sum\limits_{l=1}^{3}\left|\frac{c_{n,l}}{\psi_{n,l}} \right|^2 + \left|\frac{c_0}{\sqrt{2b\pi}} \right|^2\right) .
					\end{align}
					Let us define 
					\begin{equation*}
						a_n^l(x)=\frac{c_{n,l}}{\psi_{n,l}} e^{inx}, \text{ for } |n|>N, l=1,2,3,\text{ and } x\in (0, 2\pi).
					\end{equation*}
					Using \eqref{nmaxcgs-theta_nl} and \eqref{nmaxest_alphanl}, we have $\displaystyle{\sum_{|n|\ge N}}\sum_{l=1}^{3}|a_n^l(x)|^2<\infty$.  Thanks to \Cref{propI-4}, for $T> {2\pi}\left(\frac{1}{|\beta_1|}+\frac{1}{|\beta_2|}+\frac{1}{|\beta_3|}\right)$, we have 
					\begin{align*}
						\sum\limits_{|n|> N} \sum\limits_{l=1}^{3}\left| \frac{c_{n,l}}{\psi_{n,l}} e^{inx}\right|^2 \leq C \int_{0}^{T}\left|\sum\limits_{\left| n\right| > N}\sum_{l=1}^{3}\frac{c_{n,l}}{\psi_{n,l}}e^{inx}e^{\overline{\lambda_n^l} (T-t)}\right|^2 \, \rd t .
					\end{align*}
					
					\noindent
					Integrating both sides over $\mathcal{O}_1$, we get 
					\begin{align}\label{nmaxobsin1}
						\sum\limits_{|n|> N} \sum\limits_{l=1}^{3}\left| \frac{c_{n,l}}{\psi_{n,l}}\right|^2\leq C \int_{0}^{T}\int_{\mathcal{O}_1}\left|\sum\limits_{\left| n\right| > N}\sum_{l=1}^{3}\frac{c_{n,l}}{\psi_{n,l}}e^{inx}e^{\overline{\lambda_n^l} (T-t)} \right|^2 \, \rd x \, \rd t.  
					\end{align}
					Therefore only finitely many terms are leaving. Since by \ref{nmaxsimpleeigen}, all the eigenvalues are distinct, then the
					missing finitely many exponential ( for $|n| \leq N$) can be added one by one in the inequality \eqref{nmaxobsin1} as the required
					gap condition for the eigenvalues holds. For similar details one can see \cite[Chapter 4,
					Theorem 4.3 ]{micu2004introduction}.
					
					Thus for $T> {2\pi}\left(\frac{1}{|\beta_1|}+\frac{1}{|\beta_2|}+\frac{1}{|\beta_3|}\right)$, we get
					\begin{align}\label{nmaxobsin2}
						\sum\limits_{n\in\mathbb{Z}^*} \sum\limits_{l=1}^{3}\left| \frac{c_{n,l}}{\psi_{n,l}}\right|^2+ \left| \frac{c_0}{\sqrt{2b\pi}}\right|^2&\leq C \int_{0}^{T}\int_{\mathcal{O}_1}\left|\sum\limits_{n\in\mathbb{Z}^*}\sum_{l=1}^{3}\frac{c_{n,l}}{\psi_{n,l}}e^{inx}e^{\overline{\lambda_n^l} (T-t)}+\frac{c_0}{\sqrt{2b\pi}} \right|^2 \rd x \, \rd t\notag\\
						& = C \int_{0}^{T}\int_{\mathcal{O}_1}\left|\sigma(t,x)  \right|^2 \rd x \,\rd t\: (\text{ using }\eqref{nmaxrepsigma}).
					\end{align}
					Therefore using \eqref{nmaxobsleft1}, \eqref{nmaxobsleft2} and \eqref{nmaxobsin2}, we obtain \eqref{nmaxobsden}. Then by \Cref{nmaxthobs1}, we conclude \Cref{nmaxthm_pos}.\qed
					\begin{rem}
						In the proof of \Cref{nmaxthm_pos}, i.e, in the proof of the observability inequality \eqref{nmaxobsin2}, we assume that all the eigenvalues of $\mc A$ are simple to get the Ingham-type inequality for full class of exponentials $\{e^{-\overline{\lambda_n^l}t}: n\in \z^*, l\in\{1,2,3\}\}.$ In \Cref{rem multi}, we have mentioned that the spectrum of $\mc A$ may have finite number of multiple eigenvalues. In this case also, one can prove the observability inequality with a slight modification of the technique for simple eigenvalue case. Details analysis can be found in \cite[Section 4.2]{chowdhury2014null}. One can also refer to the book \cite[Remarks, Page 178]{KV} for a version of Ingham inequality with repeated eigenvalues. In the next section, we discuss a brief description of the proof of concerned observability inequality.
					\end{rem}
					{
						\subsubsection{The case of multiple eigenvalues}\label{sec_multi_ev_b}
						In this section, we prove null controllability of the system \eqref{nmaxeq3} with $f_2=f_3=0$ in the presence of multiple eigenvalues. Our main goal is to prove the observability inequality \eqref{nmaxobsden}. The proof will be similar for all cases (control acting in density or velocity or stress), so we present a detailed proof for the density case only.
						The proof is inspired from \cite[Section 4.4]{KV} and \cite[Section 4.2]{chowdhury2014null}. 
						
						Let us assume (without loss of generality) that $\overline{\mu}_n$ be the eigenvalues of $\mc A^*$ with multiplicity $N_n$ for $n=1,2,\dots,m$ and for all $|{n}|>m$, the eigenvalues $\overline{\mu}_n$ of $\mc A^*$ are simple. We denote the set of generalized eigenfunctions corresponding to $\overline{\mu}_n$ (for $n=1,2,\dots,m$) as $$\left\{{\xi}_{n}^{*1}, \, \tilde{\xi}_{n}^{*j}, \, j=1,2\dots,N_n-1\right\}.$$
						
						Let $(\sigma_{T}, v_{T}, \tilde{S}_T )^\top \in \mathcal {{Z}}_m.$ We decompose it as
						\begin{equation}
							(\sigma_T,v_T, \tilde{S}_T)^{^\top}=(\sigma_{T,1},v_{T,1}, \tilde{S}_{T,1})^{^\top}+(\sigma_{T,2},v_{T,2},\tilde{S}_{T,2})^{^\top},
						\end{equation}
						where
						\begin{equation*}
							(\sigma_{T,1},v_{T,1},\tilde{S}_{T,1})^{^\top}=\sum_{n=1}^{m} \left(a_n^1 \xi_{n}^{*1}+\sum_{j=1}^{N_n-1} \tilde{a}_n^{j} \tilde{\xi}_{n}^{*j} \right)
						\end{equation*}
						and
						\begin{equation*}
							(\sigma_{T,2},v_{T,2},\tilde{S}_{T,2})^{\top}=\sum_{|n|>m}\sum\limits_{l=1}^{3}c_{n,l} \xi^*_{n,l}.
						\end{equation*}
						Let $(\sigma_1,v_1, \tilde{S}_{1})^{\top}$ and $(\sigma_2,v_2, \tilde{S}_{2})^{\top}$ be the solutions of the adjoint system \eqref{nmaxeqadj} with the terminal data $(\sigma_{T,1},v_{T,1}, \tilde{S}_{T,1})^{\top}$ and $(\sigma_{T,2},v_{T,2}, \tilde{S}_{T,2})^{\top}$ respectively. Then, we have
						\begin{align}
							(\sigma_1,v_1,\tilde{S}_{1})^{\top}&=\sum_{n=1}^{m}e^{\overline{\mu}_n(T-t)}\left(a_n^1 \xi_{n}^{*1}+\sum_{j=1}^{N_n-1} (T-t)^j \tilde{a}_n^{j} \tilde{\xi}_{n}^{*j} \right),\\
							(\sigma_2,v_2,\tilde{S}_{2})^{\top}&=\sum_{\mod{n}>m}\sum\limits_{l=1}^{3}c_{n,l}e^{\overline{\lambda_n^l}(T-t)} \xi^*_{n,l}.
						\end{align}
						\begin{lem}
							There exists a constant $C>0$ such that the following holds
							\begin{equation}\label{lmm3.4}
								\int_0^{\epsilon}\int_{\mathcal{O}_1}|\sigma_1(t,x)|^2\, \rd x\,\rd t\geq  c\int_0^T\int_{0}^{2\pi}|\sigma_1(t,x)|^2\, \rd x\,\rd t\geq c	\norm{\left(\sigma_{T,1}, v_{T,1}, \tilde{S}_{T,1}\right)}_{(L^2(0,2\pi))^3}^{2}.
							\end{equation}
						\end{lem}
						\begin{proof}We first prove that, there exists a constant $c>0$ independent of $(\sigma_{T,1}, v_{T,1}, \tilde{S}_{T,1})$ such that the following inequality holds:
							\begin{equation}\label{1}
								\int_0^T\int_{0}^{2\pi}|\sigma_1(t,x)|^2\, \rd x\,\rd t\geq c	\norm{\left(\sigma_{T,1}, v_{T,1}, \tilde{S}_{T,1}\right)}_{(L^2(0,2\pi))^3}^{2}.
							\end{equation}
							Let $\hat{\mathcal Z}$ be the finite dimensional space of solutions to the adjoint equation which is generated by the generalized eigenfunctions of $\mathcal A^*$ corresponding to multiple eigenvalues. Let us define the norms on $\hat{\mc Z}$:
							\begin{align*}
								\norm{(\hat{\sigma}_{T,1},\hat{v}_{T,1}, \hat{\tilde S}_{T,1})^{\top}}_1^2:&=\int_{0}^{\epsilon} \int_{0}^{2\pi}\left|\sum_{n=1}^{m}e^{\overline{\mu}_n(T-t)}\left(a_n^1 \mc B^*\xi_{n}^{*1}+\sum_{j=1}^{N_n-1} \tilde{a}_n^{j}(T-t)^j \mc B^*\tilde{\xi}_{n}^{*j} \right)\right|^2 \rd x\, \rd t,\\
								\norm{(\hat{\sigma}_{T,1},\hat{v}_{T,1},\hat{\tilde S}_{T,1})^{\top}}_2^2:&=\norm{(\hat{\sigma}_1(T),\hat{v}_1(T), \hat{\tilde S}_{1}(T))^{\dagger}}_{({L}^2(0,2\pi))^3}^2,
							\end{align*}
							where $(\hat{\sigma}_1,\hat{v}_1, \hat{\tilde S}_1)^{\top}$ denotes the solution of the adjoint system with terminal data $(\hat{\sigma}_{T,1},\hat{v}_{T,1},\hat{\tilde S}_1)^{\top}\in {\hat{\mc{Z}}}$ and the observation operator $\mathcal{B}^*\in \mathcal{L}\left( \mathcal{Z}_m, L^2\left(0,2\pi \right) \right) $ is given by
							\begin{equation}
								\mathcal{B}^*\begin{pmatrix}
									\hat\sigma_1\\\hat v_1\\\hat{\tilde S}_1
								\end{pmatrix}=b\mathbbm{1}_{\mathcal{O}_{1}}\hat{\sigma}_1.
							\end{equation}
							In fact, 
							$\norm{(\hat{\sigma}_{T,1},\hat{v}_{T,1},\hat{\tilde S}_1)^{\top}}_1=0$ implies $\mc{B}^*\xi_{n}^{*1}=\mc{B}^*\tilde{\xi}_{n}^{*j}=0$, $j=1,2,N_n-1, n=1,2,...,m$ as $\{e^{\overline{\mu}_n(T-t)}, (T-t)^j e^{\overline{\mu}_n(T-t)}\}$ are linearly independent. This gives $\xi_{n}^{*1}=\tilde{\xi}_{n}^{*j}=0$ (thanks to \Cref{secmultiev}) and hence $(\hat{\sigma}_{T,1},\hat{v}_{T,1},\hat{\tilde S}_1)=(0,0,0)$. Clearly $\norm{.}_{2}$ is a norm on $\hat{\mc Z}$. 
							
							\noindent Since any norm in finite dimensional space is equivalent, we have the inequality \eqref{1}. 
							We conclude the final inequality by noting the fact that all generalized eigenfunctions are analytic in $x$ (combination of trigonometric functions) and thus $\sigma_1$ on $(0,2\pi)$ uniquely determined by its restriction to $\mathcal{O}_1$.
						\end{proof}
						In the expression of $(\sigma_2,v_2,\tilde{S}_{2})^{\top}$, all eigenvalues are simple, so we have the following observability inequality
						\begin{equation}\label{obser for simple}
							\int_{0}^{2\pi} |\sigma_{T,2}(x)|^{2} \ \rd x \; + \; \int_{0}^{2\pi} |v_{T,2}(x)|^{2} \ \rd x + \; \int_{0}^{2\pi} |\tilde{S}_{T,2}(x)|^{2} \ \rd x 
							\leqslant C_T  \int_0^T\int_{\mathcal{O}_1}|\sigma_2(t,x)|^2\, \rd x\,\rd t.
						\end{equation}
						Next, our goal is to add the finitely many terms associated to the multiple eigenvalues in the above observability estimate. Let us first add the term 
						$e^{\overline{\mu}_1(T-t)}\left(a_1^1 \mc B^*\xi_{1}^{*1}+ (T-t) \tilde{a}_1^{1} \mc B^*\tilde{\xi}_{1}^{*1} \right)$ in the above inequality.
						Denote for each $x\in [0,2\pi],$
						\begin{align}\label{q}
							\mc{Q}(t):=b\mathbbm{1}_{\mathcal{O}_{1}}\sigma_2(t,x) +e^{\overline{\mu}_1(T-t)}\left(a_1^1 \mc B^*\xi_{1}^{*1}+ (T-t) \tilde{a}_1^{1} \mc B^*\tilde{\xi}_{1}^{*1} \right),
						\end{align}
						and
						\begin{align*}
							\mc{R}(t):=\mc{Q}(t)-\frac{1}{2\delta}\int_{-\delta}^{\delta}e^{\overline{\mu}_{1}s}\mc{Q}(t+s)ds
						\end{align*}
						for $t\in(\delta,T-\delta)$ with $\delta>0$ (chosen later accordingly). Then, we have the following estimate (see \cite[Section 4.4]{KV} for details).
						\begin{equation}\label{inq_1}
							\int_{\delta}^{T-\delta}\int_{\mathcal{O}_1}\mod{\mc{R}(t)}^2 \rd x\, \rd t\leq C\int_{0}^{T}\int_{\mathcal{O}_1}\mod{\mc{Q}(t)}^2 \rd x\, \rd t.
						\end{equation}
						We now prove that
						\begin{equation}
							\int_{\delta}^{T-\delta}\int_{\mathcal{O}_1}\mod{\mc{R}(t)}^2 \rd x\, \rd t\geq C\norm{(\sigma_2(T),v_2(T), \tilde{S}_2(T))^{\top}}_{(\dot{L}^2(0,2\pi))^2}^2.
						\end{equation}
						From the expression of $\mc{Q}(t)$, we can get
						\begin{align*}
							\mc{R}(t)&=\sum_{\mod{n}>m}\sum\limits_{l=1}^{3}c_{n,l}e^{\overline{\lambda_n^l}(T-t)} \mc B^*\xi^*_{n,l}\left(1-\frac{1}{2\delta}\int_{-\delta}^{\delta}e^{(\overline{\mu}_{1}-\overline{\lambda_n^l})s}\right) \\
							&=\sum_{\mod{n}>m}\sum\limits_{l=1}^{3}c_{n,l}e^{\overline{\lambda_n^l}(T-t)} \mc B^*\xi^*_{n,l}\left(1-\frac{\sinh((\overline{\mu}_{1}-\overline{\lambda_n^l})\delta)}{(\overline{\mu}_{1}-\overline{\lambda_n^l})\delta}\right).
						\end{align*}
						Since $\inf\limits_{\mod{n}>m,l=1,2,3}\mod{\overline{\mu}_{1}-\overline{\lambda_n^l}}>0$, we have (for appropriate $\delta>0$) $\inf\limits_{\mod{n}>m,l=1,2,3}\mod{1-\frac{\sinh((\overline{\mu}_{1}-\overline{\lambda_n^l})\delta)}{(\mu_{1}-\overline{\lambda_n^l})\delta}}>0.$ Since $T>{2\pi}\left(\frac{1}{|\beta_1|}+\frac{1}{|\beta_2|}+\frac{1}{|\beta_3|}\right)$, we can choose $\delta$ small enough such that $T-2\delta>{2\pi}\left(\frac{1}{|\beta_1|}+\frac{1}{|\beta_2|}+\frac{1}{|\beta_3|}\right)$. Applying Ingham-type inequality \eqref{ingham2} (for simple eigenvalues), we obtain
						\begin{align*}
							\int_{\delta}^{T-\delta}\int_{\mathcal{O}_1}\mod{\mc{R}(t)}^2 \rd x\rd t &\geq C\left[\sum_{\mod{n}>m}\sum\limits_{l=1}^{3}\mod{c_{n,l}}^2 
							e^{2\Re(\overline{\lambda_n^l})T} \right]\\
							&\geq C\norm{(\sigma_2(T),v_2(T), {\tilde S_2}(T))^{\top}}_{({L}^2(0,2\pi))^3}^2.
						\end{align*}
						Therefore, using the estimate \eqref{inq_1}, we obtain
						\begin{equation}\label{inq_2}
							\int_{0}^{T}\int_{\mathcal{O}_1}\mod{\mc{Q}(t)}^2 \rd x\, \rd t\geq C\norm{(\sigma_2(T),v_2(T), {\tilde S_2}(T))^{\top}}_{({L}^2(0,2\pi))^3}^2.
						\end{equation}
						Since $T>{2\pi}\left(\frac{1}{|\beta_1|}+\frac{1}{|\beta_2|}+\frac{1}{|\beta_3|}\right)$, we can choose $\epsilon>0$ such that $T-\epsilon>{2\pi}\left(\frac{1}{|\beta_1|}+\frac{1}{|\beta_2|}+\frac{1}{|\beta_3|}\right)$. Therefore we can write
						\begin{equation*}
							\int_{\epsilon}^{T}\int_{\mathcal{O}_1}\mod{\mc{Q}(t)}^2 \rd x\, \rd t\geq C\norm{(\sigma_2(\epsilon),v_2(\epsilon), \tilde S_2(\epsilon))^{\dagger}}_{({L}^2(0,2\pi))^3}^2
						\end{equation*}
						and thus
						\begin{equation}\label{inq_2a}
							\int_{0}^{T}\int_{\mathcal{O}_1}\mod{\mc{Q}(t)}^2 \rd x\, \rd t\geq C\int_{\epsilon}^{T}\int_{\mathcal{O}_1}\mod{\mc{Q}(t)}^2dt\geq C\norm{(\sigma_2(\epsilon),v_2(\epsilon), \tilde S_2(\epsilon))^{\dagger}}_{({L}^2(0,2\pi))^3}^2.
						\end{equation}
						Thanks to the well-posedness of the adjoint system \eqref{nmaxeqadj}, we have
						\begin{equation}\label{well_posed_inq_den_b}
							\int_{0}^{\epsilon}\int_{\mathcal{O}_1} |\sigma_2(t,x)|^2 \rd x \, \rd t\leq C\norm{(\sigma_2(\epsilon),v_2(\epsilon), \tilde S_2(\epsilon))^{\dagger}}_{({L}^2(0,2\pi))^3}^2.
						\end{equation}
						From equations \eqref{inq_2a} and \eqref{well_posed_inq_den_b}, we deduce that
						\begin{equation*}
							\int_{0}^{T}\int_{\mathcal{O}_1}\mod{\mc{Q}(t)}^2 \rd x\, \rd t\geq C	\int_{0}^{\epsilon}\int_{\mathcal{O}_1} |\sigma_2(t,x)|^2 \rd x \, \rd t.
						\end{equation*}
						Using the above inequality along with \eqref{q}, we obtain
						\begin{align}\label{inq_3}
							\nonumber	\int_{0}^{\epsilon}\int_{\mathcal{O}_1}\mod{e^{\overline{\mu}_{1}(T-t)}\left(a_{n_0}\mc{B}^*\xi_{1}^{*1}+(T-t)\tilde{a}_{n_0}\mc{B}^*\tilde\xi_{1}^{*1}\right)}^2 \rd x\, \rd t&\leq C\int_{0}^{\epsilon}\int_{\mathcal{O}_1}\mod{\mc{Q}(t)}^2 \rd x\, \rd t+C\int_{0}^{\epsilon}\int_{\mathcal{O}_1}|\sigma_2(t,x)|^2 \rd x\, \rd t.\\
							&\leq C\int_{0}^{T}\int_{\mathcal{O}_1}\mod{\mc{Q}(t)}^2 \rd x\,\rd t.
						\end{align}
						Proceeding in a similar way, one can add remaining terms and obtain the following estimates from the inequalities \eqref{inq_3} and \eqref{inq_2} respectively
						\begin{align}\label{ep}
							\int_0^{\epsilon}\int_{\mathcal{O}_1}|\sigma_1(t,x)|^2\, \rd x\,\rd t\leq C \int_0^{T}\int_{\mathcal{O}_1}|\sigma(t,x)|^2,
						\end{align}
						\begin{equation}\label{sg2}
							\int_{0}^{T}\int_{\mathcal{O}_1}\mod{\sigma(t,x)}^2 \rd x\, \rd t\geq C\norm{(\sigma_2(T),v_2(T), {\tilde S}_2(T))^{\top}}_{({L}^2(0,2\pi))^3}^2.
						\end{equation}
						Finally, amalgamating \eqref{lmm3.4}, \eqref{ep}, \eqref{sg2}, we have
						\begin{equation}
							\int_{0}^{T}\int_{\mathcal{O}_1}\mod{\sigma(t,x)}^2 \rd x\, \rd t\geq C\norm{(\sigma_{T,2},v_{T,2}, {\tilde S}_{T,2})^{\top}}_{({L}^2(0,2\pi))^3}^2+\norm{\left(\sigma_{T,1}, v_{T,1}, \tilde{S}_{T,1}\right)}_{(L^2(0,2\pi))^3}^{2}.
						\end{equation}
						This completes the proof of desired observability inequality \eqref{nmaxobsden}.
					}
					%

					{
						\subsection{Control acting everywhere in the domain}
						Now, we proceed to prove \Cref{thm_pos}. Since, exact and null controllability are equivalent for our system, we are going to prove the null controllability for the sake of computational simplicity. To do this, we first establish the null controllability of each finite-dimensional projected system \eqref{opp-eqn_proj}, and then utilize this to infer the null controllability of the system \eqref{opp-eqn}.
						
						Recall that \eqref{opp-eqn_proj} is defined on $\mathbf{Z}_n$ with control from $E_n$, where $\mathbf{Z}_n$ and $E_n$ are finite-dimensional spaces defined in \eqref{def-Zn} and \eqref{def-E_n}, respectively. Referring back to \Cref{lem-An} and \Cref{lem-Bn}, the matrix representation of $\pi_n\mathcal{A}|_{\mathbf{Z}_n}$ and $\pi_n\mathcal{B}|_{E_n}$ is denoted as $\mathcal{A}_n$ and $\mathcal{B}_n$, respectively. To assess the controllability of \eqref{opp-eqn_proj}, we apply the ``Hautus Lemma" (refer to \cite[Theorem 2.5, Section 2, Chapter 1, Part I]{bensoussan2007representation}) to the pair $(\mathcal{A}_n, \mathcal{B}_n)$.
						\begin{theorem}\label{th_finitecontrol}
							For each $n\in\mathbb{Z}$, the finite dimensional system \eqref{opp-eqn_proj} is controllable at any given $T>0$.
							For each $n\in\mathbb{Z}$, the control $f_n\in L^2(0,T; E_n)$ that brings the solution of \eqref{opp-eqn_proj}  at rest at a given time $T>0$ is given by 
							\begin{equation}\label{equ-mincontrol}
								f_n(t)= - \mathcal{B}^*_ne^{(T-t)\mathcal{A}_n^*}W^{-1}_{n,T}e^{T\mathcal{A}_n} z_{0,n},\qquad t\in[0,T],
							\end{equation}
							where $W_{n,T}\in \mathcal{L}(\mathbf{Z}^*_n, \mathbf{Z}_n)$ is the controllability operator given by:
							\begin{equation}\label{eqgram}
								W_{n,T}= \int_{0}^{T}e^{t\pi_n\mathcal{A}}\Big(\pi_n\mathcal{B}\Big)\Big(\pi_n\mathcal{B}\Big)^*e^{t(\pi_n\mathcal{A})^*}\:\rd t. 
							\end{equation}
						\end{theorem}
						\begin{proof}
							Let $n\in \mathbb{Z}^*$ be fixed. We prove this result showing $(iii)$ of \cite[Theorem 2.5, Section 2, Chapter 1, Part I]{bensoussan2007representation}). We show that for $n\in \N$
							the $3\times 3$ matrix $\mathcal{A}_n$ and $3\times 1$ matrix $\mathcal{B}_n$ satisfies 
							\begin{align*}
								\text{Rank}\left[\lambda I-\mathcal{A}_n \;;\; \mathcal{B}_n \right]=3,\text{ for all $\lambda\in\mathbb{C}$ that are eigenvalues of }\mathcal{A}_n, 
							\end{align*}
							and for $n=0$, the row vector $\left[\lambda_0 I-\mathcal{A}_0 \;;\; \mathcal{B}_0 \right]$ has rank $1$. 
							
							From direct calculation, we have 
							\begin{equation}
								\left[\lambda^\ell I-\mathcal{A}_n \;;\; \mathcal{B}_n \right]=\begin{bmatrix}
									\lambda^\ell-\lambda_n^1&0&0&b\sqrt{2\pi}\frac{1}{\bar\psi_{n,1}}\\
									0&\lambda^\ell-\lambda_n^2&0&b\sqrt{2\pi}\frac{1}{\bar\psi_{n,2}}\\
									0&0&\lambda^\ell-\lambda_n^3&b\sqrt{2\pi}\frac{1}{\bar\psi_{n,3}}
								\end{bmatrix},
							\end{equation}
							where the eigenvalues of $\mathcal{A}_n$ are $\lambda_n^\ell, \: \ell=1,2,3.$ In particular, for $\lambda^\ell=\lambda_n^1$, note that the $3\times 4$ matrix has $3$ independent columns and thus Rank$\left[\lambda_n^1 I-\mathcal{A}_n \;;\; \mathcal{B}_n \right]= 3$. Similarly we can see that Rank$\left[\lambda_n^\ell I-\mathcal{A}_n \;;\; \mathcal{B}_n \right]= 3$, for $\ell=2,3$. 
							For $n=0$, $\left[\lambda_0 I-\mathcal{A}_0 \;;\; \mathcal{B}_0 \right]= [0, \sqrt{b}]$, which implies Rank$\left[\lambda_0 I-\mathcal{A}_n \;;\; \mathcal{B}_n \right]= 1$. 
							Therefore for each $n\in\mathbb{Z}$, the system \eqref{opp-eqn_proj} is controllable at any time $T$.
							
							Since $(\pi_n\mathcal{A}, \pi_n\mathcal{B})$ is controllable, from \cite[Part I, Chapter 1, Proposition 1.1]{zabczyk2008mathematical}, it follows that the inverse of the controllability operator $W_{n,T}^{-1}$ exists  in $\mathcal{L}(\mathbf{Z}_n, \mathbf{Z}^*_n)$ and the control $f_n$ given in  \eqref{equ-mincontrol} is well-defined and it gives the null controllability of \eqref{opp-eqn_proj} at time $T$. 
						\end{proof}

						\begin{lem}\label{lem-cntrlfin-est}
							For each $n\in \mathbb{Z}$, let $W_{n,T}$ and $f_n\in L^2(0,T; E_n)$ be as defined in \eqref{eqgram} and \eqref{equ-mincontrol}, respectively. 
							Then, there exist positive constants $C_1$ and $C_2$ independent of $n$, such that 
							\begin{equation}\label{bddWnT}
								C_1\le \|W_{n,T}\|_{\mathcal{L}(\mathbf{Z}^*_n, \mathbf{Z}_n)}\le C_2, \quad \forall\, n\in \mathbb{Z},
							\end{equation}
							and 
							\begin{equation}\label{estimatef_n}
								\int_{0}^{T}\left\| f_n(t)\right\|^2_{L^2(0,2\pi)}\:\rd t \leq C \left\|z_{0,n} \right\|^2_{\mathcal{Z}},
							\end{equation}
							for some positive constant $C$ independent of $n$.
						\end{lem}
						\begin{proof}
							Let us observe that the matrix representation of the controllability operator $W_{n,T}\in \mathcal{L}(\mathbf{Z}^*_n, \mathbf{Z}_n)$
							(still denoted by same notation $W_{n,T}$) in the basis $\{\xi^*_{n,\ell}\mid \ell=1,2,3 \}$ of $\mathbf{Z}^*_n$ and $\{\xi_{n,\ell} \mid \ell=1,2,3 \}$ of $\mathbf{Z}_n$ is given by 
							\begin{equation}\label{matWnT}
								W_{n,T}=b^22\pi\begin{pmatrix}
									\frac{e^{2T\text{ Re}\lambda_n^1}-1}{2\text{ Re}\lambda_n^1\left| \psi_{n,1}\right|^2 } &\frac{e^{T\left(\lambda_n^1+\bar{\lambda}_n^2 \right) }-1}{\left(\lambda_n^1+\bar{\lambda}_n^2 \right) \bar{\psi}_{n,1}\psi_{n,2} } & \frac{e^{T\left(\lambda_n^1+\bar{\lambda}_n^3 \right) }-1}{\left(\lambda_n^1+\bar{\lambda}_n^3 \right) \bar{\psi}_{n,1}\psi_{n,3} } \\
									\frac{e^{T\left(\lambda_n^1+{\lambda}_n^2 \right) }-1}{\left(\lambda_n^1+{\lambda}_n^2 \right) \bar{\psi}_{n,2}\psi_{n,1} } &\frac{e^{2T\text{ Re}\lambda_n^2}-1}{2\text{ Re}\lambda_n^2\left| \psi_{n,2}\right|^2 }&\frac{e^{T\left(\lambda_n^2+\bar{\lambda}_n^3 \right) }-1}{\left(\lambda_n^2+\bar{\lambda}_n^3 \right) \bar{\psi}_{n,2}\psi_{n,3} } \\
									\frac{e^{T\left(\lambda_n^1+{\lambda}_n^3 \right) }-1}{\left(\lambda_n^1+{\lambda}_n^3 \right) \bar{\psi}_{n,3}\psi_{n,1} } &\frac{e^{T\left(\lambda_n^3+\bar{\lambda}_n^2 \right) }-1}{\left(\lambda_n^3+\bar{\lambda}_n^2 \right) \bar{\psi}_{n,3}\psi_{n,2} } &\frac{e^{2T\text{ Re}\lambda_n^3}-1}{2\text{ Re}\lambda_n^3\left| \psi_{n,3}\right|^2 }
								\end{pmatrix},
							\end{equation}
							where $W_{n,T}^{ij}$, for $i, j\in \{1,2,3\}$, is the $(ij)$th element of the above matrix \eqref{matWnT} satisfying 
							\begin{align}
								&\left| W_{n,T}^{ll}\right|\rightarrow \frac{\left( 1-e^{-2T\omega_l}\right) }{4\pi\omega_l\left( b+\frac{\left( \beta_l+u_s\right)^2 }{\rho_s}+\frac{\mu\left( \beta_l+u_s\right)^2 }{\kappa\rho_s^2 \beta_l^2}\right)  },\quad  \text{ for }l\in \{1,2,3\},\label{estWnii}\\
								&\text{ and }\left| W_{n,T}^{ij}\right|\rightarrow 0,\quad\text{ for }i\neq j,\: i,j\in \{1,2,3\},\label{estWnij} \quad\text{as }|n|\rightarrow \infty,
							\end{align}
							due to \eqref{nmaxasmlambda} and \eqref{nmaxcgs-theta_nl}. 
							Noting that 
							$$ M_1 \Big(\sum_{i=1}^3\sum_{j=1}^3|W_{n,T}^{ij}|^2\Big)^{\frac{1}{2}} \le \left\|W_{n,T} \right\|_{\mathcal{L}\left(\mathbf{Z}^*_n,\mathbf{Z}_n \right)} \le M_2\Big(\sum_{i=1}^3\sum_{j=1}^3|W_{n,T}^{ij}|^2\Big)^{\frac{1}{2}}, \quad \forall\, n\in \mathbb{Z}^*,$$
							for some positive constants $M_1$ and $M_2$ independent of $n$, from \eqref{estWnii}-\eqref{estWnij}, \eqref{bddWnT} follows. 
							Finally, from \eqref{equniformfinsemigr}, \eqref{equniformbddcntrl}, \eqref{equ-mincontrol} and \eqref{bddWnT}, \eqref{estimatef_n} follows. 
						\end{proof}
						Now, we are in position to prove \Cref{thm_pos}.
						\subsubsection{\textbf{Proof of \Cref{thm_pos} : }}
						To establish \Cref{thm_pos}, we aim to demonstrate that for any $T>0$ and for any $z_0\in \mathcal{Z}_m$, there exists a control $f\in L^2(0,T; L^2(0,2\pi))$ such that the solution $z$ of \eqref{opp-eqn} satisfies $z(T)=0$ in $\mathcal{Z}_m$.
						
						Given $z_{0} \in {\mathcal{Z}_m}$, for each $n\in \mathbb{Z}$, the control $f_n\in L^2(0,T; E_n)$ defined in \eqref{equ-mincontrol} brings the solution of the finite-dimensional system \eqref{opp-eqn_proj} with initial condition $z_{0,n}$ to rest at time $T$. Additionally, the control $f_n$ satisfies \eqref{estimatef_n}. 
						
						Now, let us set $f=\sum\limits_{n\in\mathbb{Z}}f_n$. Utilizing \eqref{estimatef_n} and \eqref{nmaxequ-est-z_0}, we obtain $f\in L^2(0,T; L^2(0,2\pi))$ with:
						\begin{align*}
							\sum_{n\in\mathbb{Z}}\left\|  f_n\right\|^2_{L^2(0,T;L^2(0,2\pi))} \leq C\sum_{n\in\mathbb{Z}}\left\|z_{0,n} \right\|^2_{\mathcal{Z}}
							\le C \left(\left| \left\langle z_0, \xi^*_{0} \right\rangle\right|^2+ \sum_{n\in\mathbb{Z}^*}\sum_{l=1}^{3}\left| \left\langle z_0, \xi^*_{n,l} \right\rangle\right|^2 \right)
							\leq C\left\|z_0 \right\|_{\mathcal{Z}_m}^2,
						\end{align*}
						where $C$ is a generic positive constant.
						With this control $f$ and the initial condition $z_0$, \eqref{opp-eqn} admits a unique solution $z\in C([0,T]; \mathcal{Z}_m)$, which can be decomposed as $z(t)=\sum_{n\in\mathbb{Z}} z_n(t)$, where $z_n(t)\in \mathbf{Z}_n$ satisfies \eqref{opp-eqn_proj}. By \Cref{th_finitecontrol}, it follows that $z(T)=0$ in $\mathcal{Z}_m$.
						\qed

					}

					\section{Boundary controllability}\label{nmax-bd}
					In this section, we will discuss the boundary controllability  of the compressible Navier–Stokes system with Maxwell's law \eqref{nmaxeq3 bd} with any of the boundary conditions \eqref{density}, \eqref{velocity} and \eqref{stress}. At first, we study the boundary controllability of the system \eqref{nmaxeq3 bd}, when the control acts only in density.
					\subsection{Control in density}
					Let us consider the system \eqref{nmaxeq3 bd} with \eqref{density}. For the reader's convenience, we recall the system:
					\begin{equation}\label{nmaxeq3 bd den}
						\left.
						\begin{aligned}
							&\partial_t\rho+u_s\partial_x\rho+\rho_s\partial_xu=0, \qquad&&\text{in }(0, T)\times (0, 2\pi),\\&\partial_tu+u_s\partial_xu+a\gamma{\rho_s}^{\gamma-2}\partial_x\rho-\frac{1}{\rho_s}\partial_xS=0,\qquad&&\text{in }(0, T)\times (0, 2\pi),\\&\partial_tS+\frac{1}{\kappa}S-\frac{\mu}{\kappa}\partial_xu=0, \qquad&&\text{in }(0, T)\times (0, 2\pi),\\&\rho(t, 0)=\rho(t, 2\pi)+q(t),\: u(t,0)=u(t,2\pi),\: S(t,0)=S(t,2\pi),\qquad&&  t\in (0, T),\\&\rho(0,x)=\rho_0(x),\quad u(0,x)=u_0(x),\quad S(0,x)=S_0(x),\qquad&&  x\in (0,2\pi).
						\end{aligned}
						\right\}
					\end{equation}
					\subsubsection{Well-posedness}
					Like interior controllability cases, at first let us write the system \eqref{nmaxeq3 bd den} in the following abstract form:
					\begin{equation} \label{nmaxop-eqnbd}
						\dot{z}(t) = \mathcal{A} z(t) + \mathcal{B}_{\rho} q(t), \quad t\in (0,T), \qquad z(0) = z_{0},
					\end{equation}
					where we have set $z(t) = (\rho(t,\cdot), u(t, \cdot), S(t,\cdot))^{\top},$ $z_{0} =(\rho_{0}, u_{0}, S_0)^{\top}$.
					One can identify the operator $\mc A, \mc B$ through its adjoint by taking the scalar product of \eqref{nmaxop-eqnbd} with a vector of smooth functions $(\phi, \psi, \xi)$ and comparing it with \eqref{nmaxeq3 bd den}.
					$\mc A$ is the unbounded operator given by \eqref{op} with its domain {\small$$\mathcal D(\mathcal A; \mathcal Z)=\left\lbrace \begin{pmatrix}
							\rho\\u\\S
						\end{pmatrix}\in \mathcal Z\: :\:\left(\rho,u,S \right)^\top \in H^1(0,2\pi)\times H^1(0,2\pi)\times H^1(0,2\pi),\:\begin{aligned}
							&\rho(0)=\rho(2\pi),\\ &u(0)=u(2\pi),\\&S(0)=S(2\pi).
						\end{aligned}
						\right\rbrace.
						$$}
					The control operator $\mathcal{B}_{\rho} \in \mathcal{L}(\mathbb{C};\mc D (\mc A^*)')$ defined by 
					\begin{equation*}
						\ip{\mc B_{\rho} q(t)}{(\phi, \psi, \xi)}_{\mc D (\mc A^*)',\mc D (\mc A^*)}=q(t)\left(bu_s\overline{\phi(2\pi)}+b\rho_s\overline{\psi(2\pi)}\right).
					\end{equation*}
					Clearly $\mc B_{\rho}$ is well-defined as $\mc B_{\rho}q$ is continuous on $H^1(0,2\pi)\times H^1(0,2\pi)\times H^1(0,2\pi)$ (by the embedding theorem $H^1(0, 2\pi)\hookrightarrow C[0,2\pi]$). Its adjoint $\mc B_{\rho}^* \in \mathcal{L}(\mc D (\mc A^*); \mathbb{C})$ is
					\begin{equation}\label{nmaxeqcontr bd}
						\mathcal{B}_{\rho}^*(\phi, \psi, \xi)=\mathbbm{1}_{\{x=2\pi\}}\begin{pmatrix}
							bu_s\\b\rho_s\\0
						\end{pmatrix}(\phi, \psi, \xi) =bu_s {\phi(2\pi)}+b\rho_s{\psi(2\pi)}.
					\end{equation}
					One can prove that the operator $\mc B_{{\rho}}$ satisfies the following so-called admissibility condition
					\begin{equation}\label{adm}
						\int_{0}^{T}|\mc B_{\rho}^* \mathbb{T}_{T-t}^*\Phi|^2\, \rd t\leq C \norm{\Phi}_{{\mc Z}}, \forall \, \Phi \in \mc D(\mc A^*),
					\end{equation}
					where $C$ is some positive constant. { The proof relies on the explicit expression of the solution of the adjoint \eqref{nmaxeqadj} and the right hand side inequality of \eqref{I-4}. Indeed, performing an integration of the square of identity \eqref{obs term bd} over $[0,T]$ and then applying \eqref{I-4}, we get the required admissibility conditions.}
					
					Since $\mc A$ generates a contraction semigroup and $\mc B_{\rho}$ is an admissible operator, we can prove the following well-posedness result (see Theorem $2.37$ in page 53 of \cite{Co07}, for more details):
					\begin{prop}\label{prop wel}
						Let $T>0$, $(\rho_0,u_0, S_0)^{\top} \in \dot L^2(0, 2\pi)\times \dot L^2(0, 2\pi)\times \dot L^2(0, 2\pi)$ and $q \in L^2(0,T)$. Then the system \eqref{nmaxeq3 bd den} has a unique solution $(\rho, u, S)^{\top}\in C\left([0,T]; (L^2(0, 2\pi)^3)\right)$ and the solution satisfies 
						\begin{align*}
							\|(\rho, u, S)\|_{C([0,T];\mathcal {{Z}})} \leqslant  C \Big(\|(\rho_0, u_0, S_0)\|_{\mathcal {{Z}}}+ \|q\|_{L^2(0,T)}\Big), 
						\end{align*}
						where $C=C(T)$ is a positive constant, independent of $\rho_0, u_0, S_0, q$ and $t$.
					\end{prop}
					
					\subsubsection{Exact Controllability}
					In this section, we prove the exact controllability of the system \eqref{nmaxeq3 bd den}. As interior control case, here at first we discuss the classical approach to deduce the observability inequality for the adjoint system \eqref{nmaxeqadj}  which essentially gives the exact controllability of the main system \eqref{nmaxeq3 bd den}. Let us consider the linear map $\mathcal{F}_T: L^2(0,T)\to (L^2(0,2\pi))^3$ by $\mathcal{F}_T(q)= \left(\rho(T,\cdot),u(T, \cdot), S(T, \cdot)\right)^{\top}$, where $(\rho, u, S)$ is the solution of the system \eqref{nmaxeq3 bd den} with $(\rho_0, u_0, S_0)=(0,0,0)$. It is clear that exact controllability for \eqref{nmaxeq3 bd den} is equivalent to the surjectivity of the map $\mathcal{F}_T$. Note that the map $\mathcal{F}_T$ is surjective if and only if there exists a constant $C>0$ such that the following inequality holds:
					\begin{align}\label{eq:fstar bd}
						\norm{\mathcal{F}_T^* \mathbf {z}}_{L^2(0,T)}\geq C\norm{\mathbf z}_{(L^2(0,2\pi))^3}, \, \text{ for all } \mathbf z \in (L^2(0,2\pi))^3.
					\end{align} A direct computation shows that
					$\mathcal{F}_T^*((\sigma_{T}, v_{T}, \tilde{S}_T) )=b u_s \sigma(t,2\pi)+b\rho_s  v(t,2\pi)$, where $(\sigma,v, \tl S)$ is the solution of the adjoint system \eqref{nmaxeqadj} with the terminal data $(\sigma_{T}, v_{T}, \tilde{S}_T)$. Thus we get the following proposition.
					\begin{prop}\label{nmaxthobs1 bd} 
							The system \eqref{nmaxeq3 bd den} is exactly controllable in $\mathcal Z_{m,m}$ at time $T>0$ using a control $q$ in $L^2(0,T)$ acting only in the density, if and only if, there exists a positive constant $C_T>0$  such that for any $(\sigma_T, v_T, \tilde{S}_T)\in \mathcal Z_{m,m}$, 
							$(\sigma, v, \tilde{S})$, the solution of \eqref{nmaxeqadj}, satisfies the following observability inequality:
							\begin{align}\label{nmaxobsden bd}
								\int_{0}^{2\pi} |\sigma_T(x)|^{2} \, \rd x+\int_{0}^{2\pi} |v_T(x)|^{2}\, \rd x+\int_{0}^{2\pi} |\tilde{S}_T(x)|^{2} \, \rd x\leq C_T\int_{0}^{T}|bu_s \sigma(t,2\pi)+b\rho_s v(t,2\pi)|^2 \,\rd t.
							\end{align}
						\end{prop} 
						Before giving the proof of the above observability inequality, we mention the following lemma which ensures that the observation term is nonzero.
						\begin{lem}\label{non zero ob}
							Let us recall the eigen functions $\mc E(\mc A^*)=  \left\lbrace \xi^*_{n,l}\: |\: 1\leq l\leq 3, n\in\mathbb{Z}^*\right\rbrace$ of the unbounded operator $\mc A^*.$ Then we have the following result:
							$$\mc B_{\rho}^* \xi\neq 0, \forall \, \xi \in \mc E(\mc A^*), $$
							where $\mc B_{\rho}^*$ is the observation operator associated to the system \eqref{nmaxeq3 bd den}, defined by \eqref{nmaxeqcontr bd}.
						\end{lem}
						\begin{proof}
							Observe that, $\mc B_{\rho}^* \xi^*_{n,l}=\frac{b}{{\psi}_{n}^l}\left({ u_s \alpha^1_{n,l}+\rho_s \alpha^2_{n,l}}\right).$ 
							From the first equation of the eigen equations $\mathcal{A}^*\xi^*_{n,l}=\overline{\lambda_n^l}\xi^*_{n,l}$, we obtain
							\begin{equation}\label{nmaxeigenequ 1}
								\frac{b}{{\psi}_{n}^l}\left({ u_s \alpha^1_{n,l}+\rho_s \alpha^2_{n,l}}\right)=\frac{b\alpha_{n,l}^1\overline{\lambda_n^l}}{{\psi}_{n}^l in}=\frac{\overline{\lambda_n^l}}{{\psi}_{n}^l in}\neq 0 \, ( \text{ using } \eqref{negrealpart},\, \eqref{nmaxequ-psi_nl}) .
							\end{equation}	
							Thus	$\mc B_{\rho}^* \xi^*_{n,l}\neq0.$
						\end{proof}
						\subsubsection{\textbf{Proof of \Cref{nmaxthm_pos bd}}}
						\begin{proof}	Let $(\sigma_T, v_T, \tilde{S}_T )^\top \in \mathcal Z_{m,m}.$ From \Cref{nmaxbasis_representation}, we have
							\begin{equation}\label{nmaxest_alphanl bd}
								\begin{pmatrix}
									\sigma_T\\v_T\\\tilde{S}_T
								\end{pmatrix}=\sum\limits_{n\in\mathbb{Z}^*}\sum\limits_{l=1}^{3}c_{n,l} \, \xi^*_{n,l} \quad \mathrm{with} \quad 
								\sum\limits_{n\in\mathbb{Z}^*}\sum\limits_{l=1}^{3}\left| c_{n,l}\right|^2   < \infty .
							\end{equation}
							Then the corresponding solution $(\sigma(t), v(t), \tilde{S}(t))^\top$ of the adjoint problem \eqref{nmaxeqadj} can be written as
							\begin{equation}\label{nmaxrepressoln bd}
								\begin{pmatrix}
									\sigma\\v\\\tilde{S}
								\end{pmatrix}(t,x)=\sum\limits_{n\in\mathbb{Z}^*}\sum\limits_{l=1}^{3}c_{n,l}e^{\overline{\lambda_n^l}(T-t)} \xi^*_{n,l}(x). 
							\end{equation}
							In particular, using the expression of $\xi^*_{n,l}$  \eqref{nmaxrepressoln bd}, we get
							\begin{align}
								&\sigma(t,x)=\sum\limits_{n\in\mathbb{Z}^*}^{\infty}\sum\limits_{l=1}^{3}\frac{c_{n,l}}{\psi_{n,l}}e^{\overline{\lambda_n^l}(T-t)}e^{inx} , \quad \forall\, t\in (0,T), \quad x\in (0,2\pi),\label{nmaxrepsigma bd}\\
								&v(t,x)=\sum\limits_{n\in\mathbb{Z}^*}^{\infty}\sum\limits_{l=1}^{3}\frac{c_{n,l}}{\psi_{n,l}}\alpha_{n,l}^2e^{\overline{\lambda_n^l}(T-t)}e^{inx}, \quad \forall\, t\in (0,T), \quad x\in (0,2\pi),\label{nmaxrepv bd}\\
								\label{st obs}		&\tilde S(t,x)=\sum\limits_{n\in\mathbb{Z}^*}^{\infty}\sum\limits_{l=1}^{3}\frac{c_{n,l}}{\psi_{n,l}}\alpha_{n,l}^3 e^{\overline{\lambda_n^l}(T-t)}e^{inx}, \quad \forall\, t\in (0,T), \quad x\in (0,2\pi).
							\end{align} 
							Then, in a similar technique used in the proof of \Cref{nmaxthm_pos}, we get a positive constant $C$, independent of $n$, and a large $N\in \mathbb{N}$, such that
							\begin{align}\label{nmaxobsleft2 bd}
								&\sum\limits_{|n|> N}\left(\left| \sum\limits_{l=1}^{3}\frac{c_{n,l}}{\psi_{n,l}}\right|^2+\left| \sum\limits_{l=1}^{3}\frac{c_{n,l}}{\psi_{n,l}}\alpha_{n,l}^2\right|^2+\left| \sum\limits_{l=1}^{3}\frac{c_{n,l}}{\psi_{n,l}}\alpha_{n,l}^3\right|^2 \right)\notag\\
								& \hspace{4cm}\leq  C  \sum\limits_{|n|> N}\sum\limits_{l=1}^{3}\left|\frac{c_{n,l}}{\psi_{n,l}} \right|^2  .
							\end{align}
							Putting $x=2\pi$ in \eqref{nmaxrepsigma bd} and \eqref{nmaxrepv bd}, we have
							\begin{align}
								&\sigma(t,2\pi)=\sum\limits_{n\in\mathbb{Z}^*}^{\infty}\sum\limits_{l=1}^{3}\frac{c_{n,l}}{\psi_{n,l}}e^{\overline{\lambda_n^l}(T-t)}, \quad \forall\, t\in (0,T),\label{nmaxrepsigmatwopi bd}\\
								&v(t,2\pi)=\sum\limits_{n\in\mathbb{Z}^*}^{\infty}\sum\limits_{l=1}^{3}\frac{c_{n,l}}{\psi_{n,l}}\alpha_{n,l}^2 \, e^{\overline{\lambda_n^l}(T-t)}, \quad \forall\, t\in (0,T),\label{nmaxrepvtwopi bd}\\
								\label{nmaxrepstress}	&\tilde S(t,2\pi)=\sum\limits_{n\in\mathbb{Z}^*}^{\infty}\sum\limits_{l=1}^{3}\frac{c_{n,l}}{\psi_{n,l}}\alpha_{n,l}^3 \, e^{\overline{\lambda_n^l}(T-t)}, \quad \forall\, t\in (0,T).
							\end{align} 
							Now the observation term becomes
							\begin{align}\label{obs term bd}
								|bu_s \sigma(t,2\pi)+b\rho_s v(t,2\pi)|=\left|\sum\limits_{n\in\mathbb{Z}^*}^{\infty}\sum\limits_{l=1}^{3}\frac{c_{n,l}}{\psi_{n,l}}(bu_s+\alpha_{n,l}^2 b\rho_s)e^{\overline{\lambda_n^l}(T-t)}\right|.
							\end{align}
							Using \Cref{propI-4}, along with \eqref{nmaxobsleft2 bd} and \eqref{obs term bd}, for $T>{2\pi}\left(\frac{1}{|\beta_1|}+\frac{1}{|\beta_2|}+\frac{1}{|\beta_3|}\right)$, we have the following inequality:
							\begin{align}\label{nmax ob}
								&\sum\limits_{|n|> N}\left(\left| \sum\limits_{l=1}^{3}\frac{c_{n,l}}{\psi_{n,l}}\right|^2+\left| \sum\limits_{l=1}^{3}\frac{c_{n,l}}{\psi_{n,l}}\alpha_{n,l}^2\right|^2+\left| \sum\limits_{l=1}^{3}\frac{c_{n,l}}{\psi_{n,l}}\alpha_{n,l}^3\right|^2 \right)\notag\\
								& \quad \quad\leq  C \int_{0}^{T} \left|\sum\limits_{|n|> N}\sum\limits_{l=1}^{3}\frac{c_{n,l}}{\psi_{n,l}}(u_s+\alpha_{n,l}^2 b)e^{\overline{\lambda_n^l}(T-t)}\right|^2 \, \rd t.
							\end{align}
							Then by \ref{nmaxsimpleeigen} and \Cref{non zero ob}, the missing finitely many exponential ( for $|n| \leq N$) can be added one by one in the inequality \eqref{nmax ob} as the required
							gap condition for the eigenvalues holds. 
							Thus finally we deduce the following observability inequality
							\begin{align}\label{obs 21}
								&\int_{0}^{2\pi} |\sigma_T(x)|^{2} \ \rd x \; + \; \int_{0}^{2\pi} |v_T(x)|^{2} \ \rd x + \; \int_{0}^{2\pi} |\tilde{S}_T(x)|^{2} \ \rd x\notag\\
								&\quad \quad \quad \quad \leq C \int_{0}^{T} \left|u_s \sigma(t,2\pi)+b v(t,2\pi)\right|^2 \, \rd t,
							\end{align}
							provided $T>{2\pi}\left(\frac{1}{|\beta_1|}+\frac{1}{|\beta_2|}+\frac{1}{|\beta_3|}\right)$. Hence \Cref{nmaxthm_pos bd} is proved.
						\end{proof}
						
						In a similar approach as density control case, we can also prove the exact controllability result (see \Cref{nmaxremnul}) of the system \eqref{nmaxeq3 bd}, when control acts only on the velocity or stress. In the following sections, we just indicate the changes according to the system \eqref{nmaxeq3 bd} with \eqref{velocity} and \eqref{nmaxeq3 bd} with \eqref{stress}.
						{\begin{rem}
								The proof of the observability inequality for the case of multiple eigenvalues can be demonstrated in a similar approach as done in the case of interior controllability in \Cref{sec_multi_ev_b}.
						\end{rem}}
						\subsection{Control in velocity}
						Here the control operator $\mathcal{B}_{u} \in \mathcal{L}(\mathbb{C};\mc D (\mc A^*)')$ is defined by 
						\begin{equation*}
							\ip{\mc B_{u} r(t)}{(\phi, \psi, \xi)}_{\mc D (\mc A^*)',\mc D (\mc A^*)}=r\left(b\rho_s\overline{\phi(2\pi)}+\rho_s u_s\overline{\psi(2\pi)}-\overline{\xi(2\pi)}\right).
						\end{equation*}
						Clearly $\mc B_{u} $ is well-defined as $\mc B_{u} r$ is continuous on $H^1(0,2\pi)\times H^1(0,2\pi)\times H^1(0,2\pi)$ (by the embedding theorem $H^1(0, 2\pi)\hookrightarrow C^0[0,2\pi]$). Its adjoint $\mc B_{u} ^* \in \mathcal{L}(\mc D (\mc A^*); \mathbb{C})$ is given by
						\begin{equation}\label{nmaxeqcontr bd_vel}
							\mathcal{B}_{u} ^*(\phi, \psi, \xi)=\mathbbm{1}_{\{x=2\pi\}}\begin{pmatrix}
								b\rho_s\\\rho_s u_s\\-1
							\end{pmatrix}(\phi, \psi, \xi) =b\rho_s{\phi(2\pi)}+\rho_s u_s{\psi(2\pi)}-{\xi(2\pi)}.
						\end{equation}
						As in the density control case, here also one can prove that the control operator satisfies the admissibility condition \eqref{adm}.
						Thus for any $(\rho_0,u_0, S_0)^{\top} \in \mathcal Z_{m,m}$ and $r \in L^2(0,T)$, the system \eqref{nmaxeq3 bd}-\eqref{velocity} has a unique solution and the solution satisfies: 
						\begin{align*}
							\|(\rho, u, S)\|_{C([0,T]; (L^2(0, 2\pi)^3)} \leqslant  C \Big(\|(\rho_0,u_0, S_0)\|_{\mathcal Z}+ \|r\|_{L^2(0,T)}\Big), 
						\end{align*}
						where $C=C(T)$ is a positive constant, independent of $\rho_0, u_0, S_0, r$ and $t$.
						\begin{prop}\label{nmaxthobs1 bd_vel} 
								The system \eqref{nmaxeq3 bd}-\eqref{velocity} is exactly controllable in $\mathcal {{ Z}}_{m,m}$ at time $T>0$ using a  boundary control $r$ in $L^2(0,T)$ acting only in the velocity, if and only if, there exists a positive constant $C_T>0$  such that for any $(\sigma_T, v_T, \tilde{S}_T)^{\top}\in \mathcal Z_{m,m}$, 
								$(\sigma, v, \tilde{S})$, the solution of \eqref{nmaxeqadj}, satisfies the following observability inequality:
								\begin{multline}\label{nmaxobsden bd_vel}
									\int_{0}^{2\pi} |\sigma_T(x)|^{2} \, \rd x+\int_{0}^{2\pi} |v_T(x)|^{2} \, \rd x+\int_{0}^{2\pi} |\tilde{S}_T(x)|^{2}\, \rd x\leq C_T\int_{0}^{T}|b\rho_s \sigma(t,2\pi)\\+\rho_s u_s v(t,2\pi)-\tilde{S}(t,2\pi)|^2\, \rd t.
								\end{multline}
							\end{prop}
							
							\begin{lem}
								Let us recall the eigen functions $\mc E(\mc A^*)= \left\lbrace \xi^*_{n,l}\: |\: 1\leq l\leq 3, n\in\mathbb{Z}^*\right\rbrace$ of the unbounded operator $\mc A^*.$ Then we have the following result:
								$$\mc B_{u} ^* \xi\neq 0, \forall \, \xi \in \mc E(\mc A^*). $$
							\end{lem}
							\begin{proof}
								From \eqref{nmaxeqcontr bd_vel} and \eqref{nmaxcoeffofxi}, we have
								\begin{equation}\label{nmaxb*xi}
									\mathcal{B}_{u} ^*\xi^*_{n,l}=\frac{1}{\psi_{n,l}}(b\rho_s+\alpha_{n,l}^2 \rho_s u_s-\alpha_{n,l}^3).
								\end{equation}
								The second equation of the eigen-equations $\mathcal{A}^*\xi^*_{n,l}=\overline{\lambda_n^l}\xi^*_{n,l}$ gives
								\begin{equation}\label{nmaxeigenequ}
									b\rho_s+\alpha_{n,l}^2 \rho_s u_s-\alpha_{n,l}^3=\frac{\rho_s\alpha_{n,l}^2\overline{\lambda_n^l}}{in}.
								\end{equation}	
								Thus from \eqref{nmaxb*xi}, \eqref{nmaxeigenequ} and \eqref{nmaxequ-xi_n^*}, we have
								\begin{equation*}
									\mathcal{B}_{u} ^*\xi^*_{n,l}=-\frac{\rho_s\overline{\lambda_n^l}(\overline{\lambda_n^l}-inu_s)}{n^2\psi_{n,l}\rho_s}\neq 0,
								\end{equation*}
								since all the eigenvalues have negative real part.
							\end{proof}

							\subsection{Control in stress}
							The control operator $\mathcal{B}_{S}  \in \mathcal{L}(\mathbb{C};\mc D (\mc A^*)')$ defined by 
							\begin{equation*}
								\ip{\mc B_{ S} p(t)}{(\phi, \psi, \xi)}_{\mc D (\mc A^*)',\mc D (\mc A^*)}=-p\psi(2\pi).
							\end{equation*}
							Clearly $\mc B_{S} $ is well-defined as $\mc B_{S} p$ is continuous on $H^1(0,2\pi)\times H^1(0,2\pi)\times H^1(0,2\pi)$ (by the embedding theorem $H^1(0, 2\pi)\hookrightarrow C^0[0,2\pi]$). Its adjoint $\mc B_{S}^* \in \mathcal{L}(\mc D (\mc A^*); \mathbb{C})$ is
							\begin{equation}\label{nmaxeqcontr bd st}
								\mathcal{B}_{S} ^*(\phi, \psi, \xi)=\mathbbm{1}_{\{x=2\pi\}}\begin{pmatrix}
									0\\-1\\0
								\end{pmatrix}(\phi, \psi, \xi) =-{\psi(2\pi)},
							\end{equation}
							and also $\mc B_{S} ^* \xi\neq 0, \forall \,\xi \in \mc E(\mc A^*). $
							
							\section{Small time lack of controllability}\label{lack sec}
							{\subsection{Interior control} In this section, we study the lack of exact controllability of the system \eqref{nmaxeq3} when the localized interior control is acting on the density equation. Similar result holds true for the remaining two cases (control in velocity or control in stress). We establish the result by violating the observability inequality on an appropriate solution of the corresponding adjoint system. Our proof is in the same spirit of \cite[Section 3]{Beauchard}. We first find a candidate function as a terminal data of the transport equation and exploit the lack of controllability of the same in small time.

								Without loss of generality, we consider the control domain $\mathcal{O}=(l_1,l_2), 0\leq l_1<l_2<2\pi.$	
								\begin{theorem}\label{lack theorem}
									Let us denote $\beta=\min\{|\beta_i|, i=1,2,3\}$ and	consider $\beta T<\max\{l_1, 2\pi-l_2\}$. Then the system \eqref{nmaxeq3} is not exactly controllable in time $T$ by means of the controls $f_i\in L^2\left( 0,T;L^2(\mathcal{O})\right), i=1,2,3$ acting on every equations of \eqref{nmaxeq3}.
								\end{theorem}
								\begin{proof}
							{\color{blue}Throughout this proof,	without loss of generality, we assume  $\beta=|\beta_1|$ and $\beta_1>0$.}	Since $\beta T<\max\{l_1, 2\pi-l_2\}$, there exists a nontrivial function $\hat \sigma_T \in C_c^{\infty}(0,2\pi)$ 
									such that the solution of the following equation \begin{equation}
										\begin{cases}
											\hat{\sigma}_t-\beta\hat{\sigma}_x-\omega_1\hat \sigma=0, &  \mbox{ in } (0,T) \times (0, 2\pi),\\
											\hat \sigma(t,0)=\hat \sigma(t, 2\pi), &  \mbox{ in } (0,T),\\
											\hat \sigma(T,x)=\hat \sigma_T(x), &  \mbox{ in } (0, 2\pi),
										\end{cases}
									\end{equation} 
									satisfies that $\text{supp}(\hat{\sigma}) \cap [0,T]\times \mathcal{O} =\emptyset $ (see also \cite[Section 3]{Beauchard}). We will use the asymptotic behavior of the spectrum of the linearized operator to show the lack of controllability.  At first, we construct a high-frequency version of the candidate function $\hat \sigma_T$. Let $N>0$ be a fixed integer. Next, we define the polynomial
									\begin{equation}\label{pn}
										P^N(X):=\prod_{j=-N}^{N}(X-j),
									\end{equation}
									and the function
									\begin{equation*}
										\hat{\sigma}_T^N:=P^N\left(-i\frac{d}{dx}\right)\hat{\sigma}_T.
									\end{equation*}
									Since $\hat{\sigma}_T^N$ is the image of $\hat{\sigma}_T$ by a differential operator, we have $\text{supp}({\hat{\sigma}_T^N})\subset \text{supp}(\hat{\sigma}_T)$. Let us write
									$$
									\hat{\sigma}_T=\sum_{n \in \mathbb{Z}} a_n e^{ in   x} \text {, where } a_n=\frac{1}{2\pi}\int_{0}^{2\pi} e^{ in   x} \hat{\sigma}_T(x) d x, \forall n \in \mathbb{Z} .
									$$
									Then using the definition of $P^N$ and $\hat\sigma_T^N$, we can write $\hat\sigma_T^N$ in the following:
									\begin{equation*}
										\hat{\sigma}_T^N(x)=\sum_{n\in\mathbb{Z}}a_n\prod_{j=-N}^{N}\left(-i\frac{d}{dx}-j\right)e^{inx}=\sum_{n\in\mathbb{Z}}a_n\prod_{j=-N}^{N}\left(n-j\right)e^{inx}=\sum_{n\in\mathbb{Z}}a_nP^N(n)e^{inx},
									\end{equation*}
									for $(t,x)\in (0,T)\times(0,2\pi)$. Note that $P^N(n)=0$ for all $|{n}|\leq N$ and therefore
									\begin{equation}\label{term}
										\hat{\sigma}_T^N(x)=\sum_{{|n|}\geq N+1}a_nP^N(n)e^{inx}.
									\end{equation}
									Next, we denote the function $\left(\hat{\sigma}^N,\hat{v}^N, \hat{S}^N \right)$ is the solution of the following equation
									\begin{equation}\label{tr}
										\begin{cases}
											\hat{\sigma}^N_t-\beta\hat{\sigma}^N_x-\omega_1\hat \sigma^N=0, &  \mbox{ in } (0,T) \times (0, 2\pi),\\	\hat{v}^N_t-\beta\hat{v}^N_x-{\color{blue}\omega_1}\hat v^N=0, &  \mbox{ in } (0,T) \times (0, 2\pi),\\	\hat{S}^N_t-\beta\hat{S}^N_x-{\color{blue}\omega_1}\hat S^N=0, &  \mbox{ in } (0,T) \times (0, 2\pi),\\
											\hat \sigma^N(t,0)=\hat \sigma^N(t, 2\pi), \hat v^N(t,0)=\hat v^N(t, 2\pi), \hat S^N(t,0)=\hat S^N(t, 2\pi), &  \mbox{ in } (0,T),\\
											\left(\hat{\sigma}^N,\hat{v}^N, \hat S^N\right)(T,x)=\hat{\sigma}_T^N(x)\left(1,  -\frac{ \beta_l+u_s }{ \rho_s}, \frac{\mu\left( \beta_l+u_s\right) }{\kappa \rho_s \beta_l}\right), &  \mbox{ in } (0, 2\pi).
										\end{cases}
									\end{equation} 
									Let $(\sigma^N, v^N, \tl S^N)$ be the solution of the following adjoint system:
									\begin{equation}\label{nmaxeqadj lack}
										\begin{cases}
											\partial_t \sigma^N +u_s\pa_x\sigma^N+{\rho_s}\pa_x v^N =0, & \mbox{ in } (0,T) \times (0, 2\pi), \\
											\partial_t  v^N+b \partial_{x}  \sigma^N+u_s\pa_x  v^N - \frac{1}{\rho_s} \partial_x\tilde{S}^N= 0, &  \mbox{ in } (0,T) \times (0, 2\pi),\\
											\partial_t\tilde{S}^N-\frac{1}{\kappa}\tilde {S}^N - \frac{\mu}{\kappa} \partial_x  v^N= 0, &  \mbox{ in } (0,T) \times (0, 2\pi),\\
											\sigma^N(t,0) = \sigma^N(t, 2\pi),  v^N(t,0) = v^N(t, 2\pi),\:  \tilde S^N(t,0)=\tilde S^N(t,2\pi),&  \mbox{ in } (0,T), \\
											\sigma^N(T,x)=\hat \sigma^N_{T}(x), \quad  v^N(T,x)= \hat v^N_T(x), \quad \tilde{S}^N(T,x)=\hat{S}^N_T(x), & \mbox{ in } (0,2\pi),
										\end{cases}
									\end{equation}
									where $\displaystyle (\hat\sigma^N_{T}, \hat v^N_T, \hat S^N_T)^\top=\sum_{|n|\geq N+1}c^1_n  \,\xi^*_{n,1},  \text{ where } \frac{c^1_n}{\psi_{n,1}} =a_n^N=a_n P^N(n)$ and $\psi_{n,1}, \,\xi^*_{n,1}$ are defined in \eqref{nmaxcoeffofxi}.
									
									\noindent
									Let us denote the $3\times 3$ matrix $R:=\text{diag}(\omega_1, {\color{blue}\omega_1}, {\color{blue}\omega_1}).$
									We write the solutions of the systems \eqref{tr} and \eqref{nmaxeqadj lack} respectively as
									\begin{align}
										\left(\hat{\sigma}^N, 	\hat{v}^N, 	\hat{S}^N\right)^\top(t,x)&=\sum_{|{n}|\geq N+1}a^N_n \, e^{-\beta in(T-t)}e^{inx}e^{-tR}\left(1,  -\frac{ \beta_l+u_s }{ \rho_s}, \frac{\mu\left( \beta_l+u_s\right) }{\kappa \rho_s \beta_l}\right)^{\top},\\
										\left(\sigma^N,v^N,\tilde {S}^N\right)^\top(t,x)&=\sum_{|{n}|\geq N+1}a^N_n \, e^{\overline{\lambda_n^1}(T-t)}e^{inx}\left(1,  \alpha_{n,1}^2,   \alpha_{n,1}^3 \right)^{\top},
									\end{align}
									where $\alpha_{n,1}^2,\alpha_{n,1}^3$
									are given by \eqref{nmaxcoeffofxi}, \eqref{nmaxcgsalpha2} and \eqref{nmaxcgsalpha3}. Note that $\alpha_{n,1}^2$ and $\alpha_{n,1}^3$ converge to $-\frac{ \beta_l+u_s }{ \rho_s}, $ and $\frac{\mu\left( \beta_l+u_s\right) }{\kappa \rho_s \beta_l}$, respectively as $O(1/n).$
									
									Now, we prove that the solution $(\sigma^N, v^N, S^N)$ of \eqref{nmaxeqadj lack} is an approximation of the solution $(\hat{\sigma}^N, \hat{v}^N, \hat{S}^N )$ of \eqref{tr}. Indeed,
									\begin{align*}
										&\norm{\sigma^N(t,\cdot)-\hat{\sigma}^N(t,\cdot)}_{L^2(0,2\pi)}^2\\
										&\leq C\sum_{|{n}|\geq N+1}|{a^N_n}|^2\abs{e^{\overline{\lambda_n^1}(T-t)}-e^{\left({-\omega_1-i \beta n }\right)(T-t)}}^2\\
										&\leq C\sum_{|{n}|\geq N+1}\frac{1}{|{n}|^2}|{a^N_n}|^2,
									\end{align*}
									for all $t\in[0,T]$ and therefore performing integration over $[0,T]$ we have
									\begin{equation*}
										\norm{\sigma^N-\hat{\sigma}^N}_{L^2(0,T, L^2(0,\pi))}^2\leq\frac{C}{|{N}|^2}\sum_{|{n}|\geq N+1}|{a^N_n}|^2.
									\end{equation*}
									Using triangle inequality, we deduce
									\begin{align*}
										\norm{\sigma^N}_{L^2(0,T, L^2(\mathcal{O}))}^2&\leq 	\norm{ \sigma^N-\hat{\sigma}^N}_{L^2(0,T, L^2(\mathcal{O})}^2+	\norm{\hat \sigma^N}_{L^2(0,T, L^2(\mathcal{O}))}^2.
									\end{align*}
									Since the support of $\hat \sigma_N$ does not intersect the domain $[0,T]\times \mathcal{O}$, the second term of the right hand sides vanishes. Thus we get
									\begin{equation}\label{lack 1}
										\norm{ \sigma^N}_{L^2(0,T, L^2(\mathcal{O}))}^2\leq \frac{C}{|{N}|^2}\sum_{|{n}|\geq N+1}|{a^N_n}|^2. 
									\end{equation}
									Similar analysis helps to deduce the following
									\begin{equation}\label{lack 3}
										\norm{ v^N}_{L^2(0,T, L^2(\mathcal{O}))}^2+ 	\norm{ \tilde{S}^N}_{L^2(0,T, L^2(\mathcal{O}))}^2\leq \frac{C}{|{N}|^2}\sum_{|{n}|\geq N+1}|{a^N_n}|^2. 
									\end{equation}
									Next, if possible we assume that the observability for the system \eqref{nmaxeqadj lack} holds. Therefore we have the following:
									\begin{align}\label{lack 2}
										\nonumber	\int_{0}^{2\pi} |\hat\sigma^N_T(x)|^{2} \ \rd x \; + \; \int_{0}^{2\pi} | \hat v^N_T(x)|^{2}& \ \rd x + \; \int_{0}^{2\pi} |\hat{S}^N_T(x)|^{2} \ \rd x 
										\leqslant C_T \bigg[ \int_0^T\int_{\mathcal{O}}|\sigma^N(t,x)|^2\, \rd x\,\rd t\\
										&+\int_0^T\int_{\mathcal{O}}|v^N(t,x)|^2\, \rd x\,\rd t+\int_0^T\int_{\mathcal{O}}|\tilde{S}^N(t,x)|^2\, \rd x\,\rd t\bigg].
									\end{align}
									Using \eqref{lack 1} and \eqref{lack 3} along with the expression of $(\hat\sigma^N_{T}, \hat v^N_T, \hat S^N_T)$, we have from \eqref{lack 2}
									\begin{align}\label{lack 6}
										\nonumber	\int_{0}^{2\pi} |\hat\sigma^N_T(x)|^{2} \ \rd x + \; \int_{0}^{2\pi} | \hat v^N_T(x)|^{2} \ \rd x +& \; \int_{0}^{2\pi} |\hat{S}^N_T(x)|^{2} \ \rd x \leq	
										\frac{C}{|{N}|^2}\bigg[\int_{0}^{2\pi} |\hat \sigma^N_T(x)|^{2} \rd x\\
										&+\int_{0}^{2\pi} |\hat v^N_T(x)|^{2} \rd x+\int_{0}^{2\pi} |\hat S^N_T(x)|^{2} \ \rd x\bigg],
									\end{align}
									which gives a contradiction.
								\end{proof}
								\begin{rem}
									In \Cref{lack theorem}, if the control domain $\mathcal{O}$ is $(0, l_2),$ the we get the lack of controllability when $T<\frac{2\pi-l_2}{\beta}.$ Similarly, one can prove the negative controllability result for the control domain $\mathcal{O}=(l_1, 2\pi)$ at the time $T<\frac{l_1}{\beta}.$  
								\end{rem}
								\subsection{Boundary control} Using the similar argument as done in the interior control case, here we will show that the system \eqref{nmaxeq3 bd}-\eqref{velocity} is not controllable in small time by means of boundary control acting on the velocity. Other boundary control cases \eqref{nmaxeq3 bd}-\eqref{density}, \eqref{nmaxeq3 bd}-\eqref{stress} can be done similarly.
								\begin{theorem}
									Let $0<T<\frac{2\pi}{\beta}$, where $\beta$ is defined in \Cref{lack theorem}. Then the system  \eqref{nmaxeq3 bd} is not controllable in small time by means of boundary control acting in the boundary term in the velocity.
								\end{theorem}
							\begin{proof} 	{\color{blue} Without loss of generality, let us assume $\beta=|\beta_1|$ and $\beta_1>0.$} Let us first consider the following transport equations
								\begin{equation}\label{trnsprt_b1}
									\begin{cases}
										\check{\sigma}_t-\beta\check{\sigma_x}-\omega_1\check{\sigma}=0,\ & \text{ in } (0,T)\times(0,2\pi),\\
										\check{\sigma}(t,0)=\check{\sigma}(t,2\pi),\ & \text{ in } (0,T),\\
										\check{\sigma}(T,x)=\check{\sigma}_T(x),& \text{ in } (0,2\pi),
									\end{cases}
								\end{equation}
								with $\check{\sigma}_T\in {L}^2(0,2\pi)$. Since $T<\frac{2\pi}{\beta}$, there exists a  nontrivial function $\check{\sigma}_T  \in C^{\infty}(0,2\pi)$ with $$\text{supp}(\check{\sigma}_T)\subset\left(T\beta, 2\pi\right), \,$$  
								such that the solutions $\check{\sigma}$  of \eqref{trnsprt_b1} satisfy 
									$\check{\sigma}(t,0)=\check{\sigma}(t,2\pi)=0$ 
									but $\check{\sigma}$ is not identically zero in $(0,T)\times(0,2\pi)$. 
									
									For any fixed integer $N>0$, we define
									\begin{equation*}
										\check{\sigma}_T^N:=P^N\left(-i\frac{d}{dx}\right)\check{\sigma}_T,
									\end{equation*}
									where $P_N$ is given by \eqref{pn} and
									$$
									\check{\sigma}_T=\sum_{n \in \mathbb{Z}} c_n e^{ in   x} \text {, where } c_n=\frac{1}{2\pi}\int_{0}^{2\pi} e^{ in   x} \check{\sigma}_T(x) d x, \forall n \in \mathbb{Z} .
									$$
									Then using the definition of $P^N$ and $\check\sigma_T^N$, we can write 
									\begin{equation}\label{term1}
										\check{\sigma}_T^N(x)=\sum_{{|n|}\geq N+1}c_nP^N(n)e^{inx}.
									\end{equation}
									Next, we denote the function $(\check{\sigma}^N, \check{v}^N,\check{S}^N )$ as the solution of the following equation
									\begin{equation}\label{tr1}
										\begin{cases}
											\check{\sigma}^N_t-\beta\check{\sigma}^N_x-\omega_1\check \sigma^N=0, &  \mbox{ in } (0,T) \times (0, 2\pi),\\
											\check{v}^N_t-\beta\check{v}^N_x-{\color{blue}\omega_1}\check v^N=0, &  \mbox{ in } (0,T) \times (0, 2\pi),\\
											\check{S}^N_t-\beta\check{S}^N_x-{\color{blue}\omega_1}\check S^N=0, &  \mbox{ in } (0,T) \times (0, 2\pi),\\
											\check \sigma^N(t,0)=\check \sigma^N(t, 2\pi)=\check v^N(t,0)=\check v^N(t, 2\pi)=\check S^N(t,0)=\check S^N(t, 2\pi), &  \mbox{ in } (0,T),\\
											\left(\check{\sigma}^N,\check{v}^N,\check{S}^N\right)(T, x)=\check{\sigma}_T^N(x)\left(1,  -\frac{ \beta_l+u_s }{ \rho_s}, \frac{\mu\left( \beta_l+u_s\right) }{\kappa \rho_s \beta_l}\right), &  \mbox{ in } (0, 2\pi),
										\end{cases}
									\end{equation} 
									and
									let $(\sigma^N, v^N, \tl S^N)$ be the solution of the  adjoint system \begin{equation}\label{nmaxeqadj lack1}
										\begin{cases}
											\partial_t \sigma^N +u_s\pa_x\sigma^N+{\rho_s}\pa_x v^N =0, & \mbox{ in } (0,T) \times (0, 2\pi), \\
											\partial_t  v^N+b \partial_{x}  \sigma^N+u_s\pa_x  v^N - \frac{1}{\rho_s} \partial_x\tilde{S}^N= 0, &  \mbox{ in } (0,T) \times (0, 2\pi),\\
											\partial_t\tilde{S}^N-\frac{1}{\kappa}\tilde {S}^N - \frac{\mu}{\kappa} \partial_x  v^N= 0, &  \mbox{ in } (0,T) \times (0, 2\pi),\\
											\sigma^N(t,0) = \sigma^N(t, 2\pi),  v^N(t,0) = v^N(t, 2\pi),\:  \tilde S^N(t,0)=\tilde S^N(t,2\pi),&  \mbox{ in } (0,T), \\
											\sigma^N(T,x)=\check \sigma^N_{T}(x), \quad  v^N(T,x)= \check v^N_T(x), \quad \tilde{S}^N(T,x)=\check{S}^N_T(x), & \mbox{ in } (0,2\pi),
										\end{cases}
									\end{equation}
									where $\displaystyle (\check\sigma^N_{T}, \check v^N_T, \check S^N_T)=\sum_{|n|\geq N+1}c^1_n  \,\xi^*_{n,1},$  { where } $\frac{c^1_n}{\psi_{n,1}} =c_n^N=c_n P^N(n)$ and $\psi_{n,1}, \,\xi^*_{n,1}$ are defined in \eqref{nmaxcoeffofxi}.
									We write the solutions of the systems \eqref{tr} and \eqref{nmaxeqadj lack} respectively as
									\begin{align}
										\left(\check{\sigma}^N, 	\check{v}^N, 	\check{S}^N\right)^\top(t,x)&=\sum_{|{n}|\geq N+1}c^N_n \, e^{-\beta in(T-t)}e^{inx}e^{-tR}\left(1,  -\frac{ \beta_l+u_s }{ \rho_s}, \frac{\mu\left( \beta_l+u_s\right) }{\kappa \rho_s \beta_l}\right)^{\top},\\
										\left(\sigma^N,v^N,\tilde {S}^N\right)^\top(t,x)&=\sum_{|{n}|\geq N+1}c^N_n \, e^{\overline{\lambda_n^1}(T-t)}e^{inx}\left(1,  \alpha_{n,1}^2,   \alpha_{n,1}^3 \right)^{\top}.
									\end{align}
									As the previous section, now we prove that the solution $(\sigma^N, v^N, \tilde S^N)$ of \eqref{nmaxeqadj lack} is an approximation of the solution $(\check{\sigma}^N, \check{v}^N, \check{S}^N )$ of \eqref{tr1}. Indeed,
									\begin{align*}
										&\norm{\sigma^N(\cdot,x)-\check{\sigma}^N(\cdot,x)}_{L^2(0,T)}^2\\
										&\leq\sum_{|{n}|\geq N+1}|{c_n}|^2|{P^N(n)}|^2\norm{e^{\overline{\lambda_n^1}(T-t)}-e^{\left(-i\beta_1n-\omega_1\right)(T-t)}}_{L^2(0,T)}^2\\
										&\leq\sum_{|{n}|\geq N+1}|{c_n}|^2|{P^N(n)}|^2\norm{e^{\left(-\omega_1 + O\left( \frac{1}{|n|}\right)\right)(T-t)}-e^{-\omega_1(T-t)}}_{L^2(0,T)}^2\\
										&\leq\sum_{|{n}|\geq N+1}\frac{1}{|{n}|^2}|{c_n}|^2|{P^N(n)}|^2,
									\end{align*}
									for all $x\in[0,2\pi]$ and therefore
									\begin{equation}\label{tr sig}
										\norm{\sigma^N(\cdot,x)-\check{\sigma}^N(\cdot,x)}_{L^2(0,T)}^2\leq\frac{C}{|{N}|^2}\sum_{|{n}|\geq N+1}|{c_n}|^2|{P^N(n)}|^2,
									\end{equation}
									for all $x\in[0,2\pi]$.
									Similarly we also find, for all $x\in[0,2\pi]$
									\begin{equation}\label{tri vs}
										\norm{v^N(\cdot,x)-\check{v}^N(\cdot,x)}_{L^2(0,T)}^2+	\norm{\tilde{S}^N(\cdot,x)-\check{S}^N(\cdot,x)}_{L^2(0,T)}^2\leq\frac{C}{|{N}|^2}\sum_{|{n}|\geq N+1}|{c_n}|^2|{P^N(n)}|^2.
									\end{equation}
									Let us now suppose that the following observability inequality holds
									\begin{align}\label{nmaxobsden bd lack}
										\nonumber	\int_{0}^{2\pi} |\check\sigma^N_T(x)|^{2} \, \rd x+\int_{0}^{2\pi} |\check v^N_T(x)|^{2}\, \rd x+\int_{0}^{2\pi} |\check S^N_T(x)|^{2} \, \rd x\leq C_T\int_{0}^{T}|b\rho_s \sigma^N(t,2\pi)\\+\rho_s u_s v^N(t,2\pi)-\tilde{S}^N(t,2\pi)|^2 \,\rd t.
									\end{align}
									Thanks to the estimates \eqref{tr sig} and \eqref{tri vs}, using triangle inequality along with the facts $\check\sigma(t,\cdot), \check v(t,\cdot), \check S(t,\cdot)$ are vanishes at $x=2\pi$, we have
									\begin{align}\label{nmaxobsden bd lack1}
										\nonumber	\int_{0}^{2\pi} |\check\sigma^N_T(x)|^{2} \, \rd x+\int_{0}^{2\pi} |\check v^N_T(x)|^{2}\, \rd x+\int_{0}^{2\pi} |\check S^N_T(x)|^{2} \, \rd x\leq \frac{C}{|{N}|^2}&\bigg[|\check\sigma^N_T(x)|^{2} \, \rd x+\int_{0}^{2\pi} |\check v^N_T(x)|^{2}\, \rd x\\
										&+\int_{0}^{2\pi} |\check S^N_T(x)|^{2} \, \rd x\bigg],
									\end{align}
									which gives a contradiction.\end{proof}		
								}
								\section{Rapid exponential stabilization}\label{stab} Let us recall the control system \eqref{nmaxeq3 bd}. The goal of this section is to study the boundary stabilization issues for \eqref{nmaxeq3 bd} with single boundary control force. More precisely, we construct a stationary feedback law $q(t)$, $r(t)$ or $p(t)$ of the form $\Pi\left(\rho(t,\cdot), u(t,\cdot), S(t,\cdot)\right)$ such that the solution of the closed loop system \eqref{nmaxeq3 bd} decays exponentially to zero at any prescribed decay rate. At first we describe Urquiza's approach \cite{UR05} which is the key argument of this section.
								\subsection{Urquiza's method}
								Let us consider an abstract control system
								\begin{equation} \label{eq:abs}
									\begin{cases}
										\dot{y}(t)=\mc Ay(t)+ \mc B q(t) , \quad t\in (0,T),\\
										y(0)=y_0,
									\end{cases}
								\end{equation}
								where $y(t)\in \mc H, y_0\in \mc H, q\in L^2(0,T)$, $\mc B$ is an unbounded operator from $\cplx$ to $\mc H$. $\mc A:D(\mc A)\subset \mc H\to \mc H$ is an unbounded operator and $\mc D(\mc A)$ is dense in $\mc H.$ To employ the method of Urquiza, one needs to take the following assumptions on the operator $\mc A$ and $\mc B$:
								\begin{itemize}
									\item[(H1)] $\mc A$ is the infinitesimal generator of a strongly continuous group $\{e^{t \mc A}\}_{t\in \rea}$ on $\mc H$.
									\item[(H2)] $\mc B:\cplx\to \mc D(\mc A^*)'$ is linear and continuous.
									\item[(H3)] \textit{Regularity property.} For all $T>0$ there exists $C(T)>0$ such that
									\begin{equation*}
										\int_{0}^{T}\abs{\mc B^*e^{-t\mc A^*} y_T}^2\leq C \norm{y_T}^2_{\mc H}, \quad \forall \,y_T \in \mc D(\mc A^*).
									\end{equation*}
									\item[(H4)] \textit{Controllability property.} There are two constants $T>0$ and $c(T)>0$ such that
									\begin{equation*}
										\int_{0}^{T}\abs{\mc B^*e^{-t\mc A^*} y_T}^2\geq c \norm{y_T}^2_{\mc H}, \quad \forall \,  y_T \in \mc D(\mc A^*).	
									\end{equation*}
								\end{itemize}
								These hypotheses lead to the the following stabilization result. Its proof relies on algebraic Riccati equation associated with the linear quadratic regulator problem \cite{IL88}. Let us introduce the growth bound of the semigroup $\{e^{\mc A t}\}_{t\geq 0}$ of continuous linear operator as follows:
								\begin{equation*}
									g(\mc A)=\inf_{t> 0}\frac{1}{t}\log\norm{e^{t\mc A}}_{\mathcal{L}(\mc H)}.
								\end{equation*}
								\begin{theorem}\label{thm}[Urquiza \cite{UR05}, Theorem $2.1$]
									Consider operators $\mc A$ and $\mc B$ under assumptions (H1)-(H4). For any $\omega>\max\left(g(-\mc A),0\right)$, we have
									\begin{itemize}
										\item[(i)] The symmetric positive operator $\Lambda_{\omega}$ defined by
										\begin{equation*}
											\ip{\Lambda_{\omega}x}{z}_{\mc H}=\int_{0}^{\infty}\ip{\mc B^*e^{-\tau(\mc A+\omega I)^*}x}{\mc B^*e^{-\tau(\mc A+\omega I)^*}z}_{\cplx}d\tau, \, \forall \, x, z\in \mc H
										\end{equation*}
										is coercive and is an isomorphism on $\mc H$.
										\item[(ii)] Let $\Pi_{\omega}:=-\mc B^*\Lambda_{\omega}^{-1}$. The operator $\mc A+\mc B\Pi_{\omega}$ with $\mc D(\mc A+\mc B\Pi_{\omega})=\Lambda_{\omega}(\mc D(\mc A^*))$ is the infinitesimal generator of a strongly continuous semigroup on $\mc H$.
										\item[(iii)] The closed-loop system \eqref{eq:abs} with $q=\Pi_{\omega}(y)$ is exponentially stable that is,
										\begin{equation*}
											\norm{e^{t(\mc A+\mc B \Pi_{\omega})}x}_{\mc H}\leq C e^{\left(-2\omega+g(-\mc A)\right) t} \norm{x}_{\mc H}, \forall x \in \mc H,\end{equation*} where $C$ is a positive constant.
									\end{itemize}
								\end{theorem}
								We utilize \Cref{thm} to prove the complete exponential stabilization of the linearized compressible Navier-Stokes equation with Maxwell's law \eqref{nmaxeq3 bd}-\eqref{density} with boundary feedback law.
								\subsection{Control in density}
								To apply the method mentioned above we need to show that all the assumptions (H1)-(H4) hold true for our system \eqref{nmaxeq3 bd den}. 
								Let us recall the corresponding spatial operator $(\mc A, \mc D(\mc A; \mathcal Z_{m,m})$, where $\mathcal Z_{m,m}=\dot L^2(0, 2\pi)\times \dot L^2(0, 2\pi)\times \dot L^2(0, 2\pi)$.
								
								It can be check that the operator $(\mc A, \mc D(\mc A; \mathcal Z_{m,m})$ generates a strongly continuous group $\{\mathbb{T}_t\}_{t\in \rea}$ of continuous linear operator. Hence (H1) holds. Also (H2) is true, see \eqref{nmaxeqcontr bd}. We get (H3) and (H4) by Ingham inequality.
								\subsection{Feedback control law and proof of \Cref{nmaxthm_pos st 1}}
								In this section we employ Urquiza's method to construct the feedback law for our system \eqref{nmaxeq3 bd den}. Let us take $ U^1_0=(\sigma^1_0, v^1_0, \tilde{S}^1_0), U^2_0=(\sigma^1_0, v^1_0, \tilde{S}^1_0) \in \mathcal Z_{m,m}.$ We consider the bilinear form
								\begin{equation*}
									a_w(U^1_0, U^2_0)=\int_{0}^{\infty}e^{-2\omega t}\left(b u_s {\sigma^1(t, 2\pi)}+b\rho_s{v^1(t,2\pi)}\right) \overline{\left(b u_s {\sigma^2(t, 2\pi)}+b\rho_s{v^2(t,2\pi)}\right) }\, \rd t,
								\end{equation*}
								where $U^1=(\sigma^1, v^1, \tilde{S}^1)$ and $U^2=(\sigma^2, v^2, \tilde{S}^2)$ are the solutions of the following
								systems respectively
								\begin{equation}\label{eq:adj periodic1}
									\begin{cases}
										\partial_t\sigma^1 +u_s\pa_x\sigma^1+{\rho_s}\pa_x v^1 =0 & \mbox{ in } (0,T) \times (0, 2\pi), \\
										\partial_t v^1+b \partial_{x} \sigma^1+u_s\pa_x v^1 - \frac{1}{\rho_s} \partial_x\tilde{S}^1= 0 &  \mbox{ in } (0,T) \times (0, 2\pi),\\
										\partial_t\tilde{S}^1-\frac{1}{\kappa}\tilde{S}^1 - \frac{\mu}{\kappa} \partial_x v^1= 0 &  \mbox{ in } (0,T) \times (0, 2\pi),\\
										\sigma^1(t,0) = \sigma^1(t, 2\pi), v^1(t,0) = v^1(t, 2\pi),\: \tilde S^1(t,0)=\tilde S^1(t,2\pi),&  \mbox{ in } (0,T), \\
										\left(\sigma^1(T, x), v^1(T,x), \tilde{S}^1(T,x)\right) = \mathbb{T}_{-T} \left(\sigma^1_0(x), v^1_0(x), \tilde{S}^1_0(x)\right) & x \in (0, 2\pi),
									\end{cases}
								\end{equation} and\begin{equation}\label{eq:adj periodic2}
									\begin{cases}
										\partial_t\sigma^2 +u_s\pa_x\sigma^2+{\rho_s}\pa_x v^2 =0 & \mbox{ in } (0,T) \times (0, 2\pi), \\
										\partial_t v^2+b \partial_{x} \sigma^2+u_s\pa_x v^2 - \frac{1}{\rho_s} \partial_x\tilde{S}^2= 0 &  \mbox{ in } (0,T) \times (0, 2\pi),\\
										\partial_t\tilde{S}^2-\frac{1}{\kappa}\tilde{S}^2 - \frac{\mu}{\kappa} \partial_x v^2= 0 &  \mbox{ in } (0,T) \times (0, 2\pi),\\
										\sigma^2(t,0) = \sigma^2(t, 2\pi), v^2(t,0) = v^2(t, 2\pi),\: \tilde S^2(t,0)=\tilde S^2(t,2\pi),&  \mbox{ in } (0,T), \\
										\left(\sigma^2(T, x), v^2(T,x), \tilde{S}^2(T,x)\right) = \mathbb{T}_{-T} \left(\sigma^2_0(x), v^2_0(x), \tilde{S}^2_0(x)\right) & x \in (0, 2\pi).
									\end{cases}
								\end{equation}
								Let us define the operator $\Lambda_{\omega}:\mathcal Z_{m,m}\to \mathcal Z_{m,m}$ satisfying the following
								\begin{equation}\label{lambda}
									\ip{\Lambda_{\omega} U^1_0 } {U^2_0}_{\mathcal Z_{m,m}}=a_{\omega}(U^1_0 ,U^2_0), \,  U^1_0,U^2_0 \in \mathcal Z_{m,m}.
								\end{equation}
								Next we see that \begin{align*}
									\nonumber a_w(U^1_0 ,U^2_0)=&\int_{0}^{\infty}e^{-2\omega t}\left(b u_s {\sigma^1(t, 2\pi)}+b \rho_s{v^1(t,2\pi)}\right)\,\, \overline{\left(b u_s {\sigma^2(t, 2\pi)}+b \rho_s{v^2(t,2\pi)}\right)} \,\rd t\\
									\nonumber	=&\int_{0}^{\infty}e^{-2\omega t}{\mc B^*_{\rho}} U^1(t) \,\,  \overline{{\mc B^*_{\rho}} U^2(t)} \, \rd t\\
									=&\int_{0}^{\infty}e^{-2\omega t}\left({\mc B^*_{\rho}} \mathbb{T}^*_{T-t}\mathbb{T}^*_{-T}U^1_0\right) \overline{\left({\mc B^*_{\rho}}\mathbb{T}^*_{T-t}\mathbb{T}^*_{-T}U^2_0\right)}\, \rd t\\
									=&\int_{0}^{\infty}e^{-2\omega t}\left({\mc B^*_{\rho}}\mathbb{T}^*_{-t}U^1_0\right) \overline{\left({\mc B^*_{\rho}}\mathbb{T}^*_{-t} U^2_0\right)}\, \rd t.
								\end{align*}
								Therefore from \eqref{lambda}, we have
								\begin{equation}\label{lambda 1}
									\ip{\Lambda_{\omega} U^1_0 } {U^2_0}_{\mathcal Z_{m,m}}=\int_{0}^{\infty}e^{-2\omega t}\ip{{\mc B^*_{\rho}}\mathbb{T}^*_{-t}U^1_0}{\, {\mc B^*_{\rho}}\mathbb{T}^*_{-t} U^2_0}_{\cplx} \, \rd t.
								\end{equation}
								Thanks to Theorem \ref{thm}, the operator $\Lambda_{\omega}$ defined by \eqref{lambda} is coercive and isomorphism. Finally, let us define the operator ${\Pi}_{\omega}: \mathcal Z_{m,m} \to \mathbb{C}$ by
								\begin{align*}
									{\Pi}_{\omega}(\mathbf{z})=-\left(b u_s {\sigma^1_0( 2\pi)}+b\rho_s{v^1_0(2\pi)}\right) ,	\end{align*} where $U^1_0=(\sigma^1_0, v^1_0, \tilde{S}^1_0)$ is the solution of the following Lax-Milgram  problem
								\begin{equation*}
									a_{\omega}(U^1_0, U^2_0)=\ip{\mathbf{z}}{U^2_0}, \, \forall \, U^2_0 \in \mathcal Z_{m,m}.
								\end{equation*}
								Hence we obtain $\ip{\Lambda_{\omega} U^1_0 } {U^2_0}_{\mathcal Z_{m,m}}=\ip{\mathbf{z}}{U^2_0}, \forall \, U^2_0 \in \mathcal Z_{m,m}.$ This gives $\Lambda_{\omega}U^1_0=\mathbf{z}.$ It follows that $U^1_0=\Lambda_{\omega}^{-1} \mathbf{z}.$ Thus we have ${\Pi}_{\omega}=-{\mc B}^*\Lambda_{\omega}^{-1}.$
								Thanks to Theorem \ref{thm}, rapid exponential stabilization for the system \eqref{nmaxeq3 bd den} is established by means of the feedback law $q(t)={\Pi}_{\omega}(\rho(t,\cdot), v(t,\cdot), S(t,\cdot))$. More precisely, we get a positive constant $C$ such that the solution of \eqref{nmaxeq3 bd den} satisfies the following estimate
								{\small
									\begin{align}\label{rho}
										\norm{\rho(t, \cdot )}_{L^2(0,2\pi)}+\norm{u(t, \cdot )}_{L^2(0,2\pi)}+\norm{S(t, \cdot )}_{L^2(0,2\pi)}& \leq C e^{-\nu t} \bigg[ \norm{\rho_0}_{L^2(0,2\pi)}+\norm{u_0}_{L^2(0,2\pi)}+\norm{S_0}_{L^2(0,2\pi)} \bigg].
									\end{align}
								}
								A similar process will give the complete stabilization for the control systems \eqref{nmaxeq3 bd}-\eqref{velocity} and \eqref{nmaxeq3 bd}-\eqref{stress}.
								
								\section{Construction of the biorthogonal family and Ingham-type inequality}\label{biorthogonal}
								This section is devoted to the proof of a suitable Ingham-type inequality which essentially helps to derive the required observability inequality. The proof of the Ingham-type inequality relies on the construction of a biorthogonal family of $\{e^{-{\lambda_k^l}t}\}_{k\in \mathbb{Z}^*, l\in \{1,2,3\}}$.  All the constants used in this section while finding the relevant estimates are generic, which may vary from line to line.

								Let us assume \eqref{nmaxsimpleeigen}.  We recall the expressions of the eigenvalues of $\mc A$ in the following manner.
								\begin{equation}\label{nmaxasmlambdanew}
									\left.
									\begin{aligned}
										\lambda_n^1=&-\omega_1+i \beta_1n + O\left( \frac{1}{|n|}\right), & n\in \mathbbm{Z}^*,\\
										\lambda_n^2=&-\omega_2+i \beta_2n + O\left( \frac{1}{|n|}\right),  & n\in \mathbbm{Z}^*,\\
										\lambda_n^3=&-\omega_3+i \beta_3n + O\left( \frac{1}{|n|}\right)  & n\in \mathbbm{Z}^*.
									\end{aligned}
									\right\}
								\end{equation}
								We further denote the following notation:
								\begin{equation*}
									\Sigma=\{(n,j): n\in \z^*, 1\leq j\leq 3\}.
								\end{equation*}
								\begin{lem}[Gap properties of the spectrum]\label{gap}
									For any $(n,j), (k,l)\in \Sigma$, there exists $\hat{C}>0$ depending on $\omega_1, \omega_2, \omega_3, \beta_1, \beta_2, \beta_3$, such that 
									\begin{equation}\label{gap con}
										\abs{\lambda_n^j-\lambda_k^l}\geq \hat{C}.
									\end{equation}
								\end{lem}  
								\begin{proof}
									Note that $\omega_1, \omega_2, \omega_3$ are distinct and $\beta_1, \beta_2, \beta_3$ are also distinct. Thanks to the asymptotic behavior of the eigenvalues \eqref{nmaxasmlambdanew}, we have a large $N\in \N$ such that for $(n,j), (k,j)\in \Sigma, \text{ with } n\neq k$ and $|n|, |k|>N$,
									\begin{equation*}
										\abs{\lambda_n^j-\lambda_k^j}\geq \abs{\Im({\lambda_n^j-\lambda_k^j})}\geq C \abs{\beta_j(n-k)}\geq C |\beta_j|.
									\end{equation*}
									Analogously, 
									for $(n,j), (k,l)\in \Sigma, j\neq l$, $|n|, |k|>N$, we have
									\begin{equation*}
										\abs{\lambda_n^j-\lambda_k^l}\geq \abs{\Re({\lambda_n^j-\lambda_k^l})}\geq C\abs{\omega_j-\omega_l}.
									\end{equation*}
									As the eigenvalues are simple, we also find the gap condition \eqref{gap con} for finite number of the lower frequencies. 
								\end{proof}
								Let us first define the exponential type and sine type functions.
								\begin{defn}[Entire functions of exponential type]
									An entire function $f$ is said to be of \textbf{exponential type $A$} if there exist positive constants $A, B$ such that
									\begin{align*}
										|f(z)|\leq B e^{A|z|}, \, z \in \cplx,
									\end{align*}
									and of \textbf{exponential type at most $A$}, if for any $\epsilon>0$, there exits $B_\epsilon >0$ such that
									\begin{align*}
										|f(z)|\leq B_\epsilon\, e^{(A+\epsilon)|z|}, \, z \in \cplx.
									\end{align*}
								\end{defn}
								\begin{defn}(Sine type function)
									An entire function $f$ of exponential type $\pi$ is said to be of sine type if
									\begin{itemize}
										
										\item the zeros of $f(z)$, say {$\mu_k$} satisfy the gap condition, i.e., there exists $\delta>0$ such that $|\mu_k-\mu_l|>\delta$ for $k \neq l$, and
										\item there exist positive constants $C_1$,$C_2$ and $C_3$ such that
										$$C_1e^{\pi|y|}\leq |f(x+iy)| \leq C_2e^{\pi|y|},\,\forall \,x,y \in \rea \text{ with }|y|\geq C_3. $$
									\end{itemize}
								\end{defn}
								The following proposition states some important properties of sine-type functions {(see Remark and Lemma 2, p.164 of \cite{Levin} and Lemma 2, p.172 of \cite{RY} for details)}:
								\begin{prop}
									Let $f$ be a sine type function, and let $\{\mu_k\}_{k\in \mathcal{I}}$ with $\mathcal{I}\subset\z$ be its sequence of zeros. Then, we have:
									\begin{itemize}
										\item for any $\epsilon>0$, there exist constants $K_{\epsilon},\tl K_{\epsilon}>0$ such that
										$$K_{\epsilon}e^{\pi|y|}\leq |f(x+iy)|\leq \tl K_{\epsilon}e^{\pi|y|}, \text{ if dist$\left(x+iy,\{\mu_k\}\right)>\epsilon$},$$
										\item there exist some constants $K_1,K_2>0$ such that
										$$K_1<|f'(\mu_k)|<K_2, \quad \forall k\in\mathcal{I}.$$
									\end{itemize}
								\end{prop} 
								The main goal of this section is to find a class of entire functions $\mathcal{G}=\{\Psi_k^j\}_{(k,j)\in \Sigma }$ with the following properties:
								\begin{itemize}
									\item[1.]	The family $\mathcal{G}$ contains entire functions of exponential type $\frac{T}{2}$. That means there exists a positive constant $C$ such that
									\begin{align}\label{exp typ}
										\abs{\Psi_k^j(z)}\leq C e^{\frac{T}{2}\abs{z}}, \quad \forall z\in \cplx, \,\, (k,j)\in \Sigma. 
									\end{align}
									\item[2.] All the members of $\mathcal{G}$ are square-integrable on the real line, i.e.,
									\begin{align}\label{sq intg1}
										\int_{\rea}\abs{\Psi_k^j(x)}^2\, \rd x<\infty. 
									\end{align}
								\end{itemize}
								
								Let us now state the celebrated Paley-Wiener theorem, from which we can conclude about the desired biorthogonal family using the family $\mathcal{G}$.
								\begin{theorem}[\textbf{Paley-Wiener}]\cite[Theorem 1.8]{RY}\label{Paley-Wiener}
									Let $f$ be an entire function of exponential type $A$ and suppose
									\begin{align*}\int_{-\infty}^{\infty}|f(x)|^2 \, \rd x < \infty.
									\end{align*}
									Then there exists a function $\phi \in L^2(-A, A)$ with the following representation
									\begin{align*}
										f(z)=\int_{-A}^{A}e^{izt}\phi(t) \, \rd t,\, z\in \cplx.
									\end{align*}
								\end{theorem}
								
								Thus if we have the existence of such family $\mathcal{G}$ of entire functions satisfying \eqref{exp typ}-\eqref{sq intg1}, then  applying the Paley-Wiener theorem for the same, one can get a family of functions $\mathcal{J}=\{\Theta_k^j,\}_{(k,j)\in \Sigma}$ supported in $[-\frac{T}{2}, \frac{T}{2}]$, such that the following representation holds
								\begin{align}\label{representation p}
									\Psi_k^j(z)=\int_{-\frac{T}{2}}^{\frac{T}{2}}e^{izt}\Theta_k^j(t) dt,\, z\in \cplx.
								\end{align}
								Clearly, $\Theta_k^j$ are the Fourier transformations of $\Psi_k^j$ respectively, for $ (k,j)\in \Sigma.$
								Also, by Plancharel's Theorem, we have 
								\begin{align*}
									\int_{\rea}\abs{\Psi_k^j(x)}^2=2\pi\int_{-\frac{T}{2}}^{\frac{T}{2}}\abs{\Theta_k^j(t)}^2\, \rd t.
								\end{align*}
								Now we are in the position of formulating the construction of the family $\mathcal{G}$ satisfying \eqref{exp typ}-\eqref{sq intg1}. 
								Let us first introduce the following entire function, which has simple zeros exactly at $i\overline{\lambda_k^j}:$
								\begin{align}\label{eq:P}
									P(z)=z^3\prod_{(k,j)\in \Sigma}\left(1-\frac{z}{i\overline{\lambda_k^j}}\right).
								\end{align}
								\begin{prop}
									Let $P$ be the canonical product defined in \eqref{eq:P}. Then $P$ is an entire function of exponential type $\pi\left(\frac{1}{|\beta_1|}+\frac{1}{|\beta_2|}+\frac{1}{|\beta_3|}\right)$, which satisfies the following properties:
									\begin{itemize}
										\item There exists a positive constant $C>0$ such that
										\begin{equation}\label{bound for p}
											|P(x)|\leq C, \quad \forall x \in \rea.
										\end{equation}
										\item  There exists  constant $C_2>0$ such that 
										\begin{equation}\label{p prime}
											\left|P'\left(i\overline{\lambda_k^j}\right)\right|\geq C_2, \forall (k,j) \in \Sigma.
										\end{equation}
									\end{itemize}
								\end{prop}
								\begin{proof} Let us denote $\nu_n^j=\frac{1}{\beta_j}\lambda_n^j.$ thus, we have:
									\begin{equation}\label{nmaxasmlambda 1}
										\begin{cases}
											\nu_n^1=-\frac{\omega_1}{\beta_1}+i n + O\left( \frac{1}{|n|}\right),\\
											\nu_n^2=-\frac{\omega_2}{\beta_2}+i n + O\left( \frac{1}{|n|}\right),\\
											\nu_n^3=-\frac{\omega_3}{\beta_3}+in + O\left( \frac{1}{|n|}\right).
										\end{cases}
									\end{equation}
									Let us denote the products $P_j, j\in \{1,2,3\}$ by
									\begin{equation}\label{pj}
										P_j(z)=z\prod_{k\in \z^*}\left(1-\frac{z}{i\overline{\nu_k^j}}\right).
									\end{equation}
									Thus the canonical product $P$ can be written in the following form 
									\begin{equation}\label{p pj}
										P(z)=\prod_{j\in \{1,2,3\}}\beta_j P_j\left(\frac{z}{\beta_j}\right).
									\end{equation}
									\begin{lem}[{see \cite{RY}, \cite{LR14}}]\label{Rosier}
										Let $\Lambda_k=k+d_k$, where $d_k=d+ O(k^{-1}),$
										as $\abs{k} \to \infty$ for some constant $d \in \cplx$, and that $\Lambda_k\neq\Lambda_l$ for $k \neq l.$ Then $f(z)=z\prod_{k\in \z^*} \left(1-\frac{z}{ \Lambda_k}\right)$ is an entire function of sine type.
									\end{lem}
									Note that, $i\overline{\nu_k^j}=-i\frac{\omega_j}{\beta_j}+ n + O\left( \frac{1}{|n|}\right), j\in \{1,2,3\}$. Also $i\overline{\nu_k^j}\neq i\overline{\nu_l^j},$ for $ l\neq k$.  Thanks to \Cref{Rosier}, each $P_j$ is a sine type function. Moreover, each of these functions is an entire function of exponential type $\pi$. Hence, $P$ is an entire function of the exponential type $\pi\left(\frac{1}{|\beta_1|}+\frac{1}{|\beta_2|}+\frac{1}{|\beta_3|}\right)$.

									\noindent
									Thanks to the definition of sine-type function, for any $\epsilon>0$ there exist positive constants $C_1, C_2, C_3, C_4,$ $ C_5$, where $C_2$ and $C_3$ depend on $\epsilon$, such that
									\begin{align}
										\label{bound}
										\abs{P_j(z)}\leq C_1^je^{\pi\abs{z}},&\:\: z \in \mathbb{C},\quad j \in \{1,2,3\},\\
										\label{bound1}C_2^je^{\pi\abs{y}}\leq \abs{P_j(x+iy)} \leq C_3^je^{\pi \abs{y}},\:\:& \text{if } dist\left(x+iy,\{i\overline{\nu_j^k}\}\right) > \epsilon, \\
										\label{bound2}C_4^j<\abs{P_j'\left(i\overline{\nu_k^j}\right)}<C_5^j,&\:\: \forall (k,j) \in \Sigma.
									\end{align}
									Now using the above inequality \eqref{bound1} and continuity of $P_j$, we get a positive constant $C^j$ such that
									\begin{equation}\label{eq:p1real}
										\abs{P_j(x)}\leq C^j, \quad \forall\, x\in\rea.
									\end{equation}
									Therefore the estimate \eqref{bound for p} holds.
									
									Using \eqref{p pj} and \eqref{bound2} we have:
									\begin{align}\label{p prime es} 
										\abs{P'\left(i\overline{\lambda_n^j}\right)}=\abs{P_j'\left(i\overline{\nu_n^j}\right)\prod_{\substack{l\in \{1,2,3\} \\ l\neq j}} \beta_l P_l\left(\frac{i\overline{\lambda_n^j}}{\beta_l}\right) }\geq \min_{1\leq j\leq 3} \{C_4^j\} \prod_{\substack{l\in \{1,2,3\} \\ l\neq j}}
										\abs{\beta_l P_l\left(\frac{i\overline{\lambda_n^j}}{\beta_l}\right)}.
									\end{align}
									Therefore to prove the estimate \eqref{p prime}, we need to compute the estimate of the quantities $P_l\left(\frac{i\overline{\lambda_n^j}}{\beta_l}\right), l\neq j.$ Thanks to \Cref{gap}, we have 
									\begin{align*}
										\abs{\frac{i\overline{\lambda_n^j}}{\beta_l}-i\overline{\nu_k^l}}=\frac{1}{|\beta_l|}\abs{{\overline{\lambda_n^j}}-\beta_l\overline{\nu_k^l}}=\frac{1}{|\beta_l|}\abs{{\overline{\lambda_n^j}}-\overline{\lambda_k^l}}> \hat{C}.
									\end{align*}
									Therefore using the estimate \eqref{bound1}, we finally have the estimate
									\begin{equation}\label{p prime est1}
										\abs{P_l\left(\frac{i\overline{\lambda_n^j}}{\beta_l}\right)}\geq C_2^l e^{\pi\abs{\frac{\omega_j}{\beta_l}}}.
									\end{equation}
									Hence \eqref{p prime es} along with \eqref{p prime est1} provide the estimate \eqref{p prime}.
								\end{proof}
								\begin{theorem}\label{biorthothm}
									Let $T> 2\pi\left(\frac{1}{|\beta_1|}+\frac{1}{|\beta_2|}+\frac{1}{|\beta_3|}\right).$ Then there exists a family $\{\Theta_n^j\}_{(n,j)\in \Sigma}$ which is biorthogonal to the family of exponentials $\{e^{-{\lambda_k^l}t}\}_{(k,l)\in \Sigma}$ in $L^2\left(-\frac{T}{2}, \frac{T}{2}\right)$, i.e.,
									\begin{equation}\label{bior}
										\int_{-\frac{T}{2}}^{\frac{T}{2}}   \Theta_n^j(t) e^{-\overline{\lambda_k^l}t}\, \rd t= \delta_{nk}\delta_{jl}.
									\end{equation}
									Moreover, there exists a positive constant $C>0$ such that the following estimate holds:
									\begin{equation}\label{ingham1}
										\norm{\sum_{(n,j)\in \Sigma} a_n^j \Theta_n^j}_{L^2\left(-\frac{T}{2}, \frac{T}{2}\right)}^2 \leq C \sum_{(n,j)\in \Sigma} |a_n^j|^2 ,
									\end{equation}       for any finite sequence of complex numbers $\{ a_n^j\}_{(n,j)\in \Sigma}$.
								\end{theorem}
								\begin{proof}
									Let us first define the entire function 
									\begin{equation*}
										\Psi_n^j(z):=\frac{P(z)}{(z-i\overline{\lambda_n^j}) P'\left(i\overline{\lambda_n^j}\right)}.
									\end{equation*}
									Clearly, $\Psi_n^j$ is a collection of entire functions of exponential type at most $\pi\left(\frac{1}{|\beta_1|}+\frac{1}{|\beta_2|}+\frac{1}{|\beta_3|}\right).$ Moreover, it is easy to check that 
									\begin{align}\label{sq intg}
										\int_{\rea}\abs{\Psi_n^j(x)}^2 \, \rd x< C. 
									\end{align}
									Therefore by Paley-Wiener Theorem (\Cref{Paley-Wiener}), there exists a collection of functions $\{\Theta_n^j\}_{(n,j)\in \Sigma}$ supported in $[-\frac{T'}{2}, \frac{T'}{2}]$, with $T'=2\pi\left(\frac{1}{|\beta_1|}+\frac{1}{|\beta_2|}+\frac{1}{|\beta_3|}\right)$ such that the following relation holds 
									\begin{align}\label{representation}
										\Psi_n^j(z)=\int_{-\frac{T'}{2}}^{\frac{T'}{2}}e^{izt}\Theta_n^j(t) \, \rd t,\, z\in \cplx.
									\end{align}
									Also, we can deduce
									\begin{equation*}\Psi_n^j(i\overline{\lambda_k^l})=\delta_{nk} \delta_{jl}, \,\, (n,j), (k,l)\in \Sigma.
									\end{equation*}
									Thus we have the family $\{\Theta_n^j\}_{(n,j)\in \Sigma}$ which is biorthogonal to the family of exponentials $\{e^{-{\lambda_k^l}t}\}_{(k,l)\in \Sigma}$ in $L^2\left(-\frac{T'}{2}, \frac{T'}{2}\right)$ and by Plancharel's Theorem and \eqref{sq intg}, we have 
									\begin{equation*}
										2\pi\int_{-\frac{T'}{2}}^{\frac{T'}{2}}\abs{\Theta_k^j(t)}^2\, \rd t=\int_{\rea}\abs{\Psi_k^j(x)}^2 \, \rd x\leq C.
									\end{equation*}
									Then using \cite[Proposition 8.3.9]{TW09}, we have for any $T>T'$, there exists a biorthogonal family  $\{\Theta_n^j\}_{(n,j)\in \Sigma}$ of $\{e^{-{\lambda_k^l}t}\}_{(k,l)\in \Sigma}$ in $L^2\left(-\frac{T}{2}, \frac{T}{2}\right)$ such that the estimate \eqref{ingham1} holds. 
								\end{proof}
								The following corollary is an immediate consequence of the above theorem which will essentially give us the required Ingham-type inequality.
								\begin{cor}
									Let us assume $T> 2\pi\left(\frac{1}{|\beta_1|}+\frac{1}{|\beta_2|}+\frac{1}{|\beta_3|}\right).$ Then for any finite sequence of scalars $\{ a_n^j\}_{(n,j)\in \Sigma}$, there exist two positive constants $C_1, C_2$ such that the following inequality holds:
									\begin{equation}\label{ingham2}
										C_1\sum_{(n,j)\in \Sigma} |a_n^j|^2\leq \norm{\sum_{(n,j)\in \Sigma} a_n^j e^{-\overline{\lambda_n^j}t}}_{L^2\left(-\frac{T}{2}, \frac{T}{2}\right)}^2\leq C_2 \sum_{(n,j)\in \Sigma} |a_n^j|^2.
									\end{equation}
								\end{cor}       
								\begin{proof}
									Using \eqref{bior}, we can deduce
									\begin{align}\label{ing 11}
										\nonumber
										\sum_{(n,j)\in \Sigma} |a_n^j|^2 =&\int_{-\frac{T}{2}}^{\frac{T}{2}}{\left(\sum_{(n,j)\in \Sigma} a_n^j e^{-\overline{\lambda_n^j}t}\right)}\left(\sum_{(k,l)\in \Sigma} \overline{a_k^l} \Theta_k^l(t) \right)\, \rd t \\
										\leq & \norm{\sum_{(n,j)\in \Sigma} a_n^j e^{-\overline{\lambda_n^j}t}}_{L^2\left(-\frac{T}{2}, \frac{T}{2}\right)}\norm{\sum_{(n,j)\in \Sigma} \overline{a_k^l} \Theta_k^l}_{L^2\left(-\frac{T}{2}, \frac{T}{2}\right)}.
									\end{align}
									The inequality \eqref{ing 11} along with \eqref{ingham1} establish the left inequality of \eqref{ingham2}.
									
									Straightforward estimate provides the right hand inequality. Indeed,
									\begin{align*}
										\norm{\sum_{(n,j)\in \Sigma} a_n^j e^{-\overline{\lambda_n^j}t}}_{L^2\left(-\frac{T}{2}, \frac{T}{2}\right)}^2\leq C_2\sum_{j=1}^{3}\norm{\sum_{n\in \z^*} a_n^j e^{-\overline{\lambda_n^j}t}}_{L^2\left(-\frac{T}{2}, \frac{T}{2}\right)}^2.
									\end{align*}
									Next, utilizing the Ingham-type inequality \cite[Proposition 8.1]{ahamed}, one can get the right hand inequality.\end{proof}

								{
									
									\section{Some details on multiple eigenvalues}\label{secmultiev}
									
									Recall from \Cref{rem multi} that \(\mathcal{A}\) can have multiple eigenvalues. In this section, we will demonstrate that results analogous to \Cref{nmaxbiorthonormality} and \Cref{nmaxbasis_representation} also apply in this context. This will enable us to extend our analysis to prove \Cref{nmaxthm_pos}, \Cref{thm_pos}, and \Cref{nmaxthm_pos bd}. We will specifically focus on the case where \(\mathcal{A}\) has double eigenvalues. Using a similar approach, we can also handle the case when \(\mathcal{A}\) has triple eigenvalues.
									
									\begin{prop}\label{condndoubleroot}
										For each $n\in \mathbb{N}$, let us denote
										\begin{equation}\label{defnq}
											q_n(\lambda)=-b+\frac{  \left( {\lambda}+inu_s\right)^2 }{\rho_sn^2} - \frac{\mu\kappa\left( {\lambda}+inu_s\right)^2 }{\rho_s^2\left(  1+\kappa{\lambda}\right)^2}.
										\end{equation}
										Then $\lambda$ is a double root of the characteristic equation \eqref{nmaxchar.poly} if and only if $q_n(\lambda)=0$.
									\end{prop}
									From the relation between roots and coefficients in \eqref{nmaxrc1}-\eqref{nmaxrc6}, the above result follows by a direct calculation. 
									
									\vspace{2mm}
									
									Let for some $n\in \mathbb{Z}^*$, $\mathcal{A}$ has eigenvalues $\lambda_n^1, \lambda_n^2, \lambda_n^3$ with $\lambda_n^2=\lambda_n^3 \left( \text{suppose }\lambda_n^2 \text{ is a double eigenvalue}\right) $. \Cref{condndoubleroot} gives that $q_n(\lambda_n^2)=0$.

									Now we consider $\xi_{n,1}$, an eigenfunction of $\mathcal{A}$ corresponding to the eigenvalue $\lambda_n^1$, $\xi_{n,2}$, an eigenfunction of $\mathcal{A}$ corresponding to the eigenvalue $\lambda_n^2$ and $\xi_{n,3}$, a generalized eigenfunction of $\mathcal{A}$	corresponding to the eigenvalue $\lambda_n^2$ :
									\begin{align*}
										&	\xi_{n,1}=\frac{1}{\theta_{n,1}}\begin{pmatrix}
											-1\vspace{1mm}\\\frac{\lambda_n^1+inu_s}{{in\rho_s}}\vspace{1mm}\\\frac{\mu\left( \lambda_n^1+inu_s\right) }{\rho_s\left( 1+\kappa\lambda_n^1\right) }
										\end{pmatrix}e^{inx},\quad 	\xi_{n,2}=\frac{1}{\theta_{n,2}}\begin{pmatrix}
											-1\vspace{1mm}\\\frac{\lambda_n^2+inu_s}{{in\rho_s}}\vspace{1mm}\\\frac{\mu\left( \lambda_n^2+inu_s\right) }{\rho_s\left( 1+\kappa\lambda_n^2\right) }
										\end{pmatrix}e^{inx},\\
										& \xi_{n,3}=\frac{1}{\theta_{n,2}} \begin{pmatrix}
											c_n\vspace{1mm}\\-\frac{1+c_n\left( \lambda_n^2+inu_s\right) }{in\rho_s }\vspace{1mm}\\\frac{\mu\kappa\left( \lambda_n^2+inu_s\right)-\mu\left( 1+\kappa\lambda_n^2\right)\left( 1+c_n\left( \lambda_n^2+inu_s\right)\right) }{\rho_s\left( 1+\kappa\lambda_n^2\right)^2 }
										\end{pmatrix}e^{inx},
									\end{align*}
									where $\theta_{n,1}$ and $\theta_{n,2}$ are as same as in \eqref{nmaxequ-theta_nl}. We can see that $\left(\lambda_n^2\mathcal{I}-\mathcal{A} \right)\xi_{n,3}=\xi_{n,2} $.
									
									Similarly, we consider $\xi^*_{n,1}$, an eigenfunction of $\mathcal{A}^*$ corresponding to the eigenvalue $\overline{\lambda_n^1}$, $\xi^*_{n,3}$, an eigenfunction of $\mathcal{A}^*$ corresponding to the eigenvalue $\overline{\lambda_n^2}$ and $\xi^*_{n,2}$, a generalized eigenfunction of $\mathcal{A}^*$	corresponding to the eigenvalue $\overline{\lambda_n^2}$ :
									
									\begin{align}
										&\xi^*_{n,1}=\frac{1}{\psi_{n,1}}\begin{pmatrix}
											1\vspace{1mm}\\\frac{\overline{\lambda_n^1}-inu_s}{{in\rho_s}}\vspace{1mm}\\-\frac{\mu\left( \overline{\lambda_n^1}-inu_s\right) }{\rho_s\left( 1+\kappa\overline{\lambda_n^1}\right) }
										\end{pmatrix}e^{inx},\quad \xi^*_{n,2}=\frac{1}{\psi_{n,2}}\begin{pmatrix}
											\frac{1}{\overline{\lambda_n^2}-inu_s}\vspace{1mm}\\0\vspace{1mm}\\-\frac{\mu\kappa\left( \overline{\lambda_n^2}-inu_s\right) }{\rho_s\left( 1+\kappa\overline{\lambda_n^2}\right)^2 }
										\end{pmatrix}e^{inx},\label{gen1}\\
										& \xi^*_{n,3}=\frac{1}{\psi_{n,2}}\begin{pmatrix}
											1\vspace{1mm}\\\frac{\overline{\lambda_n^2}-inu_s}{{in\rho_s}}\vspace{1mm}\\-\frac{\mu\left( \overline{\lambda_n^2}-inu_s\right) }{\rho_s\left( 1+\kappa\overline{\lambda_n^2}\right) }
										\end{pmatrix}e^{inx},\label{gen2}
									\end{align}
									where $\psi_{n,1}$ is as same as in \eqref{nmaxequ-psi_nl} and $\psi_{n,2}\neq 0$ is given by 
									\begin{equation}\label{newpsi2}
										\psi_{n,2}=\frac{1}{\theta_{n,2}}\left[ -2\pi\left(\frac{b}{\overline{\lambda_n^2}-inu_s}+\frac{\mu\kappa\left( \overline{\lambda_n^2}-inu_s\right)^2 }{\rho_s^2\left( 1+\kappa\overline{\lambda_n^2}\right)^3 } \right) \right].
									\end{equation}
									We can see that $\left(\overline{\lambda_n^2}\mathcal{I}-\mathcal{A}^* \right)\xi^*_{n,2}=\xi^*_{n,3} $.
									
									The families $\left\lbrace \xi_{n,l}\: |\: 1\leq l\leq 3, n\in\mathbb{N}\right\rbrace$ and $\left\lbrace \xi^*_{n,l}\: |\: 1\leq l\leq 3, n\in\mathbb{N}\right\rbrace$ satisfy the bi-orthonormality relations \eqref{nmaxbiortho} and \Cref{nmaxbasis_representation} holds.

									Since in the case of double eigenvalues of $\mathcal{A}$, we obtain the result analogous to \Cref{nmaxbasis_representation} using the generalized eigenvectors of $\mathcal{A}$ and $\mathcal{A}^*$, the proof of the null controllability result can be adapted to the case when $\mathcal{A}$ has double eigenvalues. The essential  modifications of the proof of the controllability results to this case are indicated below. 
									
									\begin{rem}
										$\mathcal{A}_n$ in \Cref{lem-An} and $\mathcal{B}_n$ in \Cref{lem-Bn} are given by 
										\begin{equation}
											\mathcal{A}_n=\begin{pmatrix}
												\lambda_n^1&0&0\\
												0&\lambda_n^2&-1\\
												0&0&\lambda_n^2
											\end{pmatrix}, \quad \mathcal{B}_n= b\sqrt{2\pi}\begin{pmatrix}
												\frac{1}{{\bar\psi_{n,1}}}\\\frac{1}{{\lambda_n^2\bar\psi_{n,2}}}\\\frac{1}{{\bar\psi_{n,2}}}
											\end{pmatrix}.
										\end{equation}
										We can see that, 
										\begin{align*}
											\text{Rank}\left[\lambda I-\mathcal{A}_n \;;\; \mathcal{B}_n \right]=3, \: \text{ for }\lambda=\lambda^1_n, \lambda^2_n, \lambda^3_n. 
										\end{align*}
										Thus the finite dimensional system  \eqref{opp-eqn_proj} is controllable. 
										
										The rest of the proof of \Cref{thm_pos} in the case of multiple eigenvalues is same as given in Section \ref{nmaxsecnullcont}. 
									\end{rem}
									
									In the case of double eigenvalues, to complete the proofs of \Cref{nmaxthm_pos} and \Cref{nmaxthm_pos bd}, we need to adjust the finitely many modes in the Ingham inequality where the double eigenvalues exist. For this adjustment, it is essential that the eigenfunction \(\xi^*_{n,l}\), where \(l=1,2,3\), satisfies the unique continuation property. In the following, we will verify this property for the case of density control. Using a similar approach, we can also prove the unique continuation property when the control acts only on the velocity or stress.
									
									\noindent\textbf{Interior controllability.} {\textit{Control in density: }} Here the observation operator \(\mathcal{B}^* \in \mathcal{L}\left( \mathcal{Z}_m, L^2\left(0,2\pi \right) \right)\) is given by
									\begin{equation}\label{b*intden}
										\mathcal{B}^*\begin{pmatrix}
											\phi_1\\\phi_2\\\phi_3
										\end{pmatrix}=b\mathbbm{1}_{\mathcal{O}_{1}}\phi_1.
									\end{equation}
									From \eqref{b*intden}, \eqref{gen1}, \eqref{gen2}, \eqref{newpsi2}, and \eqref{nmaxequ-psi_nl}, it is clear that \(\mathcal{B}^*\left(\xi^*_{n,l}\right)\neq0\), for \(l=1,2,3\).
									
									\noindent\textbf{Boundary controllability.} {\textit{Control in density: }} Here the observation operator \(\mathcal{B}^* \in \mathcal{L}(\mathcal{D}(\mathcal{A}^*); \mathbb{C})\) is given by
									\begin{equation}\label{b*bddden}
										\mathcal{B}^*(\phi, \psi, \xi)=bu_s{\phi(2\pi)}+b\rho_s{\psi(2\pi)}.
									\end{equation}
									From \eqref{b*bddden}, \eqref{gen1}, \eqref{gen2}, \eqref{newpsi2}, and \eqref{nmaxequ-psi_nl}, it is clear that \(\mathcal{B}^*\left(\xi^*_{n,l}\right)\neq0\), for \(l=1,2,3\).
								}

								\section*{Acknowledgments} We would like to express sincere thanks to the anonymous reviewers whose comments and suggestions immensely helped us to improve this manuscript. The authors would like to thank  Dr. Shirshendu Chowdhury and Dr. Debanjana Mitra for fruitful discussions. Sakil Ahamed expresses his gratitude to the IISER Kolkata's Department of Mathematics \& Statistics for their hospitality and assistance during his visit. The work of Subrata Majumdar is supported by the institute post-doctoral fellowship of IIT Bombay.

								\section*{Data availability statement}
								This article describes entirely theoretical research. Thus, data sharing is not applicable to this article
								as no datasets were generated or analyzed during the current study.
								\section*{Conflict of Interest}
								{The authors declare that they have no conflicts of interest.}


\begin{thebibliography}{10}
									
									\bibitem{ahamed}
									{\sc S.~Ahamed and D.~Mitra}, {\em Some controllability results for linearized
										compressible {N}avier-{S}tokes system with {M}axwell's law}, {Submitted,
										(2022) }.
									
									\bibitem{VB06}
									{\sc V.~Barbu, I.~Lasiecka, and R.~Triggiani}, {\em Tangential boundary
										stabilization of {N}avier-{S}tokes equations}, Mem. Amer. Math. Soc., 181
									(2006), pp.~x+128.
									
									\bibitem{Beauchard}
									{\sc K.~Beauchard, A.~Koenig, and K.~Le~Balc'h}, {\em Null-controllability of
										linear parabolic transport systems}, J. \'{E}c. polytech. Math., 7 (2020),
									pp.~743--802.
									
									\bibitem{bensoussan2007representation}
									{\sc A.~Bensoussan, G.~Da~Prato, M.~C. Delfour, and S.~K. Mitter}, {\em
										Representation and control of infinite dimensional systems}, Systems \&
									Control: Foundations \& Applications, Birkh\"{a}user Boston, Inc., Boston,
									MA, second~ed., 2007.
									
									\bibitem{jiten}
									{\sc K.~Bhandari, S.~Chowdhury, R.~Dutta, and J.~Kumbhakar}, {\em Boundary
										null-controllability of 1d linearized compressible navier-stokes system by
										one control force}, 2022.
									
									\bibitem{biccari2019null}
									{\sc U.~Biccari and S.~Micu}, {\em Null-controllability properties of the wave
										equation with a second order memory term}, Journal of Differential Equations,
									267 (2019), pp.~1376--1422.
									
									\bibitem{Filho2021RapidES}
									{\sc R.~d.~A. Capistrano-Filho, E.~Cerpa, and F.~A.~Gallego}, {\em Rapid
										exponential stabilization of a boussinesq system of kdv-kdv type},
									Communications in Contemporary Mathematics,  (2021).
									
									\bibitem{CE09}
									{\sc E.~Cerpa and E.~Cr\'{e}peau}, {\em Rapid exponential stabilization for a
										linear {K}orteweg-de {V}ries equation}, Discrete Contin. Dyn. Syst. Ser. B,
									11 (2009), pp.~655--668.
									
									\bibitem{Shirshendu}
									{\sc S.~Chowdhury, R.~Dutta, and S.~Majumdar}, {\em Boundary controllability
										and stabilizability of a coupled first-order hyperbolic-elliptic system},
									Evol. Equ. Control Theory, 12 (2023), pp.~907--943.
									
									\bibitem{CM15}
									{\sc S.~Chowdhury and D.~Mitra}, {\em Null controllability of the linearized
										compressible {N}avier-{S}tokes equations using moment method}, J. Evol. Equ.,
									15 (2015), pp.~331--360.
									
									\bibitem{chowdhury2014null}
									{\sc S.~Chowdhury, D.~Mitra, M.~Ramaswamy, and M.~Renardy}, {\em Null
										controllability of the linearized compressible {N}avier {S}tokes system in
										one dimension}, J. Differential Equations, 257 (2014), pp.~3813--3849.
									
									\bibitem{chowdhury2012controllability}
									{\sc S.~Chowdhury, M.~Ramaswamy, and J.-P. Raymond}, {\em Controllability and
										stabilizability of the linearized compressible {N}avier-{S}tokes system in
										one dimension}, SIAM J. Control Optim., 50 (2012), pp.~2959--2987.
									
									\bibitem{SC00}
									{\sc S.~S. Collis, K.~Ghayour, M.~Heinkenschloss, M.~Ulbrich, and S.~Ulbrich},
									{\em Numerical solution of optimal control problems governed by the
										compressible {N}avier-{S}tokes equations}, in Optimal control of complex
									structures ({O}berwolfach, 2000), vol.~139 of Internat. Ser. Numer. Math.,
									Birkh\"{a}user, Basel, 2002, pp.~43--55.
									
									\bibitem{SC02}
									\leavevmode\vrule height 2pt depth -1.6pt width 23pt, {\em Optimal control of
										unsteady compressible viscous flows}, Internat. J. Numer. Methods Fluids, 40
									(2002), pp.~1401--1429.
									
									\bibitem{Co07}
									{\sc J.-M. Coron}, {\em Control and nonlinearity}, vol.~136 of Mathematical
									Surveys and Monographs, American Mathematical Society, Providence, RI, 2007.
									
									\bibitem{ervedoza2}
									{\sc S.~Ervedoza, O.~Glass, and S.~Guerrero}, {\em Local exact controllability
										for the two- and three-dimensional compressible {N}avier-{S}tokes equations},
									Comm. Partial Differential Equations, 41 (2016), pp.~1660--1691.
									
									\bibitem{ervedoza1}
									{\sc S.~Ervedoza, O.~Glass, S.~Guerrero, and J.-P. Puel}, {\em Local exact
										controllability for the one-dimensional compressible {N}avier-{S}tokes
										equation}, Arch. Ration. Mech. Anal., 206 (2012), pp.~189--238.
									
									\bibitem{ervedoza2018local}
									{\sc S.~Ervedoza and M.~Savel}, {\em Local boundary controllability to
										trajectories for the 1{D} compressible {N}avier {S}tokes equations}, ESAIM
									Control Optim. Calc. Var., 24 (2018), pp.~211--235.
									
									\bibitem{IL88}
									{\sc F.~Flandoli, I.~Lasiecka, and R.~Triggiani}, {\em Algebraic {R}iccati
										equations with nonsmoothing observation arising in hyperbolic and
										{E}uler-{B}ernoulli boundary control problems}, Ann. Mat. Pura Appl. (4), 153
									(1988), pp.~307--382 (1989).
									
									\bibitem{Fur04}
									{\sc A.~V. Fursikov}, {\em Stabilization for the 3{D} {N}avier-{S}tokes system
										by feedback boundary control}, vol.~10, 2004, pp.~289--314.
									\newblock Partial differential equations and applications.
									
									\bibitem{RCDC}
									{\sc G.~Harris and C.~Martin}, {\em The roots of a polynomial vary continuously
										as a function of the coefficients}, Proc. Amer. Math. Soc., 100 (1987),
									pp.~390--392.
									
									\bibitem{HuRacke}
									{\sc Y.~Hu and R.~Racke}, {\em Compressible navier-stokes equations with
										revised maxwell's law}, J. Math. Fluid Mech., 19 (2017), p.~77–90.
									
									\bibitem{hu2019global}
									{\sc Y.~Hu and N.~Wang}, {\em Global existence versus blow-up results for one
										dimensional compressible navier--stokes equations with maxwell's law},
									Mathematische Nachrichten, 292 (2019), pp.~826--840.
									
									\bibitem{ingham1936some}
									{\sc A.~E. Ingham}, {\em Some trigonometrical inequalities with applications to
										the theory of series}, Math. Z., 41 (1936), pp.~367--379.
									
									\bibitem{PLissy}
									{\sc A.~Koenig and P.~Lissy}, {\em Null-controllability of underactuated linear
										parabolic-transport systems with constant coefficients}, 2023.
									
									\bibitem{KV1}
									{\sc V.~Komornik}, {\em Exact controllability and stabilization}, RAM: Research
									in Applied Mathematics, Masson, Paris; John Wiley \& Sons, Ltd., Chichester,
									1994.
									\newblock The multiplier method.
									
									\bibitem{KV}
									{\sc V.~Komornik and P.~Loreti}, {\em Fourier series in control theory},
									Springer Monographs in Mathematics, Springer-Verlag, New York, 2005.
									
									\bibitem{jiten2}
									{\sc J.~Kumbhakar}, {\em Null controllability of one-dimensional linearized
										compressible navier-stokes system in periodic setup using one boundary
										control}, https://arxiv.org/abs/2301.04080, 2023.
									
									\bibitem{Levin}
									{\sc B.~Y. Levin}, {\em Lectures on Entire Functions}, vol.~150 of Translations
									of Mathematical Monographs, American Mathematical Society, 1996.
									
									\bibitem{LR14}
									{\sc P.~Martin, L.~Rosier, and P.~Rouchon}, {\em Null controllability of the
										structurally damped wave equation with moving control}, SIAM J. Control
									Optim., 51 (2013), pp.~660--684.
									
									\bibitem{micu2004introduction}
									{\sc S.~Micu and E.~Zuazua}, {\em An introduction to the controllability of
										partial differential equations}, Quelques questions de th{\'e}orie du
									contr{\^o}le. Sari, T., ed., Collection Travaux en Cours Hermann, to appear,
									(2004).
									
									\bibitem{MRR-17}
									{\sc D.~Mitra, M.~Ramaswamy, and M.~Renardy}, {\em Interior local null
										controllability of one-dimensional compressible flow near a constant steady
										state}, Math. Methods Appl. Sci., 40 (2017), pp.~3445--3478.
									
									\bibitem{MR-17}
									{\sc D.~Mitra and M.~Renardy}, {\em Interior local null controllability for
										multi-dimensional compressible flow near a constant state}, Nonlinear Anal.
									Real World Appl., 37 (2017), pp.~94--136.
									
									\bibitem{MN-19}
									{\sc N.~Molina}, {\em Local exact boundary controllability for the compressible
										{N}avier-{S}tokes equations}, SIAM J. Control Optim., 57 (2019),
									pp.~2152--2184.
									
									\bibitem{Raymond}
									{\sc J.-P. Raymond}, {\em Feedback boundary stabilization of the
										two-dimensional {N}avier-{S}tokes equations}, SIAM J. Control Optim., 45
									(2006), pp.~790--828.
									
									\bibitem{Raymond1}
									{\sc J.-P. Raymond}, {\em Feedback boundary stabilization of the
										three-dimensional incompressible {N}avier-{S}tokes equations}, J. Math. Pures
									Appl. (9), 87 (2007), pp.~627--669.
									
									\bibitem{TW09}
									{\sc M.~Tucsnak and G.~Weiss}, {\em Observation and control for operator
										semigroups}, Birkh\"{a}user Advanced Texts: Basler Lehrb\"{u}cher.
									[Birkh\"{a}user Advanced Texts: Basel Textbooks], Birkh\"{a}user Verlag,
									Basel, 2009.
									
									\bibitem{UR05}
									{\sc J.~M. Urquiza}, {\em Rapid exponential feedback stabilization with
										unbounded control operators}, SIAM J. Control Optim., 43 (2005),
									pp.~2233--2244.
									
									\bibitem{VK08}
									{\sc R.~Vazquez and M.~Krstic}, {\em Control of turbulent and
										magnetohydrodynamic channel flows}, Systems \& Control: Foundations \&
									Applications, Birkh\"{a}user Boston, Inc., Boston, MA, 2008.
									\newblock Boundary stabilization and state estimation.
									
									\bibitem{AV}
									{\sc A.~Vest}, {\em Rapid stabilization in a semigroup framework}, SIAM J.
									Control Optim., 51 (2013), pp.~4169--4188.
									
									\bibitem{leal2021control}
									{\sc F.~J. Vielma~Leal and A.~Pastor}, {\em Control and stabilization for the
										dispersion generalized {B}enjamin equation on the circle}, ESAIM Control
									Optim. Calc. Var., 28 (2022), pp.~Paper No. 54, 42.
									
									\bibitem{leal2021simple}
									\leavevmode\vrule height 2pt depth -1.6pt width 23pt, {\em Two simple criterion
										to obtain exact controllability and stabilization of a linear family of
										dispersive {PDE}'s on a periodic domain}, Evol. Equ. Control Theory, 11
									(2022), pp.~1745--1773.
									
									\bibitem{WangHu}
									{\sc N.~Wang and Y.~Hu}, {\em Blowup of solutions for compressible
										navier-stokes equations with revised maxwell's law}, Appl. Math. Lett., 103
									(2020), pp.~106221, 6 pp.
									
									\bibitem{RY}
									{\sc R.~M. Young}, {\em An introduction to nonharmonic {F}ourier series},
									vol.~93 of Pure and Applied Mathematics, Academic Press, Inc. [Harcourt Brace
									Jovanovich, Publishers], New York-London, 1980.
									
									\bibitem{zabczyk2008mathematical}
									{\sc J.~Zabczyk}, {\em Mathematical control theory, modern birkh{\"a}user
										classics}, 2008.
									
									\bibitem{ZZ}
									{\sc Z.~Zahreddine and E.~F. Elshehawey}, {\em On the stability of a system of
										differential equations with complex coefficients}, Indian J. Pure Appl.
									Math., 19 (1988), pp.~963--972.
									
								\end{thebibliography}

							\end{document}